\documentclass[a4paper]{amsart}

\usepackage[english]{babel}
\usepackage[utf8]{inputenc}
\usepackage[T1]{fontenc}

\usepackage{amsmath}
\usepackage{amssymb}
\usepackage[dvipsnames]{xcolor}

\usepackage{enumerate}
\usepackage{enumitem} 

\usepackage[sort&compress,numbers]{natbib}
\usepackage{geometry}
\usepackage{scalerel} 

\usepackage{hyperref}
\usepackage{mathtools}
\usepackage{mathrsfs}

\usepackage{amsthm}
\newtheorem{theorem}{Theorem}
\numberwithin{theorem}{section}
\newtheorem{lemma}[theorem]{Lemma}
\newtheorem{proposition}[theorem]{Proposition}
\newtheorem{corollary}[theorem]{Corollary}
\theoremstyle{definition}
\newtheorem{definition}[theorem]{Definition}
\newtheorem{example}[theorem]{Example}
\newtheorem{remark}[theorem]{Remark}
\newtheorem{notation}[theorem]{Notation}
\newtheorem{convention}[theorem]{Convention}

\numberwithin{equation}{section}

\let\oldmarginpar\marginpar
\renewcommand\marginpar[1]{\-\oldmarginpar[\raggedleft\footnotesize #1]{\raggedright\footnotesize\color{red} #1}}
\renewcommand{\epsilon}{\varepsilon}
\renewcommand{\subset}{\subseteq}
\renewcommand{\supset}{\supseteq}
\renewcommand{\P}{\mathbb{P}}

\newcommand{\fp}[1]{{\mathbb #1}}

\newcommand{\N}{\mathbb{N}}
\newcommand{\Q}{\mathbb{Q}}
\newcommand{\R}{\mathbb{R}}

\newcommand{\E}{\mathbb{E}}
\renewcommand{\P}{\mathbb{P}}

\newcommand{\A}{\mathcal{A}}

\newcommand{\prob}{\mathcal{P}}

\newcommand{\F}{\mathcal{F}}
\newcommand{\G}{\mathcal{G}}
\newcommand{\T}{\mathcal{T}}

\newcommand{\FP}{\mathcal{FP}}
\newcommand{\FR}{\mathcal{FR}}

\newcommand{\id}{\textup{id}}
\newcommand{\pr}{\textup{pr}}
\newcommand{\cont}{\textup{cont}}
\newcommand{\HK}{\textup{HK}}

\newcommand{\law}{\mathcal{L}}
\newcommand{\sfE}{\mathsf{E}}

\newcommand{\AF}{\textup{AF}}
\newcommand{\pp}{\textup{pp}}
\newcommand{\bpp}{{\textup{\textbf{pp}}}} 
\newcommand{\cadlag}{c\`adl\`ag}
\newcommand{\uFP}{\textup{FP}}
\newcommand{\uFPc}{\uFP_{\!\mathrm{cont}}}
\newcommand{\uFR}{\textup{FR}}
\newcommand{\rank}{\textup{rank}}
\newcommand{\graph}{\textup{graph}}

\newcommand{\inlaw}{\stackrel{d}{\to}}

\definecolor{construction}{rgb}{0.05,0.05,0.55}

\title[Representing general stochastic processes as martingale laws]{Representing general stochastic processes\\ as martingale laws}

\author{M.\ Beiglböck, G.\ Pammer, S.\ Schrott, and X.\ Zhang}

\thanks{M.B.\ gratefully acknowledges support through FWF projects Y00782 and P35197, St.S.\ gratefully acknowledges support through FWF projects P35197 and P34743. All authors thank Daniel Bartl for many helpful discussions and remarks.}

\subjclass[2020]{46N30, 60B10, 60G07, 60G44}

\begin{document}

\begin{abstract}
	Random variables $X^i, i=1,2$  are `probabilistically equivalent’ if they have the same law. Moreover, in any class of equivalent random variables it is easy to select canonical representatives. 
 
	The corresponding questions are  more involved 
	for processes $X^i$ on filtered stochastic bases $(\Omega^i, \F^i, \P^i, (\F^i_t)_{t\in [0,1]})$.
    Here equivalence in law does not capture relevant properties of processes such as the solutions to stochastic control or multistage decision problems. This  motivates Aldous to introduce the stronger notion of \emph{synonymity} based on prediction processes. 
    Stronger still, Hoover--Keisler formalize what it means that $X^i, i=1,2$ have \emph{the same probabilistic properties}.
    We establish that canonical representatives of the Hoover--Keisler equivalence classes are given precisely by the set of all Markov-martingale laws on a specific nested path space $\mathsf{M}_\infty$. 
    As a consequence we obtain that, modulo Hoover--Keisler equivalence, the class of stochastic processes forms a Polish space. 

    On this space, processes are topologically close iff they model similar probabilistic phenomena. In particular, this means that their laws as well as the information encoded in the respective filtrations are similar.  
    Importantly, compact sets of processes admit a Prohorov-type characterization.
    We also obtain that for every stochastic process defined on some abstract base space, there exists a process with identical probabilistic properties which is defined on a standard Borel space.
\end{abstract}
\maketitle

\keywords{\small \it keywords: equivalence of stochastic processes, synonymity, Prohorov's theorem, adapted distribution, fixed points}

\section{Introduction}

\subsection{Outline}

In this article we seek to identify a natural topological structure on the collection of all stochastic processes. 
To this end, we want to identify stochastic  processes  which model the same probabilistic phenomenon and we would like to consider processes as close  if they have similar probabilistic properties.
For random variables, this usually means that they are regarded as equivalent if their laws are equal and proximity is expressed in the weak topology. 

For stochastic processes, this approach is not fine enough, because we lose the information that the filtration carries, e.g.\ whether a process is a martingale or the values of a stochastic control problems for a certain process are not a property of the law of the process alone. 
We refer to the work of Aldous \cite[Section 11]{Al81} for a broader  discussion of this point. 

Building on the work of Hoover--Keisler \cite{HoKe84, Ho91} we define the \emph{adapted distribution} of a stochastic process (see Subsection \ref{ssec:adap_dist}) which captures the law of a stochastic process together with the information structure inherent in its filtration. The adapted distributions which arise in this way are the  martingale laws on a specific iterated path space $\mathsf{M}_\infty$ (see \eqref{eq:Minfty}) and provide  canonical representatives of stochastic processes. Stochastic processes then have similar properties if and only if their adapted distributions are close w.r.t.\ the weak topology on $\mathcal P (\mathsf{M}_\infty)$.
That is,  weak convergence of the adapted distributions induces an \emph{adapted weak topology} for stochastic processes. Notably, this topology satisfies 
an extended version of Prohorov's theorem: A family of stochastic processes is precompact if and only if the corresponding family of laws is tight. 
In particular, the set of all processes with \emph{one} fixed law is compact.

We note that in discrete time there is recently substantial interest in adapted weak topologies and adapted versions of Wasserstein distances from different communities, see Section 1.8 below. 
A main aim of this article is to provide the foundations for analogous developments also in continuous time. In particular, the setup we build here is used in the accompanying article \cite{BaBePaScZh23} to construct an adapted Wasserstein distance between stochastic processes. It is shown there that the space of all filtered processes is precisely the completion of the class of processes equipped with their natural filtration.  
Adapted Wasserstein distance yields stability of optimal stopping problems. 
In fact, on processes with natural filtration, the adapted weak topology is the weakest topology which makes optimal stopping continuous.

\subsection{The iterated prediction process}\label{sec:IntroInteratedPP}

As noted above, it is important for our purposes that we consider stochastic processes together with their  filtrations: 
\begin{definition}
A filtered process is a tuple 
\[
    \fp X = (\Omega,\F,\P, (\F_t)_{t \in [0,1]}, X)
\]
consisting of a probability space $(\Omega,\F,\P)$, a filtration $(\F_t)_{t \in [0,1]}$ that satisfies the usual conditions, and a process $X$ with paths in 
\[
D([0,1];\R^d)=: \mathsf{M}_0
\]
equipped with Skorohod's $J_1$-topology. The collection of all filtered processes is denoted by $\FP$. 
\end{definition}

The concept of prediction process goes back to Knight \cite{Kn75, Kn77} and is used by Aldous \cite{Al81} to encode information which  a filtration provides on a stochastic process. The prediction process $\pp^1(\fp X)$ of a filtered process $\fp X$ is the measure valued martingale
\[
\pp^1(\fp X )  {= (\pp_t^1(\fp X) )_{t \in [0,1]}}   =  (\law( X | \F_t  ) )_{t \in [0,1]} \in D([0,1];\prob(\mathsf{M}_0)) =:  \mathsf{M}_{1}.
\]
In words, $\pp_t^1(\fp X)$ describes the best available prediction about the evolution of the path of $X$ given the information at time $t\in [0,1]$. That $X$ is $(\F_t)_{t \in [0,1]}$-adapted is reflected by the fact that the evolution of $X$ up to time $t$ is fully determined by $\pp_t^1(\fp X)$.  Precisely, when writing $e_s(f) = f(s)$ for the evaluation of paths at time $s$ and $\#$ for the pushforward of measures, we have for  $s \le t$ 
\begin{align}\label{eq:ppp1}
    {e_s}_\# \pp^1_t(\fp X) = \law(X_s | \F_t) = \delta_{X_s} 
\end{align}
and in particular $\pp^1_1(\fp X) = \delta_X$.

Aldous called processes $\fp X, \fp Y$ \emph{synonymous} if $\law(\pp^1(\fp X)) = \law(\pp^1(\fp Y))$. This relation is capable of capturing  important probabilistic properties such as being Markov or being a martingale, i.e.\ if $\fp X$ and $\fp Y$ are synonymous $\fp X$ is a martingale if and only if $\fp Y$ is a martingale etc. Yet, this equivalence relation turns out to be not strong enough to preserve other properties of interest. For instance,  there are synonymous processes $\fp X, \fp Y$ that yield different values in optimal stopping problems (see \cite[Section 7]{BaBePa21}).  

In order to capture all probabilistic properties, Hoover--Keisler repeat  the Aldous--Knight construction and    introduce \emph{iterated prediction processes}
\[
\pp^{r+1}(\fp X )  =     (\law( \pp^r(\fp X) | \F_t  ) )_{t \in [0,1]} \in 
D([0,1];\prob(\mathsf{M}_r)) =:\mathsf{M}_{r+1}.
\]
Following the approach of Hoover--Keisler, in the iterative definition of the spaces $\mathsf{M}_r$ for $r>0$, the space of \cadlag-paths is equipped with the Meyer--Zheng topology and not with Skorohod's $J_1$-topology. This choice of the topology is motivated by the compactness results for martingale measure in this topology, see \cite{MeZh84}.

Prediction processes of higher rank contain increasingly more information. Indeed, the iterative definition implies that for every $r\ge 1$
\begin{align}\label{eq:ppp2}
\pp^{r+1}_1(\fp X) = \law(\pp^r(\fp X) |\F_1 ) = \delta_{\pp^r(\fp X) }.  
\end{align}
In general, the prediction process of rank $r+1$ contains strictly more information than the prediction process of rank $r$. Therefore, it is natural to consider the prediction process of rank $\infty$, which is the vector valued process that contains all finite-rank prediction processes, i.e.\ 
\begin{align}\label{eq:Minfty}
\pp^\infty(\fp X) := (\pp_t^1(\fp X), \pp^2_t(\fp X), \dots)_{t \in [0,1]}\in   D\bigg([0,1];\prod_{n=0}^\infty \prob(\mathsf{M}_n) \bigg) =: \mathsf{M}_\infty.
\end{align}
The process $\pp^\infty(\fp X)$ is a martingale in the sense that all of its components are measure-valued martingales. We will see in Subsection~\ref{sec:1.4} below the prediction process of rank $\infty$ contains all relevant information about the process $\fp X$ in a specific sense.


\subsection{Adapted distribution}\label{ssec:adap_dist}
We call $\law(\pp^\infty(\fp X))$ the \emph{adapted distribution} of the filtered process $\fp X$. 
All adapted distributions are martingale measures which terminate consistently in the sense of \eqref{eq:ppp1} and \eqref{eq:ppp2}. Conversely every such probability arises as the adapted distribution of a filtered process. 

To formalize this, we write $Z = (Z_t)_{t \in [0,1]}$  for the canonical process on $\mathsf{M}_\infty$. That is,  for every $t \in [0,1]$, $Z_t$ is a (countable) vector of measures 
\[
Z_t = (Z^r_t)_{r\geq 1},\quad   Z^r_t \in \prob(\mathsf{M}_{r-1}).  
\]
We call $\mu \in \prob(\mathsf{M}_\infty)$  martingale measure if under $\mu$ for every $r \geq 1$, the component process $Z^r := (Z_t^r)_{t \in [0,1]}$ is a $\prob(\mathsf{M}_{r-1})$-valued martingale  w.r.t.\ the filtration generated by the  vector-valued process $Z$. 

We call a martingale measure $\mu \in \prob(\mathsf{M}_\infty)$ consistently terminating if under $\mu$
\begin{itemize}
    \item ${e_t}_\# Z^1_t$ is a.s.\ a Dirac measure for every $t \in [0,1]$, 
    \item $Z^{r+1}_1 = \delta_{Z^r}$ a.s.\ for every $r \ge 1$.
\end{itemize}
Note that these two properties precisely reflect the properties of the prediction proess stated in \eqref{eq:ppp1}  and \eqref{eq:ppp2}.

\begin{theorem}\label{thm:introAdaptedDistrChar}
	A probability $\mu \in \prob(\mathsf{M}_\infty)$ is an adapted distribution if and only if it is a consistently terminating martingale measure.
\end{theorem}

The interesting implication in Theorem \ref{thm:introAdaptedDistrChar} is to show that every consistently terminating martingale measure $\mu$ is the adapted distribution of a filtered process. 
To build this process we attach a further coordinate $Z^0_t$ to the process $(Z_t^1, Z^2_t, \ldots)$. For this, fix a Borel map $\delta^{-1}: \prob( D([0,1];\R^d) ) \to D([0,1];\R^d) $ which satisfies  $\delta^{-1}(\delta_x) = x$ for $x \in D([0,1];\R^d)$ and is  arbitrary otherwise, cf. Remark~\ref{rem:Intensity_Delta_Inv}. We then set $$Z^0 := (Z^0_t)_{t\in [0,1]}:= \delta^{-1}(Z_1^1).$$ 
This leads to the following `canonical' construction of the desired filtered process:
\begin{theorem}\label{thm:CanonicalRep}
For a consistently terminating martingale measure $\mu \in \prob(\mathsf{M}_\infty)$ consider the filtered process 
\begin{align*}
	\fp X^\mu = (\mathsf{M}_\infty, \F, \mu, (\F_t)_{t \in [0,1]}, Z^0),
\end{align*}
 where $(\F_t)_{t \in [0,1]}$ is the right-continuous augmentation w.r.t.\ $\mu$ of the filtration generated by the process $Z$ on $\mathsf{M}_\infty$ and $\F := \F_1$.
Then the adapted distribution of $\fp X^\mu$ satisfies $$\law(\pp^\infty(\fp X^\mu)) = \mu.$$ 
\end{theorem}

\subsection{Stochastic processes as a Polish space}\label{sec:1.4}

Hoover--Keisler define the notion of \emph{adapted function} of a stochastic process and argue that two stochastic processes have the same probabilistic properties if and only if all adapted functions take the same value on these processes. Furthermore, this happens if and only if these processes have the same adapted distribution, see Section \ref{sec:AF} for details. We thus identify processes based on having the same adapted distribution as defined above: 

\begin{definition}
Filtered processes $\fp X, \fp Y$ are Hoover--Keisler-equivalent if and only if $\law(\pp^\infty(\fp X )) = \law(\pp^\infty(\fp Y ))$. The space of  filtered processes is the factor space 
	\[
	\uFP:= \FP/_{\approx_\infty}.
	\]

A sequence $(\fp X^n)_n$ in $\uFP$ converges to $\fp X \in \uFP$ in the adapted weak topology if the corresponding sequence of adapted distributions converges weakly on $\prob(\mathsf{M}_\infty)$.   
\end{definition}
By  Theorem \ref{thm:CanonicalRep}, the adapted distribution $\mu$ of a filtered process $\fp X$ provides a canonical representative $\fp X^\mu \approx_\infty \fp X$. Note that as $\mathsf M_\infty$ is standard Borel, this implies that up to equivalence, every process can be assumed to be supported by a standard Borel space. 


Besides Theorem \ref{thm:CanonicalRep}, a  main result of this article is:

\begin{theorem}\label{thm:FPPolish}
The space $\uFP$ equipped with  adapted weak convergence is Polish.
\end{theorem}

Our proof of Theorem \ref{thm:FPPolish} (implicitly) yields a complete metric on $\uFP$. However, building on the present work, we provide in the parallel paper \cite{BaBePaScZh23} an adapted Wasserstein distance which metrizes the adapted weak topology and which seems more convenient in view of applications.

\medskip
An important ingredient in the proof of Theorem \ref{thm:FPPolish} is the following extension of Prohorov's theorem which appears interesting in its own right. 

\begin{theorem}\label{thm:ProhorovIntro}
A collection of filtered processes is relatively compact in the adapted weak topology if and only if the respective collection of laws is relatively compact.
\end{theorem}
We emphasize that Theorem~\ref{thm:ProhorovIntro} was in essence already established by Hoover \cite[Theorem~4.3]{Ho91} based on repeated applications of the compactness result of Meyer--Zheng \cite{MeZh84} to the iterated prediction processes. The difference to Theorem~\ref{thm:ProhorovIntro} is that Hoover \cite{Ho91} does not consider stochastic processes in the classical sense, in particular, no adaptedness-constraint is present in his setting.

Theorem~\ref{thm:ProhorovIntro} guarantees the existence of numerous compact sets, and by identifying filtered processes with their adapted distributions, Theorem~\ref{thm:introAdaptedDistrChar} establishes a linear structure on $\uFP$.  Hence, it is thus natural to ask whether $\uFP$ allows for a Brouwer-type fixed point theorem. Indeed, such a theorem holds true.
\begin{theorem}\label{thmm:FixedPointIntro}
Let $f : \uFP \to \uFP$ be continuous with relatively compact range. Then $f$ has a fixed point. 
\end{theorem}

\begin{remark}
For simplicity the results in this section are stated for processes with paths in $D([0,1];\R^d)$ equipped with Skorohod's $J_1$-topology. We want to emphasize that most of our constructions are independent of the specific choice of the path space and the topology on it. 
\end{remark}

\subsection{Related literature}

This article is most closely related to works of Aldous \cite{Al81}, which uses the prediction process to define synonymity and closeness of stochastic processes, Hoover--Keisler \cite{HoKe84}, where the iterated prediction  processes is used to define a refined notion of equivalence of stochastic processes, and to Hoover \cite{Ho91} where this iterated prediction process is also used to define a mode of convergence of stochastic processes (or more precisely, random variables with filtration). 
 We note however, that the   idea to extend the weak topology or to define metrics that account for flow of information has arisen independently in different communities. Specifically, different groups of authors have introduced  `adapted' variants of the Wasserstein distance in discrete time, this includes the works of 
 Vershik \cite{Ve70, Ve94}, 
R\"{u}schendorf \cite{Ru85},  Gigli \cite[Chapter~4]{Gi04}, Pflug--Pichler \cite{PfPi12, PfPi14} and Nielsen--Sun \cite{NiSu20}. Another distance which accounts for the order of time is induced by the Knothe--Rosenblatt distance \cite{BePaPo21}. 
Further adapted extensions of the weak topology  were introduced by  Hellwig in economics \cite{He96} and using higher rank signatures  by Bonnier, Liu, and Oberhauser \cite{BoLiOb23}. Remarkably, when considering  processes in finite discrete time with their natural filtration, all of these approaches define the same \emph{adapted weak topology}, see \cite{BaBaBeEd19b, BePaPo21, Pa22, BoLiOb23}. Very much in the spirit of these results, we will establish in the parallel paper \cite{BaBePaScZh23} that also the adapted weak topology in continuous time defined in this paper admits a number of different characterizations, when restricting to processes with their natural filtration. In particular, we construct in \cite{BaBePaScZh23} a Wasserstein-type metric for the adapted weak topology. 

We note however, that in continuous time several different adapted Wasserstein distances have been considered (typically for continuous semimartingales), see \cite{BaBeHuKa20, BaBaBeEd19a, La18, BiTa19, Fo22a}. In these articles the focus is on specific applications and the respective Wasserstein distances are considered on smaller subclasses of processes (continuous (semi-) martingales or diffusions) and induce significantly  finer modes of convergence, see \cite[Appendix C]{BaBePa21} for more details.

Adapted topologies and adapted transport theory have been  utilized in domains such as  geometric inequalities  (e.g.\  \cite{BoKoMe05, BaBeLiZa17, La18, Fo22a}),    stochastic optimization and multistage programming (e.g.\ \cite{ Pf09, PfPi12, KiPfPi20, BaWi23}), 
mathematical finance (e.g.\ \cite{Do14, AcBaZa20, GlPfPi17, BaDoDo20, BaBaBeEd19a, BeJoMaPa21a, AcBePa20}), and maching learning (e.g.\ \cite{AkGaKi23, XuLiMuAc20, AcKrPa23, XuAc21}). We refer to \cite{PfPi16, BaBaBeWi20,  AcHo22} for the estimation of $\mathcal{AW}_p$ from statistical data, to \cite{PiWe21, EcPa22, BaHa23} for efficient  numerical methods for  adapted transport problems.

\subsection{Structure of the paper}
Section~\ref{sec:prelim} provides background on probability measures on Lusin spaces and \cadlag\ paths whose values are probability measures on Lusin spaces. These objects arise naturally in the construction of the nested spaces $\mathsf{M}_r$, $r \in \N \cup \{\infty \}$. The reader who is not interested in these technical details may skip the Sections~\ref{sec:ProbaLusin} and \ref{sec:MZLusin}.

Section~\ref{sec:MVM} concerns measure-valued martingales on Lusin spaces. We generalize the compactness result for real-valued martingales from Meyer--Zheng \cite{MeZh84} to measure-valued martingales and strengthen their results on convergence of finite dimensional distributions. The results which are needed further on in this article are summarized in Section~\ref{sec:MVMIntro}. 
For the first read the reader may skip the Sections~\ref{sec:MVMFDD} and~\ref{sec:MVMComp} containing the proofs of these results.

In Section~\ref{sec:FP} we introduce the concept of filtered random variables, a generalization of filtered processes, which was first considered by Hoover--Keisler \cite{HoKe84} and is technically more convenient for certain parts of this article. Moreover, we discuss the construction of the iterated prediction process in more detail. 

Section~\ref{sec:Space} is about the space of filtered processes (filtered random variables) modulo Hoover--Keisler-equivalence. We introduce the adapted weak topology on this space, prove that it is Polish and that it admits a Prokohorov-type compactness criterion.

In Section~\ref{sec:Discr} we introduce an operator that discretizes filtered random variables in time. This operator is useful to connect the continuous-time framework to the discrete-time framework that was studied in \cite{BaBePa21}.

Section~\ref{sec:AF} concerns the concept of adapted functions. This concept makes precise that Hoover--Keisler-equivalence can be interpreted as `having the same probabilistic properties'.

In Section~\ref{sec:adapted} we translate results on filtered random variables to the framework of filtered processes, in particular, we prove the theorems stated in the introduction.


\section{Preliminaries}\label{sec:prelim}

\subsection{Notations and Conventions}
Throughout this paper $S, S_1, S_2, \dots$ are Lusin spaces.\footnote{A Lusin space is a topological space that is homeomorphic to a Borel subset of a Polish space. We use Lusin spaces because the space of \cadlag{} functions with convergence in measure is not a Polish space, but merely Lusin. Further details are explained in Section~\ref{sec:ProbaLusin}.} If we need stronger assumptions on the spaces (e.g.\ Polish, compact metrizable) we will always state it explicitly. We denote compatible metrics by $d,d_1,d_2, \dots$ and always assume that they are complete if the respective space is Polish. 

Following the ideas of Meyer--Zheng \cite{MeZh84} we equip $[0,1]$ with the measure $\lambda = \frac{1}{2}( \textup{Leb} + \delta_1 )$, where Leb is the Lebesgue measure on $[0,1]$. We call the topology of convergence in measure w.r.t.\ $\lambda$ the Meyer--Zheng topology.
$C_b(S)$ denotes the set of continuous bounded functions $S \to \R$.

\begin{convention}\label{con:spaces}
Let $S$ be a Lusin space.
\begin{itemize}
	\item Subspaces are always endowed with the subspace topology.
	\item Product spaces are always equipped with the product topology.
	\item The space of probability measures on $S$ is denoted by $\prob(S)$ and is always equipped with the weak topology (testing against continuous bounded functions).
	\item The space of Borel functions $[0,1] \to S$ modulo $\lambda$-a.s.\ equality is denoted by $L_0(S)$ and equipped with the Meyer--Zheng topology.
	\item The space of \cadlag{} functions $[0,1] \to S$ is denoted by $D(S)$. By minor abuse of notation, we consider it as subspace of $L_0(S)$ and call $f \in L_0(S)$ \cadlag\ if it admits a representative in $D(S)$.
    In particular, $D(S)$ is equipped with the subspace topology inherited from the Meyer--Zheng topology.
\end{itemize}
\end{convention}
In Sections \ref{sec:ProbaLusin} and \ref{sec:MZLusin} below we will further discuss the topological properties of the spaces $\prob(S), L_0(S)$ and $D(S)$. 

Let $f : S \to \R$ be bounded and Borel. Then we define the mapping
\begin{align}
f^\ast : \prob(S) \to \R: \mu \mapsto \int f d\mu.
\end{align}
Note that by definition the weak topology is the initial topology w.r.t.\ the family $\{ f^\ast : f \in C_b(S) \}$. In particular, $f^\ast$ is continuous if $f$ is continuous. 

If $f: S_1 \to S_2$ is a Borel measurable mapping and $\mu \in \prob(S_1)$,  we denote the pushforward of $\mu$ under $f$ by $f_\#\mu$. We introduce a notation for the pushforward map
\[
\prob(f) : \prob(S_1) \to \prob(S_2) : \mu \mapsto f_\#\mu.
\]
On the first glance, this might seem like a notational excess, but it will help to keep track when dealing with nested spaces (e.g.\ probability measures on probability measures) later in this paper. 

It is convenient to note that $\prob(f \circ g) = \prob(f) \circ \prob(g)$ and $\prob(\id_S)=\id_{\prob(S)}$ (i.e.\  the operation $\prob$ is a covariant functor in the category of Lusin spaces).

\begin{remark}\label{rem:P}
	The operation $\prob$ preserves many properties of spaces and functions:
	\begin{itemize}
		\item $S \mapsto \prob(S)$ preserves the properties Lusin, Polish and compact metrizable.
		\item $f \mapsto \prob(f)$ preserves the properties Borel measurable, continuous, injective, surjective, bijective, topological embedding and homeomeomorphism.
	\end{itemize}
\end{remark}

We denote the mapping which maps $s$ to the Dirac measure at $s$ by $\delta$, i.e.
$$
\delta : S \to \prob(S) : s \mapsto \delta_s.
$$
Note that $\delta$ is a topological embedding and its range 
$$
\delta(S) := \{ \delta_s : s \in S\}
$$
is closed in $\prob(S)$. We will often use its inverse
$$
\delta^{-1} : \delta(S) \to S : \delta_s \mapsto s.
$$

\begin{remark}\label{rem:et}
For $t \in [0,1]$ we introduce the map
$$
e_t : D(S) \to S : f \mapsto f(t).
$$ 
First we need to argue that $e_t$ is well-defined because we consider $D(S)$ as a subset of $L_0(S)$, which formally consists of equivalence classes of functions w.r.t.\ $\lambda$-a.s.\ equality. For $t<1$ this is guaranteed by the right continuity of the functions. This argument fails for the endpoint $t=1$, but this causes no problem because we adopted the convention that $\lambda$ puts positive mass at the endpoint $t=1$. It is important to keep in mind that $e_t$ is not continuous for $t<1$, but it is continuous for $t=1$.

If $T$ is a finite set, i.e.\ $T=\{t_1,\dots, t_N\}$ with $t_1 < \dots < t_N$, we write 
$$
e_T : D(S) \to S^N : f \mapsto ( f(t_1),\dots,f(t_N) ).
$$

\end{remark}

The $\sigma$-algebra generated by the random variable $X$ augmented with $\P$-null sets is denoted by $\sigma^{\P}(X)$.
If $X^n \to X$ in law, we either write $\law(X^n) \to \law(X)$ or $X^n \inlaw X$.

Throughout this paper $\N$ denotes the set of natural number without 0 and $\N_0$ the set of natural numbers including 0.

\subsection{Probability measures on Polish and Lusin spaces}\label{sec:ProbaLusin}
Recall that a Polish space is a separable completely metrizable space. Polish spaces provide a convenient framework for measure theory, however, they are not sufficiently general for the purposes of this paper: The space of \cadlag{} functions equipped with the Meyer--Zheng-topology is not Polish, but only a Borel subset of a Polish space, see \cite[Appendix]{MeZh84}. Clearly, this also applies to the spaces $\mathsf{M}_r$, $ r \in \N \cup \{\infty \}$, that have been introduced in \eqref{eq:Minfty} and are crucial throughout this article. Therefore, we  have to deal with the class of topological spaces  that are homeomorphic to a Borel subset of a Polish space:
\begin{definition}\label{def:Lusin}
	A topological spaces $S$ is called Lusin space if there is a Polish space $S'$ and a topological embedding $\iota : S \to S'$ such that $\iota(S)$ is a Borel subset of $S'$.
\end{definition}

As Polish spaces are separable metrizable and this property carries over to subspaces, Lusin spaces are  separable metrizable as well. The proof of the following result is given in Appendix~\ref{app:proofs}:
\begin{proposition}\label{prop:lusin_facts}
	\begin{enumerate}[label = (\alph*)]
		\item A metrizable space $(S,\T)$ is Lusin if and only if there is a stronger topology $\T' \supset \T$ such that  $(S,\T')$ is a Polish space.
		\item Let $(S,\T)$ be a Lusin space and $A \subset S$. Then $A$ is Borel if and only if there is a Polish topology $\T'$ on $A$ that is stronger than the subspace topology $\T|_A$. 
	\end{enumerate}
\end{proposition}

\begin{remark}
	There are two non-equivalent definitions in the literature. The definition that we use follows Bourbaki  \cite[IX.4~Definition~6]{Bo66} and Dellacherie--Meyer \cite[Chapter~III, Definition~16]{DeMeA}. Many authors (e.g.\ \cite{Bo07}) use the following definition: A Hausdorff space $S$ is called Lusin if it is the continuous bijective image of a Polish space (or equivalently, if there is a stronger Polish topology on $S$ itself). 
	
	Proposition~\ref{prop:lusin_facts} implies that this is weaker than the Definition~\ref{def:Lusin} that we use. Indeed, it is strictly weaker: A separable infinite dimensional Banach space $X$ equipped with 
		the weak-$\ast$-topology is Hausdorff and the norm topology is a stronger Polish topology on $X$, however, $X$ with the  weak-$\ast$-topology is not Lusin according to Definition~\ref{def:Lusin} because the  weak-$\ast$-topology is not metrizable.
	
		We work with Definition~\ref{def:Lusin} because it is sufficiently general for our purposes (the space of \cadlag{} functions is Lusin according to that definition, see Proposition~\ref{prop:pseudo-paths} below) and we want to restrict to separable metric spaces in this paper because we want to work with a countable family of convergence determining functions, see Definition~\ref{def:conv_deter_fam} below.   
\end{remark}

It is well known that separable metric spaces can be embedded into the Hilbert cube $[0,1]^\N$, see e.g. \cite[Theorem 4.14]{Ke95}. Polish spaces are (up to homeomorphism) the $G_\delta$-subsets of $[0,1]^\N$ and Lusin spaces the Borel subsets of $[0,1]^\N$.

Next, we introduce the notion of convergence determining families:
\begin{definition}\label{def:conv_deter_fam}
	A family $(\phi_j)_{j \in J}$ of $[0,1]$-valued functions on $S$ is called
	\begin{enumerate}[label = (\alph*)]
		\item point separating, if for $x,y \in S$ we have 
		$$
		x=y \iff \forall j \in J : \phi_j(x) = \phi_j(y);
		$$ 
		\item convergence determining, if for every sequence $(x_n)_{n}$ in $S$ and $x \in S$ we have 
		\[
		x_n \to x \iff \forall j \in J \colon \phi_j(x_n) \to \phi_j(x).
		\]
	\end{enumerate}		
\end{definition}

Note that $(f_j)_{j \in J}$ is convergence determining if and only if the product map
$$
\prod_{j \in J} f_j : S \to [0,1]^J : s \mapsto (f_j(s))_{j \in J}
$$
 is a topological embedding. Conversely, if $S$ can be embedded into a cube $[0,1]^J$, the projections are a convergence determining family. 
So, working with convergence determining functions is equivalent to using an embedding into the cube $[0,1]^J$, but notionally more convenient for our purposes. In particular, we have
\begin{lemma}\label{lem:convergence.determining}
	Let $S$ be a separable metric space. Then there is a countable convergence determining family of functions on $S$ that is closed under multiplication. 
\end{lemma}

Next, we discuss how point separating  [convergence determining] families on $S$ relate to point separating  [convergence determining] families on $\prob(S)$. The following useful results can be found in \cite{BlKo10}.
\begin{lemma}\label{lem:P(S)_ptsep_convdet}
	Let $S$ be a Lusin space and $\{ \phi_j : j \in J\}$ be a family of $[0,1]$-valued functions on $S$ that are closed under multiplication. Then we have: 
	\begin{enumerate}[label = (\alph*)]
		\item  If $\{ \phi_j : j \in J\}$ is point separating and consists of continuous functions, then $\{ \phi^\ast_j : j \in J\}$ is point separating on $\prob(S)$.
		\item  If $\{ \phi_j : j \in J\}$ is point separating, consists of Borel functions and $J$ is countable, then $\{ \phi^\ast_j : j \in J\}$ is point separating on $\prob(S)$.
		\item  If $\{ \phi_j : j \in J\}$ is convergence determining, then $\{ \phi^\ast_j : j \in J\}$ is convergence determining on $\prob(S)$.
	\end{enumerate}
\end{lemma}

Next, we discuss compactness in $\prob(S)$. Recall that $\A \subset \prob(S)$ is called tight, if for every 
$\epsilon >0$ there is a compact set $K_\epsilon \subset S$ such that $\mu(K_\epsilon) \ge 1-\epsilon$ for all $\mu \in \A$. If $S$ is a Polish space, Prohorov's theorem characterizes the compact subsets of $\prob(S)$ via tightness:
\begin{theorem}[Prohorov in Polish spaces]
	Let $S$ be a Polish space and $\A \subset \prob(S)$. Then $\A$ is relatively compact if and only if it is tight.
\end{theorem}
In the case that $S$ is merely Lusin, the situation gets a bit more complicated because compact subsets of $\prob(S)$ are in general not tight (see \cite[Theorem~8.10.17]{Bo07a} for an example in the case $S = \Q$). We still have the following results (see \cite[Theorem~8.6.4, Theorem~8.6.7]{Bo07a})
\begin{theorem}[Prohorov in Lusin spaces]\label{thm:prohorov_lusin}
	Let $S$ be a Lusin space. Then we have:
	\begin{enumerate}[label = (\alph*)]
		\item Tight subsets of $\prob(S)$ are relatively compact.
		\item Convergent sequences in $\prob(S)$ are tight.
	\end{enumerate}
\end{theorem}
Note that assertion (b) does not imply that all countable relatively compact sets in $\prob(S)$ are tight.

The intensity operator will be crucial throughout this paper:

\begin{definition}[intensity operator]\label{def:intensity}
	The intensity operator $I \colon \prob(\prob(S)) \to \prob(S)$ is defined as the unique mapping that satisfies, for all Borel measurable $f \colon S \to [0,1]$,
	\begin{align}\label{eq:DefI}
	\int f(s) \, I(P)(ds) = \iint f(s) \, q(ds) \, P(dq). 
	\end{align}         
\end{definition}
An equivalent way to write \eqref{eq:DefI} is
\begin{align}\label{eq:DefI2}
	f^\ast(I(P)) = \int f^\ast(q) P(dq).
\end{align}
We see that $I$ is continuous by applying  \eqref{eq:DefI2} to all $f \in C_b(S)$.

\begin{remark}\label{rem:Intensity_Delta_Inv}
Consider the function $\delta : \prob(S) \to \delta(\prob(S)) : \mu \mapsto \delta_\mu$. Then we have $I(\delta_\mu) = \mu$, so the intensity operator $I : \prob(\prob(S)) \to \prob(S)$ is a continuous extension of $\delta^{-1}$ from $\delta(\prob(S))$ to the whole space $\prob(\prob(S))$. Of course, $I$ is not the inverse of $\delta$, it is merely a left inverse. 

If $S$ is (a convex subset of) a separable Banach space, then 
$$
\prob(S) \to S : \mu \mapsto \int x \mu(dx)
$$
is a continuous left inverse of $\delta$. 

In general there is not always a continuous left inverse of the function $\delta : S \to \prob(S)$. Indeed, a left inverse of $\delta : S \to \prob (S)$ is a continuous surjective mapping $\prob(S) \to S$. As $\prob(S)$ is connected, this can only exist, if $S$ is connected as well. In particular, if $S$ is discrete, $\delta : S \to \prob (S)$ has no continuous left inverse.
\end{remark}

\begin{proposition}
	\label{prop:intensity_tightness}
	For $\mathcal A \subset\prob(\prob(S))$ we have the following implications:
	$$
	\begin{array}{ccc}
		\text{(i) }  \mathcal A \text{ is tight } & \Leftarrow  & \text{(ii) }  I(\mathcal A) \text{ is tight } \\
		\Downarrow &  & \Downarrow \\
		\text{(iii) }  \mathcal A \text{ is relatively compact } & \Leftrightarrow  & \text{(iv) }  I(\mathcal A) \text{ is relatively compact }   
	\end{array}
	$$
	If $S$ is Polish all of these statements are equivalent. 
\end{proposition}

\begin{proof} We start with the case that $S$ is Lusin.
	
	$(i) \implies (iii)$ and $(ii) \implies (iv)$ are immediate from Theorem~\ref{thm:prohorov_lusin}.
	
	$(ii) \implies (i)$ can be found in \cite[p. 178, Ch. II]{Sz91}. The claim there is for Polish spaces, but this implication did not use Polish.
	
	$(iii) \implies (iv)$ is immediate as $I$ is continuous and continuous images of relatively compact sets are relatively compact.
	
	$(iv) \implies (iii)$ Let $(\mu^n)_n$ be a sequence in $\mathcal A$. We need to show that it has a convergent subsequence. As $(I(\mu^n))_n$ is a sequence in the relatively compact set $I(\mathcal A)$, there is a subsequence $(I(\mu^{n_k}))_k$ converging to some $\nu \in \prob(S)$. By Theorem~\ref{thm:prohorov_lusin}(b), the sequence $(I(\mu^{n_k}))_k$ is tight. By the already proven implication $(ii) \implies (i)$, the sequence $(\mu^{n_k})_k$ is tight as well and therefore relatively compact by Theorem~\ref{thm:prohorov_lusin}(a). So,  $(\mu^{n_k})_k$, and therefore $(\mu^n)_n$, has a convergent subsequence. 
	
	In the case that $S$ is Polish, we get $(iii) \implies (i)$ and $(iv) \implies (ii)$ from Prohorov's theorem. Then we have proven enough implications to conclude that all statements are equivalent.
\end{proof}

\begin{lemma}\label{lem:LusinProdComp}
Let $S_j, j \in \N$ be Lusin spaces and $S:= \prod_{j \in \N} S_j$. Then we have
\begin{enumerate}[label = (\alph*)]
	\item A sequence $(\mu^n)_n$ in $\prob(S)$ converges to $\mu$ in $\prob(S)$ if and only if $(\pr_1,\dots,\pr_j)_\#\mu^n \to (\pr_1,\dots,\pr_j)_\#\mu$ for all $j \in \N$. 
	\item $\A \subset \prob(S)$ is relatively compact if and only if $\prob(\pr_j)(\A)$ is relatively compact in $\prob(S_j)$ for all $j \in \N$.
\end{enumerate}
\end{lemma}
For the sake of completeness, a proof is given in the appendix. While (a) is trivial, (b) is not because it relies on Prohorov's theorem and we need to be careful there in non-Polish spaces.

\begin{remark}\label{rem:PGamma}
Let $\prob_\Gamma(S_1 \times S_2)$ denote the set of $\mu \in \prob(S_1 \times S_2)$ that are concentrated on the graph of a Borel function from $S_1$ to $S_2$. It is known (see e.g.\ \cite{Ed19}) that $\prob_\Gamma(S_1 \times S_2)$ is a $G_\delta$-subset of $\prob(S_1 \times S_2)$, if $S_1, S_2$ are Polish. As Lusin spaces are Borel subsets of Polish space and Borel functions defined on a Borel subset can always be extended to the entire space, this assertion readily extends to Lusin spaces $S_1,S_2$.
\end{remark}

%
%
%
%

\subsection{Pseudopaths with values in Lusin spaces}\label{sec:MZLusin}
	The construction of Hoover--Keisler is based on processes with \cadlag{} paths in nested probability spaces whose path spaces are equipped with the pseudo-path topology of Meyer--Zheng \cite{MeZh84}.
For this reason, it is necessary to first establish some basic properties of this topology.

\begin{definition}
	An $S$-valued pseudo-path is a probability measure on $[0,1] \times S$ that is concentrated on a graph of function and has first marginal $\lambda$. We denote the set of all $S$-valued pseudo-paths as $\Psi(S)$.
\end{definition}

A probability measure $\mu \in \prob([0,1]\times S)$ with first marginal $\lambda$ is a pseudo-path, if and only if there is a Borel function $f: [0,1] \to S$ such that 
$$
\mu_t = \delta_{f(t)} \qquad \lambda\text{-a.s.},
$$
where $(\mu_t)_{t \in [0,1]}$ is a  disintegration of $\mu$. The mapping 
\begin{equation}
	\label{eq:def_iota_S}
	\iota_S \colon L_0(S) \to \prob([0,1] \times S) \colon f \mapsto  (\id,f)_\#\lambda
\end{equation}
is injective by the a.s.\ uniqueness of the disintegration and its range is $\Psi(S)$. 

\begin{definition}
	A pseudo-path $\mu$  is called \cadlag{} pseudo-path if there is a \cadlag{} function $f$ such that $\mu=\iota_S(f)$. We denote the set of $S$-valued \cadlag{} pseudo-paths as $\Psi_D(S)$, i.e.\ $\Psi_D(S) = \iota_S(D(S))$.
\end{definition}

The main goal of this section is to prove the following proposition:
\begin{proposition}\label{prop:pseudo-paths}
	$ $
	\begin{enumerate}[label = (\alph*)]
		\item \label{it:pseudo-paths.embedding}
		The mapping $\iota_S$ defined in \eqref{eq:def_iota_S} is a topological embedding with range $\Psi(S)$. In particular, $\Psi(S)$ is homeomorphic to $L_0(S)$.
		\item \label{it:pseudo-paths.Polish}
		$\Psi(S)$ is a $G_\delta$-subset of $\prob( [0,1] \times S )$. In particular, if $S$ is Polish, then $\Psi(S)$ is Polish; if $S$ is Lusin, then $\Psi(S)$ is Lusin.
		\item \label{it:pseudo-paths.cadlag_Borel}
		$D(S)$ is a Borel subset of $L_0(S)$. In particular, if $S$ is Lusin, then $D(S)$ is Lusin.
	\end{enumerate}
\end{proposition}
\begin{remark}\label{rem:pseudoL0}
	Proposition~\ref{prop:pseudo-paths} implies that we will no longer have to distinguish between $L_0$-functions and pseudo-paths. Will will choose the point of view depending on what is technically more convenient in the respective context. For the rest of this section, which is dedicated to the proof of this equivalence, we will of course carefully distinguish between $L_0$-functions and pseudo-paths.  
\end{remark}

The results of Proposition~\ref{prop:pseudo-paths} are well known from \cite{MeZh84} in the case  $S=\R$. However, the proof of \cite[Lemma 1]{MeZh84} made use of the fact that the paths are real-valued (they used that weak $L_2$-convergence plus convergence of the norms implies strong $L_2$-convergence).

We will generalize these results using convergence determining real-valued functions (see Definition \ref{def:conv_deter_fam} below). This generalization is straightforward, but we do not skip it because it gives us the opportunity to introduce notations that we need later and to derive lemmas about convergence determining families that will be useful several times throughout this paper.

Given two Lusin spaces $S_1, S_2$ and a Borel function $f : S_1 \to S_2$ one can consider the operation ``compose with $f$'' defined by
$$
L_0(S_1) \to L_0(S_2) : g \mapsto f \circ g.
$$
If $f$ is continuous, this operation maps \cadlag{} paths to \cadlag{} paths.

Next, we define this operation in the framework of pseudo-paths. For technical reasons, we define this operation not only on pseudo-paths, but we define an operation on the whole of $\prob([0,1] \times S_1)$ that acts on pseudo-paths as the composition.	
\begin{proposition}
	\label{prop:Psi_lift}
	Given $f \colon S_1 \to S_2$ be Borel, we define
		\begin{equation}
		\label{eq:def_Psi_lift}
		\Psi(f) := \prob( \,(t,s) \mapsto (t,f(s)) \, )
		 : \prob([0,1] \times S_1) \to \prob([0,1] \times S_2).
	\end{equation}
	Then $\Psi(f)$ has the following properties:
	\begin{enumerate}[label = (\alph*)]
		\item \label{it:Psi_lift.kernel_rep}
		On a probability measure  $\mu \in \prob([0,1] \times S_1)$ with disintegration $(\mu_t)_{t \in [0,1]}$ the mapping $\Psi(f)$ acts as
		$$\Psi(f)(\mu)(dt,ds) = f_\#\mu_t (ds)  \lambda(dt).$$
		\item \label{it:Psi_lift.iota} $\Psi(f)$ maps $\Psi(S_1)$ to $\Psi(S_2)$ and acts there as composition with $f$, i.e.
		\begin{align}\label{eq:PsiComp}
					\Psi(f)(\iota_{S_1}(g)) = \iota_{S_2}(f \circ g).
		\end{align}
		\item \label{it:Psi_lift.cadlag}
		If $f$ is continuous, $\Psi(f)$ is continuous from $\prob([0,1] \times S_1)$ to $\prob([0,1] \times S_2)$ and it maps $\Psi_D(S_1)$ to $\Psi_D(S_2)$.
		\item \label{it:Psi_lift.embedding}
		If $f$ is a topological embedding, $\Psi(f)$ is a topological embedding of $\Psi(S_1)$ into $\Psi(S_2)$ and of $\Psi_D(S_1)$ into $\Psi_D(S_2)$.			
	\end{enumerate}
\end{proposition}

\begin{proof}
	Straightforward.
\end{proof}

\begin{lemma}\label{lem:char_pseudo-path}
	Let $\mu \in \prob([0,1] \times S)$ with first marginal $\lambda$, let $\{ \phi_k \colon k \in \N \}$ be a point separating family of Borel functions on $S$, and let $\{ \psi_k \colon k \in \N \}$ be a convergence determining family. Then we have:
	\begin{enumerate}[label = (\alph*)]
		\item \label{it:char_pseudo-path.separating}
		$\mu$ is a pseudo-path if and only if  $\Psi(\phi_k)(\mu)$ is a pseudo-path for every $k \in \N$. 
		\item \label{it:char_pseudo-path.convergence_det}
		$\mu$ is a \cadlag{} pseudo-path if and only if  $\Psi(\psi_k)(\mu)$ is  a \cadlag{} pseudo-path for every $k \in \N$. 
	\end{enumerate}
\end{lemma}

\begin{proof}
	\ref{it:char_pseudo-path.separating}
	The forward implication is due to Proposition \ref{prop:Psi_lift} \ref{it:Psi_lift.iota}.
	For the backward implication, note that for evrey $k \in \N$ there is a $\lambda$-full set $A_k$ such that ${\phi_k}_\#\mu_t = \delta_{x_{k,t}}$ for some $x_{k,t} \in [0,1]$. Then the set $A:= \bigcap_{k \in \N } A_k$ has $\lambda$-full measure and we have ${\phi_k}_\#\mu_t(\{x_{k,t}\})=1$	for all $t \in A$. The set $$S_t:= \{ x \in S: \phi_k(x)=x_{k,t}  \text{ for all } k \in \N\}$$
	contains at most one point as $(\phi_k)_{k\in \N}$ is point separating. Moreover, we have $\mu_t(S_t)=1$ for all $t \in A$, so $\mu_t$ is a Dirac for all $t \in A$.

	\ref{it:char_pseudo-path.convergence_det}
	Again, the direct implication is an easy consequence of Proposition \ref{prop:Psi_lift} \ref{it:Psi_lift.cadlag}.      	For the reverse implication note that $\mu$ is a pseudo-path by point (a) of this lemma, so there is some $f \in L_0(S)$ such that $\mu=\iota_S(f)$. The assumption and Proposition~\ref{prop:Psi_lift}\ref{it:Psi_lift.iota} imply that $\psi_k \circ f$ is \cadlag{} for all $k \in \N$. As $(\psi_k)_k$ is convergence determining we can conclude that $f$ is \cadlag.
\end{proof}

Before we prove the main result of this section, we recall some standard results about convergence in measure:	
\begin{lemma}\label{lem:convinprob_simpleppty}
	Let $f_n,f \in L_0(S_1)$ and $\phi \colon S_1 \to S_2$ be continuous.
	\begin{enumerate}[label = (\alph*)]
		\item \label{it:convinprob_simpleppty.subsequences}
		$(f_n)_{n \in \N}$ converges to $f$ in probability if and only if every subsequence of $(f_n)_{n \in \N}$ admits an a.s.-convergent subsequence with limit $f$.
		\item \label{it:convinprob_simpleppty.composition}
		If $(f_n)_{n \in \N}$ converges to $f$ in probability then $(\phi \circ f_n)_{n \in \N}$ converges to $\phi \circ f$ in probability.
	\end{enumerate}
\end{lemma}
\begin{proof}
	See e.g.\ \cite[Lemma~3.2 and Lemma~3.3]{Ka97}.
\end{proof}

\begin{lemma}\label{lem:convinprob_initialtop}
	Let $\{ \phi_k \colon k \in \N \}$ be a convergence determining family on $S$ and let $f_n,f \in L_0(S)$.
	Then $(f_n)_{n \in \N}$ converges to $f$ in probability if and only if, for every $k \in \N$, $(\phi_k \circ f_n)_{n \in \N}$ converges to $\phi_k \circ f$ in probability.
\end{lemma}
\begin{proof}
	The forward implication follows from Lemma \ref{lem:convinprob_simpleppty} \ref{it:convinprob_simpleppty.composition}. The reverse implication follows from a standard diagonalization argument and Lemma \ref{lem:convinprob_simpleppty} \ref{it:convinprob_simpleppty.subsequences}.
\end{proof}
Finally we are in the position to prove the main result of this section.
\begin{proof}[Proof of Proposition \ref{prop:pseudo-paths}]
	\ref{it:pseudo-paths.embedding}:
	Since $\iota_S$ is a bijection onto $\Psi(S)$, it remains to prove continuity of $\iota_S$ and $\iota_S^{-1}$ restricted to $\Psi(S)$.
	
	In order to prove continuity of $\iota_S$, let $(f_n)_{n \in \N}$ be a convergent sequence in $L_0(S)$ with limit $f \in L_0(S)$ and $g \in C_b([0,1] \times S)$.
	By Lemma \ref{lem:convinprob_simpleppty} \ref{it:convinprob_simpleppty.composition} we have that $(g \circ (\id,f_n))_{n \in \N}$ converges to $g \circ (\id,f)$ in probability.
	  Using dominated convergence we obtain
	\[
	\limsup_{n \to \infty}\left| \int g \, d\iota_S(f_n) - \int g \, d\iota_S(f) \right| \le
	\lim_{n \to \infty} \int |g(t,f_n(t)) - g(t,f(t))| \, \lambda(dt) = 0.
	\]
	As $g \in C_b([0,1] \times S)$ is arbitrary we obtain that $(\iota_S(f_n))_{n \in \N}$ converges to $\iota_S(f)$ in $\prob([0,1]\times S)$.
	
	In order to prove the continuity of $\iota^{-1}_S|_{\Psi(S)}$ assume that $(\iota_S(f_n))_{n \in \N}$ converges weakly to $\iota_S(f)$.
	Pick a convergence determining family $\{ \phi_k \colon k \in \N \}$.
		The map $\Psi(\phi_k) : \prob([0,1] \times S) \to  \prob([0,1] \times \R)$ is continuous (cf.\  Proposition~\ref{prop:Psi_lift}\ref{it:Psi_lift.cadlag}), so using \eqref{eq:PsiComp} we get 
	\begin{equation}
		\label{eq:prop.pseudo-paths.limits}
		\iota_{\R}(\phi_k \circ f_n)
		= \Psi(\phi_k)(\iota_S(f_n))
		\to \Psi(\phi_k)(\iota_S(f)) = \iota_\R(\phi_k \circ f).
	\end{equation}
	By the corresponding real-valued result  \cite[Lemma 1]{MeZh84}, we have $\phi_k \circ j_n \to \phi_k \circ f$ in measure for all $k \in \N$. Using 
	Lemma \ref{lem:convinprob_initialtop} we conclude that $f_n \to f$ in measure.
	
	\ref{it:pseudo-paths.Polish}: $\prob_\Gamma([0,1] \times S)$ is $G_\delta$, see Remark~\ref{rem:PGamma}, and the set of $\mu \in \prob([0,1] \times S)$ with first marginal $\lambda$ is closed and therefore $G_\delta$. So, $\Psi(S)$ is $G_\delta$ as intersection of two $G_\delta$-sets.
	
	\ref{it:pseudo-paths.cadlag_Borel}:
	Let $(\phi_k)_{k \in \N}$ be a convergence determining family on $S$.  
	Lemma \ref{lem:char_pseudo-path} implies 
	\begin{equation}\label{eq:prf_pseuopath_c}
		\Psi_D(S) = \bigcap_{k \in \N} \left\{ \mu \in \prob([0,1] \times S) \colon \Psi(\phi_k)(\mu) \in \Psi_D(\R) \right\}.
	\end{equation}
	By the corresponding real-valued result \cite[Theorem 2]{MeZh84}, $\Psi_D(\R)$ is a Borel subset of $\prob([0,1] \times \R)$. As $\Psi(\phi_k)$ is continuous for every $k \in \N$, \eqref{eq:prf_pseuopath_c} implies that $\Psi_D(S)$  is Borel as countable intersection of Borel sets. We conclude by \ref{it:pseudo-paths.embedding} and \ref{it:pseudo-paths.Polish} that $D(S)$ is a Borel subset of $L_0(S)$.
\end{proof}

\begin{lemma}\label{lem:product_map}
Let $J$ be a finite or countable set and let $S_j, j\in J$ be Lusin spaces. Then the mapping
\begin{align}\label{eq:prod_map}
\pi :\prod_{j \in J} L_0(S_j) \to L_0\Big(\prod_{j \in J} S_j\Big) : (f_j)_{j \in J} \mapsto ( t \mapsto (f_j(t))_{j \in J} )
\end{align}
is a homeomorphism that maps $\prod_{j \in J} D(S_j)$ onto  $D(\prod_{j \in J} S_j)$.
\end{lemma}
\begin{proof}
This is a straightforward application of Lemma~\ref{lem:convinprob_initialtop} and Proposition~\ref{prop:Psi_lift}.
\end{proof}


\section{Measure valued martingales}\label{sec:MVM}
Measure-valued martingales are central tools in the theories of Knight, Aldous and Hoover--Keisler to encode the information w.r.t.\ a stochastic process carried by the filtration. In this section we develop the theory of measure-valued martingales to the extend that we need for this article. The main results are Theorem~\ref{thm:M_compactness}, which implies that compactness properties are preserved through process of iterating prediction processes, and Theorem~\ref{thm:weakFDD} on convergence of finite dimensional distributions of measure-valued martingales.

We summarize all results on measure-valued martingales that we will need later in the article in Section~\ref{sec:MVMIntro}, postponing all longer proofs to the Sections~\ref{sec:MVMFDD} and~\ref{sec:MVMComp}.
\subsection{Definitions and main results}\label{sec:MVMIntro}
First we discuss the present notation to clarify things:\footnote{Throughout we use the following convention: Measure-valued random variables are denoted with $Z$, random variables with values in an arbitrary Polish space are denoted with $X$. Generally, identities involving $X$ are true for all random variables (also for $Z$'s), whereas identities formulated for $Z$'s are only true (or the appearing expressions are only well-defined) for measure-valued random variables.} If $Z$ is a random variable with values in $\prob(S)$, the expression $I(\law(Z))$ plays the role of the expectation of $Z$ and the expression $I(\law(Z|\G))$ plays the role of the conditional expectation of $Z$ given a sub-$\sigma$-algebra $\G$. In order to keep notation short and get identities which look similar to the case of real-valued random variables, we introduce the following  notations
\begin{align}
	\sfE[Z] &:= I(\law(Z)), \\
	\sfE[Z|\G] &:= I(\law(Z|\G)). 
\end{align}
Indeed, recalling \eqref{eq:DefI2}, we have for all bounded Borel $f : S \to \R$
\begin{align*}
	f^\ast(\sfE[Z]) &= \E[f^\ast(Z)],\\
	f^\ast(\sfE[Z|\G]) &=  \E[f^\ast(Z)|\G].
\end{align*}
Using these identities, the properties of the expectation and conditional expectation carry over to  $\sfE[\cdot ]$ and $\sfE[\cdot | \G]$. Indeed, we have for all $\sigma$-algebras $\G_1 \subset \G_2 \subset \F$ and all bounded Borel $f : S \to \R$ 
$$
f^\ast( \sfE[\sfE[Z|\G_2]|\G_1] ) = \E[ f^\ast(\sfE[Z|\G_2])   |\G_1 ] = \E[\E[f^\ast(Z)|\G_2]|\G_1] = \E[f^\ast(Z)|\G_1] = f^\ast(\sfE[Z|\G_1]),
$$
which implies
\begin{align*}
	\sfE[\sfE[Z|\G_2]|\G_1] &= \sfE[Z|\G_1], \\
	\sfE[ \sfE[ Z | \G ]] &= \sfE[Z].
\end{align*}
In view of these considerations, the following is the natural definition of a measure-valued martingale:
\begin{definition}
	A measure-valued martingale on the stochastic base $(\Omega, \F, \P, (\F_t)_{t \in [0,1]})$ is a process $Z=(Z_t)_{t \in [0,1]}$ with values in $\prob(S)$ that satisfies for all $s \le t$
	\begin{align}
		\mathsf{E}[Z_t|\F_s] = Z_s.
	\end{align}
We denote by $\mathcal{M}(S) \subset \prob(D(\prob(S)))$ the set of probability measures on $D(\prob(S))$ such that the canonical process on $D(\prob(S))$ is a measure-valued martingale in its own filtration. This space carries the topology as a subspace of $\prob(D(\prob(S)))$ according to Convention~\ref{con:spaces}. 
\end{definition}
We want to emphasize that in this paper the term ``measure-valued martingale'' always refers to martingales whose values are probability measures; we will never deal with the more general case of  martingales whose values are positive measures of different mass or signed measures. 

We will always denote the canonical process on $D(\prob(S))$ by $Z=(Z_t)_{t \in [0,1]}$.

Given a Borel function $f : S \to \R$ we can assign to every measure-valued process $Z$ the real-valued process $Z[f]$  given by
\begin{align}\label{eq:Bracket}
	Z[f]_t := \int f(s) Z_t(ds) = f^\ast(Z_t).
\end{align}
Throughout this section we will use the processes $Z[f]$ to investigate $Z$. For example, our previous considerations imply that $\mathsf{E}[Z_t|\F_s] = Z_s$ if and only if $\E[f^\ast(Z)|\F_s] =  f^\ast(Z_s)$ for all bounded Borel $f$. Hence, we can characterize measure-valued martingales via these associated real-valued processes:
\begin{lemma}\label{lem:TestMVM}
	A measure-valued process $Z$ is a measure-valued martingale w.r.t.\ $(\F_t)_{t \in [0,1]}$ if and only if $Z[f]$ is an $(\F_t)_{t \in [0,1]}$-martingale for all bounded Borel $f : S \to \R$. 
    In fact, it is sufficient to test against continuous bounded functions or any other class of point separating functions.
\end{lemma}

\begin{remark}\label{rem:MVMTesting}
	One has to read Lemma~\ref{lem:TestMVM} carefully: 
	Let $Z$ be a measure-valued process and $(\F_t)_{t \in [0,1]}$ be the filtration generated by $Z$.  Applying Lemma~\ref{lem:TestMVM} to $Z$ and $(\F_t)_{t \in [0,1]}$ yields that $Z$ is a measure-valued martingale in its own filtration if and only if $Z[f]$ is a martingale w.r.t.\ the filtration generated by $Z$ for all $f : S \to \R$ bounded Borel. It is not true that $Z$ is a measure-valued martingale in its own filtration if and only if all processes $Z[f]$ are martingales in \emph{their} own filtrations.\footnote{In fact, the respective claim is already false in $\R^2$: One can construct a vector-valued process $(X_t)_{t \in [0,1]}$, which is not a martingale, but has the property that for all $v \in \R^2$ the process $(v \cdot X_t)_{t \in [0,1]}$ is a martingale \emph{in its own} filtration. Note that martingales with values in the unit-simplex $\{ x \in \R^n: x^1,\dots, x^n \ge 0,  x^1+\dots +x^n \le 1 \}$ correspond to measure-valued martingales, where $S$ has cardinality $n+1$.   }  
\end{remark}
The proof of the following result is given in Section~\ref{sec:MVMComp}
\begin{proposition}\label{prop:mvmCadlag}
	A measure-valued martingale has a \cadlag{} modification if and only if it is right continuous in probability. In particular, every measure-valued martingale w.r.t.\ a filtration satisfying the usual conditions has a \cadlag{} modification. If a \cadlag{} modification exists, it is unique up to indistinguishability.
\end{proposition}

In particular, any measure-valued martingale $Z$ that is right continuous in probability can be seen as a random variable with values in $D(\prob(S))$. Its law $\law(Z)$ is an element of $\mathcal{M}(S) \subset \prob(D(\prob(S)))$.  Recalling the definition of the evaluation map $e_1(f):=f(1)$ from Remark~\ref{rem:et}, we can define the mapping
\begin{align}\label{eq:Phi}
	\Phi: \mathcal{M}(S) \to \prob(S) : \mu \mapsto I({e_1}_\#\mu) .
\end{align}
If $\mu = \law(Z)$ then $\Phi(\mu) = \mathsf{E}[Z_1]$. Via $\Phi$ we can characterize compactness in $\mathcal{M}(S)$: 

\begin{theorem}\label{thm:M_compactness}
	$\A \subset \mathcal{M}(S)$ is relatively compact in $\prob(D(\prob(S)))$ if and only if $\Phi(\A)$ is relatively compact in $\prob(S)$.
\end{theorem}
Put differently, a collection  of measure-valued martingales $\{ Z^j : j \in J \}$ is  relatively compact in law (i.e.\ $\{ \law(Z^j) : j \in J \} $ is  relatively compact in $\mathcal M(S)$) if and only if $ \{ \mathsf{E}[Z^j _1] : j \in J \}$ is relatively compact subset of $\prob (S)$. The proof of this theorem is given at the end of Section~\ref{sec:MVMComp}.

\begin{corollary}
	The operation $S \mapsto \mathcal{M}(S)$ preserves the properties Lusin, Polish and compact metrizable.
\end{corollary}
\begin{proof}
	The claim for Lusin spaces is already known from Proposition~\ref{prop:pseudo-paths}\ref{it:pseudo-paths.cadlag_Borel}. For compact metrizable spaces it is a direct consequence of Theorem~ \ref{thm:M_compactness}. 
	
	If $S$ is Polish, the mapping $\Phi: \mathcal{M}(S) \to \prob(S) : \mu \mapsto I({e_1}_\#\mu)$ is a continuous mapping into a Polish space. Theorem~\ref{thm:M_compactness}  states that $\Phi$-preimages of compact sets are compact. Hence, we are precisely in the setting of Lemma~\ref{lem:Lusin2Polish} from the appendix and conclude that $\mathcal{M}(S)$ is Polish.
\end{proof}

\noindent\textbf{Continuity points.}
Continuity points of (measure-valued) martingales will play a central role below because they allow us to overcome the issue that point evaluation is not continuous w.r.t.\ the Meyer--Zheng topology. First, we give the definition in a slightly more general framework:
\begin{definition}\label{def:contPt}
Let $X$ be an $S$-valued process. The set of continuity points of $X$, denoted by $\cont(X)$, is defined as the set of $t \in [0,1]$ such that the mapping 
$$
[0,1] \to \prob(S) : s \mapsto \law(X_s)
$$
is continuous at $t$.   
\end{definition}
We emphasize that $\cont(X)$ only depends on $\law(X)$ and not on $X$ itself.

\begin{lemma}\label{lem:TBig}
Let $X$ be an $S$-valued process that is right continuous in probability. Then $\cont(X)$ is co-countable (i.e.\ $[0,1] \setminus\cont(X)$ is countable). In particular, it is dense in $[0,1]$ and satisfies $\lambda(\cont(X)\cup \{1\} )=1$.
\end{lemma}
\begin{proof}
If $X$ is right continuous in probability, then $s \mapsto \law(X_s)$ is right continuous as well, so it has at most countable many discontinuities, cf. Lemma~\ref{lem:ContBigApp} in the appendix.
\end{proof}

In the case of (measure-valued) martingales continuity points are exactly those points, where the trajectories are a.s.\ continuous:
\begin{lemma}
Let $Z$ be a measure-valued martingale. Then $\cont(Z) = \{ t \in [0,1] : Z_t =Z_{t-} \textup{ a.s.} \}$. 
\end{lemma}
\begin{proof}
The corresponding claim is true for $[0,1]$-valued martingales $X$ because $\E |X_t - X_{t -\epsilon}|^2 = \E|X_t|^2 - \E|X_{t-\epsilon}|^2 \to 0$ for $\epsilon \searrow 0$ if and only if $\law(X_{t - \epsilon}) \to \law(X_t)$ for $\epsilon \searrow 0$.
It then extends to measure-valued martingales by testing against a countable convergence determining family.
\end{proof}

After having introduced the notion of continuity points, we can finally state our main result about convergence of finite dimensional distributions:
\begin{theorem}\label{thm:weakFDD}
Let $(Z^n)_n$ be a sequence of $\prob(S)$-valued \cadlag{} martingales, and let $Z$ be a $\prob(S)$-valued \cadlag{} process. Then the following are equivalent:
\begin{enumerate}[label = (\roman*)]
	\item $Z^n \inlaw Z$ as $D(\prob(S))$-valued random variables.
	\item $Z^n_T \inlaw Z_T$  for all finite $T \subset \cont(Z)$.
	\item There is a dense set $T' \subset [0,1]$ that contains 1 such that
	$Z^n_T \inlaw Z_T$  for all finite $T \subset T'$.
\end{enumerate}
If one (and therefore all) of the above are satisfied, $Z$ is a measure-valued martingale.
\end{theorem}
The proof of this theorem is given at the end of Section~\ref{sec:MVMComp}.
\medskip

\noindent\textbf{Terminating measure-valued martingales.}
Let $X$ be an $\F_1$-measurable  $S$-valued random variable on a filtered probability space $(\Omega,\F,\P, (\F_t)_{t \in [0,1]})$  satisfying the usual conditions.

We have for all $f : S \to [0,1]$ Borel and all $s \le t$
$$
f^\ast(\sfE[\law(X|\F_t)|\F_s]) = \E[f^\ast(\law(X|\F_t))|\F_s] = \E[\E[f^\ast(X)|\F_t]|\F_s] = \E[f^\ast(X)|\F_s] = f^\ast(\law(X|\F_s))
$$
and therefore
\begin{align}
	\sfE[\law(X|\F_t)|\F_s] = \law(X|\F_s).
\end{align}
So, $(\law(X|\F_t))_{t \in [0,1]}$ is a measure-valued martingale. Moreover, we have $\law(X|\F_1)=\delta_X$. This motivates the following definition:

\begin{definition}\label{def:termMVM}
We call a measure-valued martingale $Z=(Z_t)_{t \in [0,1]}$ terminating, if there exists a random variable $X$ such that $Z_1=\delta_X$. In this case we say that $Z$ terminates in $X$.  We denote the set of laws of terminating \cadlag{} measure-valued martingales by $\mathcal{M}_0(S)$. 
\end{definition}

Note that $\mathcal{M}_0(S)$ is closed in $\mathcal{M}(S)$ because it is the preimage of the closed set $\delta(S)$ under the continuous mapping $\prob(e_1)$. 

\begin{remark}\label{rem:TermMVM}
If $X$ terminates $Z$, we have for all $f : S \to [0,1]$ Borel and  $t \in [0,1]$
$$
f^\ast(Z_t) = f^\ast(\sfE[Z_1|\F_t]) = f^\ast(\sfE[\delta_X|\F_t]) = \E[f^\ast(\delta_X)|\F_t] = \E[f(X)|\F_t] = f^\ast(\law(X|\F_t)) 
$$
and therefore for all $t \in [0,1]$
\begin{align}
	Z_t = \law(X|\F_t) \quad \text{a.s.}
\end{align}
This observation implies: If $Z$ is a \cadlag{} measure-valued martingale w.r.t.\ $(\F_t)_{t \in [0,1]}$ that terminates at $X$, then $Z$ is the up to indistinguishability unique \cadlag{} version of $(\law(X|\F_t))_{t \in [0,1]}$.
\end{remark}
The following result specializes Theorem~\ref{thm:M_compactness} to the case of terminating martingales.
\begin{corollary}\label{cor:MVMcompTermin}
Let $\{Z^j : j \in J\}$ be a family of \cadlag\ measure-valued martingales such that $Z^j$ terminates at  $X^j$ for all $j \in J$. Then 
$
\{ \law(Z^j) : j \in J \} \subset \mathcal{M}_0(S) 
$
is relatively compact in $\mathcal{M}_0(S)$ if and only if $\{\law(X^j) : j \in J \}$ is relatively compact in $\prob(S)$.
\end{corollary}
\begin{proof}
We have $\Phi(\law(Z^j)) = \sfE[Z^j_1] = \law(X^j)$ for all $j \in J$, so $\{ \law(Z^j) : j \in J  \}$ is relatively compact in $\mathcal{M}(S)$ by Theorem~\ref{thm:M_compactness}. As $\mathcal{M}_0(S)$ is closed in  $\mathcal{M}(S)$, we obtain relative compactness in $\mathcal{M}_0(S)$ as well.
\end{proof}

\subsection{Convergence of finite dimensional distributions}\label{sec:MVMFDD}
 Fubini's theorem implies that convergence in probability of stochastic processes whose path space is equipped with the Meyer--Zheng topology can be seen from different perspectives: 
\begin{lemma}\label{lem:convinprob_fubini}
	Let $(\Omega ,\F,\P)$ be a probability space and equip $[0,1]$ with the measure $\lambda$. Let $X_n, X \colon \Omega \times [0,1] \to S$, $n \in \N$, be jointly measurable. Then the following are equivalent:
	\begin{enumerate}[label = (\roman*)]
		\item $X_n \to X$ in measure w.r.t.\ $\P \otimes \lambda$ as $S$-valued functions.
		\item $X_n \to X$ in measure w.r.t.\ $\P$ as $L_0(([0,1],\lambda);S)$-valued functions.
		\item $X_n \to X$ in measure w.r.t.\ $\lambda$ as $L_0((\Omega,\P);S)$-valued functions.
	\end{enumerate}	
\end{lemma}
\begin{proof}
Straightforward.
\end{proof}
It is well known that $f_n \to f$ in measure if and only if  every subsequence $(f_{n_k})_k$ has a further subsequence $(f_{n_{k_j}})_j$ that converges a.s.\ to $f$, cf. Lemma~\ref{lem:convinprob_simpleppty}. Adopting the point of view $(ii)$ in the preceding lemma, this readily implies

\begin{proposition}\label{prop:ConvInProbFPP}
	Let $X^n, X$, $n\in \N$, be jointly measurable processes with values in $S$. Then the following are equivalent:
	\begin{enumerate}[label = (\roman*)]
		\item $X^n \to X$ in probability as $L_0([0,1];S)$-valued random variables.
		\item For every subsequence $(X^{n_k})_k$ there exists a further subsequence $(X^{n_{k_j}})_j$ and a	set $T \subset [0,1]$ satisfying $\lambda(T)=1$ such that    $X^{n_{k_j}}_t \to X_t$ in probability for all $t \in T$.  
	\end{enumerate}
\end{proposition}
In particular, if $X^n \to X$ in probability as $L_0([0,1];S)$-valued random variables, then there is $\lambda$-full set $T$ and a  subsequence $X^{n_{k}}$ such that $X^{n_{k}}_t \to X_t$ in probability for all $t \in T$. By the Skorohod representation theorem (see e.g.\ \cite[Theorem~3.30]{Ka97}), this translates to the following assertion on convergence in law:
\begin{corollary}\label{cor:FDDsubseq}
	Let $X^n, X$, $n\in\N$, be jointly measurable processes with values in $S$ such that $X^n \to X$ in law as $L_0([0,1];S)$-valued random variables. Then there is a $\lambda$-full set $T \subset [0,1]$ and a  subsequence $(X^{n_{k}})_{k}$ such that the finite dimensional distributions of $(X^{n_k}_t)_{t \in T}$ converge weakly to the  finite dimensional distributions of $(X_t)_{t \in T}$. 
\end{corollary}

Meyer--Zheng used this observation in  \cite[Theorem 11]{MeZh84} to prove that the set of martingale measures is closed in $\prob(D(\R))$. The following Lemma generalizes   this idea:
\begin{lemma}\label{lem:MartFClosed}
	Denote the canonical process on $D(S)$ by $X = (X_t)_{t \in [0,1]}$, let  $\phi : S \to [0,1]$ be continuous and let $\mu^n \in  \prob(D(S))$ be such that  the process $(\phi(X_t))_{t \in [0,1]}$ is a martingale w.r.t.\ the filtration generated by $X$ under $\mu^n$. If $\mu^n \to \mu  \in \prob(D(S))$, then $(\phi(X_t))_{t \in [0,1]}$ is a martingale w.r.t.\ the filtration generated by $X$ under $\mu$.
\end{lemma}
\begin{proof}
	By Corollary~\ref{cor:FDDsubseq} there exist a $\lambda$-full set $T \subset [0,1]$ and a subsequence $(\mu^{n_k})_k$ such that for all $t_1, \dots, t_m \in T$ we have 
	\begin{align}\label{eq:prf:MartFClosed0}
		\law_{\mu^{n_k}}(X_{t_1},\dots, X_{t_m}) \to \law_{\mu}(X_{t_1},\dots, X_{t_m}).
	\end{align}
	As $(\phi(X_t))_{ t \in [0,1]}$ is a martingale w.r.t.\ the filtration generated by $X$ under $\mu^{n_k}$ we have for all $s \le t$
	$$
	\E_{\mu^{n_k}}[ \phi(X_t) | X_r : r \le s ] = \phi(X_s).
	$$
	This is equivalent to
	\begin{align}\label{eq:prf:MartFClosed}
		\E_{\mu^{n_k}}[ \phi(X_t) g_1(X_{s_1}) \cdots g_m(X_{s_m})   ] = \E_{\mu^{n_k}}[ \phi(X_s) g_1(X_{s_1}) \cdots g_m(X_{s_m})   ] 
	\end{align}
	for all $s_1 \le \dots \le s_m \le s \le t$ and all $g_1, \dots, g_m \in C_b(S)$. If $s_1, \dots, s_m, s,t \in T$ we can take the limit $k \to \infty$ in \eqref{eq:prf:MartFClosed} due to \eqref{eq:prf:MartFClosed0} and obtain
\begin{align}\label{eq:prf:MartFClosed2}
		\E_{\mu}[ \phi(X_t) g_1(X_{s_1}) \cdots g_m(X_{s_m})   ] = \E_{\mu}[ \phi(X_s) g_1(X_{s_1}) \cdots g_m(X_{s_m})   ].
	\end{align}
	As $X$ is \cadlag, \eqref{eq:prf:MartFClosed2} for all $s_1, \dots, s_m, s,t \in T$ implies that \eqref{eq:prf:MartFClosed2} holds for all $s_1, \dots, s_m, s,t \in [0,1]$, so under $\mu$,  $(\phi(X_t))_{ t \in [0,1] }$ is a martingale w.r.t.\ the filtration generated by $X$.
\end{proof}

Lemma~\ref{lem:MartFClosed} will be helpful several times in this paper. Indeed, it addresses exactly the issue discussed in Remark~\ref{rem:MVMTesting}: A $\prob(S)$-valued processes $Z$ is a measure-valued martingale in its own filtration if and only if for all $f : S \to\R$ the processes $Z[f]$ are martingales w.r.t.\ the filtration generated by $Z$. So, applying Lemma~\ref{lem:MartFClosed} to all functions $f^\ast$, where $f \in C_b(S)$, readily implies
\begin{corollary}\label{cor:MVMclosedPDPS}
	$\mathcal{M}(S)$ is closed in $\prob(D(\prob(S)))$.
\end{corollary}
In the context of Corollary~\ref{cor:FDDsubseq} we can say even more, when restricting to the closed set of martingale measures: We do not need to pass to a subsequence and we can characterize the set $T \subset [0,1]$ on which the finite dimensional distributions converge. A first result in this direction is the following elementary observation about increasing functions. For the sake of completeness, a proof of this lemma is given in the appendix.

\begin{lemma}\label{lem:convinprob_monotone_pw}
	Let $f, f_n : [0,1] \to \R$ be increasing functions. Then the following are equivalent:
	\begin{enumerate}[label = (\roman*)]
		\item $f_n \to f$ in the Meyer--Zheng topology.
		\item $f_n \to f$ pointwise in all continuity points of $f$ and $f_n(1) \to f(1)$.
	\end{enumerate}
\end{lemma}

Applying this result to the increasing function $t \mapsto \E X_t^2$, where $X$ is a martingale, yields

\begin{proposition}\label{prop:martMargConv1}
	Let $(X^n)_n$ be \cadlag\ martingales in their own filtration with values in $[0,1]$ and $X$ be a measurable process which is right continuous in probability. Denote by $T$ the set of continuity points of the mapping $t \mapsto \E |X_t|^2$. Suppose that $X^n \to X$ in probability w.r.t.\ the Meyer--Zheng topology on the path space.
	
	Then $X$ is a martingale in its own filtration, $[0,1] \setminus T$ is countable and we have  $X^n_t \to X_t$ in probability for all $t \in T$.
\end{proposition}
In order to prove Proposition~\ref{prop:martMargConv1}, we define for $f \in L_0([0,1])$ and $\delta >0$ 
\begin{align}\label{eq:def_eta_delta}
	\eta_\delta(f)_t := \frac{1}{\lambda([t,t+\delta])} \int_{[t,t+\delta]} f d\lambda.
\end{align}
Note that the mapping 
$$L_0([0,1]) \to \R : f \mapsto \eta_\delta(f)_t$$ is well-defined and continuous for all $t \in [0,1]$ and $\delta>0$. 
\begin{proof}[Proof of Proposition~\ref{prop:martMargConv1}]
	$T$ is co-countable by Lemma~\ref{lem:TBig} and $X$ is a martingale by Lemma~\ref{lem:MartFClosed} applied with $\phi=\id$. Let $t \in T$. Our aim is to show that $\E|X^n_t - X_t|^2 \to 0$ as $n \to \infty$. Since $\lambda(\{1\})>0$, the claim is trivial for $t=1$, so we may assume $t<1$. Fix $\epsilon >0$ and denote $f_s := \E|X_s|^2$ and $f^n_s := \E|X^n_s|^2$. As $f$ is right continuous in $t$ and $T$ is co-countable, there is $\delta>0$ such that $t+\delta \in T$ and $|f_t - f_{t+\delta}|<\epsilon$. 
	
	By applying Lemma~\ref{lem:convinprob_monotone_pw} to the increasing functions $f_n, f$ at the points $t, t+\delta \in T$, we find some $n_0$ such that $|f^n_t-f_t| < \epsilon$ and $|f^n_{t+\delta} -f_{t +\delta}| < \epsilon$ for all $n \ge n_0$. 
	
	Using Jensen's inequality we estimate
	\begin{align*}
		\nonumber
		\E|X_t - \eta_\delta(X)_t|^2 &\le \frac{1}{\lambda([t,t+\delta])}  \int_t^{t + \delta}  \E|X_t-X_s|^2 ds \\
		\nonumber
		& = \frac{1}{\lambda([t,t+\delta])} \int_t^{t + \delta} \E|X_s|^2 -\E|X_t|^2 ds
		\\
		&\le \E|X_{t+\delta}|^2 - \E|X_t|^2.
	\end{align*}
	Notice that the analogous inequality holds true when $X$ is replaced by $X^n$. Using this, we can further estimate for $n \ge n_0$
	
	\begin{align*}
		\frac{1}{3}
		\E|X^n_t-X_t|^2  
		&\le \E|X^n_t - \eta_\delta(X^n)_t|^2 + \E|\eta_\delta(X^n)_t -\eta_\delta(X)_t|^2 + \E|\eta_\delta(X)_t -X_t|^2\\
		&\le \E|X^n_{t+\delta}|^2 - \E|X^n_t|^2 + \E|\eta_\delta(X^n)_t -\eta_\delta(X)_t|^2 + \E|X_{t+\delta}|^2 - \E|X_t|^2 \\
		&\le |f^n_{t+\delta} - f_{t+\delta} | + |f_{t+\delta} -f_t| + |f^n_t -f_t| + \E|\eta_\delta(X^n)_t -\eta_\delta(X)_t|^2 + |f_{t + \delta} -f_t|\\
		&\le \E|\eta_\delta(X^n)_t -\eta_\delta(X)_t|^2 + 4 \epsilon.
	\end{align*}
	As the mapping $g \mapsto \eta_\delta(g)_t$ is continuous w.r.t.\ the Meyer--Zheng topology and $|X^n_t|,|X_t| \le 1$, we have  $\E|\eta_\delta(X^n)_t -\eta_\delta(X)_t|^2  \to 0$ for $n \to \infty$ and we conclude $\E|X^n_t-X_t|^2  \to 0$.
\end{proof}

Recall that the set of continuity points of a process $X$, denoted by $\cont(X)$, is defined as the set of continuity points of the mapping 
$
[0,1] \to \prob(S) : t \mapsto \law(X_t).
$
Since the map $\prob([0,1]) \to [0,1] : \mu \mapsto \int x^2 \mu(dx)$ is continuous, every continuity point of $t \mapsto \law(X_t)$ is also a continuity point of $t \mapsto \E|X_t|^2$. Therefore, we conclude $X^n_t \to X_t$ for all $t \in \cont(X)$ in Proposition~\ref{prop:martMargConv1}. 

Using this notion of continuity points, we can easily extend the result to measure-valued martingales. 

\begin{corollary}\label{cor:MVMMargConv1}
Let $Z^n, Z$ be $\prob(S)$-valued martingales such that $Z^n \to Z$ in probability w.r.t.\ the Meyer--Zheng topology on the path space. Let $T \subset \cont(Z)$ be finite. Then $Z^n_T \to Z_T$ in probability.
\end{corollary}
\begin{proof}
Let $\{ f_k : k \in \N \} \subset C_b(S)$ such that $\{ f_k^\ast : k \in \N \}$ is convergence determining on $\prob(S)$. By continuity of $f_k$ we have $T \subset \cont(Z) \subset \cont(Z[f_k])$ for all $k \in \N$. Fix $t \in T$. Proposition~\ref{prop:martMargConv1} applied to $Z^n[f_k], Z[f_k]$ yields 
$$
f_k^\ast(Z^n_t) = Z^n[f_k]_t \to Z[f_k]_t =  f_k^\ast(Z_t)
$$
in probability, for all $k \in \N$, so we can conclude $Z^n_t \to Z_t$ in probability. Since $T$ is finite, this implies $Z^n_T \to Z_T$ in probability.
\end{proof}

\begin{corollary}\label{cor:MVMMargConv2}
Let $Z^n, Z$ be $\prob(S)$-valued \cadlag{} martingales such that $Z^n \inlaw Z$  w.r.t.\ Meyer--Zheng on the path space, and let $T \subset \cont(Z)$. Then $(Z^n ,Z^n_T) \inlaw (Z,Z_T)$.
\end{corollary}
\begin{proof}
The Skorohod representation theorem  applied to $Z^n,Z$ regarded as random variables with values in the Lusin space $D(\prob(S))$ yields the existence of a probability space $(\Omega,\F,\P)$ and random variables  $\widetilde{Z}^n, \widetilde{Z}$ on this space such that $\widetilde{Z}^n \sim Z^n$,  $\widetilde{Z} \sim Z$ and $\widetilde{Z}^n\to \widetilde{Z}$ $\P$-a.s.\ and hence in probability. 

Let $T \subset \cont(Z)$ and recall that $ \cont(Z)$ merely depends on $\law(Z)$ and not on $Z$ itself, so we have $\cont(\widetilde{Z}) = \cont(Z)$. Corollary~\ref{cor:MVMMargConv1} implies that $\widetilde{Z}^n_T \to \widetilde{Z}_T$ in probability. Hence, $(\widetilde{Z}^n, \widetilde{Z}^n_T) \to (\widetilde{Z}, \widetilde{Z}_T)$ in probability. We conclude $\law(Z^n,Z^n_T) = \law(\widetilde{Z}^n, \widetilde{Z}^n_T) \to \law(\widetilde{Z}, \widetilde{Z}_T) = \law(Z,Z_T)$.
\end{proof}

\subsection{Tightness, compactness and existence of \cadlag{} modifications}\label{sec:MVMComp}
In this section we finally prove Theorem~\ref{thm:M_compactness} and Theorem~\ref{thm:weakFDD}. The following lemma can be seen as a Doob maximal inequality for measure-valued martingales.
\begin{lemma}\label{lem:MVM_tight}
Let $\A$ be a family of measure-valued martingales such that $\{ I(\law(Z_1)) : Z \in \A \}$ is tight and let $T \subset [0,1]$ be countable. Then, for every $\epsilon>0$, there exists a compact set $\mathcal{K}_\epsilon \subset \prob(S)$ such that for all $Z \in \A$ 
\begin{align}\label{eq:lem:MVM_tight1}
\P[ Z_t \in \mathcal{K}_\epsilon \text{ for all }t \in T] \ge 1 - \epsilon.
\end{align}
If additionally the measure-valued martingales have \cadlag{} paths, we have 
\begin{align}\label{eq:lem:MVM_tight2}
\P[ Z_t \in \mathcal{K}_\epsilon \text{ for all }t \in [0,1]  ] \ge 1 - \epsilon.
\end{align}
\end{lemma}
\begin{proof}
For all $n \in \N$, let $K_n \subset S$ be a compact set such that $\E[Z_1(K_n^c)] = I(\law(X))(K_n^c) \le 2^{-2n}\epsilon$ for all $Z \in \A$. As $(Z_t(K_n^c))_t$ is a $[0,1]$-valued martingale by Lemma \ref{lem:TestMVM}, Doob's martingale inequality implies 
$$
\P\left[  \sup_{t \in T } Z_t(K^c_n) \ge 2^{-n} \right] \le {2^{n} \E[Z_1(K_n^c)]} \le  2^{-n}\epsilon.
$$
The set
$$
\mathcal{K}_\epsilon := \left\{ p \in \prob(S) : p(K_n^c) \le 2^{-n} \text{ for all } n \in \N \right\}
$$
is tight by construction and hence compact by Prohorov's theorem.
We have
$$
\P[ Z_t \in \mathcal{K}_\epsilon \text{ for all }t \in T ] \ge 1 - \sum_{n \in \N} \P[\exists t \in T : Z_t(K^c_n) \ge 2^{-n} ] \ge   1 - \sum_{n \in \N} 2^{-n}\epsilon = 1-\epsilon,
$$
for all $Z \in \A$. If $Z$ has \cadlag{} paths, \eqref{eq:lem:MVM_tight1} applied to  $T=[0,1] \cap \Q$ readily implies \eqref{eq:lem:MVM_tight2}. 
\end{proof}
Using this, we can derive the existence of \cadlag{} modifications for measure-valued martingales:
\begin{proof}[Proof of Proposition~\ref{prop:mvmCadlag}]
Let $Z$ be a measure-valued martingale on $(\Omega,\F,\P)$ that is right continuous in probability and let $\{\phi_n \colon n \in \N\}$ be a convergence determining family on $S$ that is closed under multiplication. For all $n \in \N$, $Z[\phi_n]$ is a real-valued martingale that is right-continuous in probability. Thus it admits a \cadlag{} version $Y^n$.

By Lemma~\ref{lem:MVM_tight}, there is for every $m \in \N$ a set $\Omega^m \in \F$ s.t. $\P[\Omega^m] \ge 1 - \frac1m$ and a compact set $\mathcal{K}_m \subseteq \prob(S)$ such that for all $\omega \in \Omega^m$, $n \in \N$ and $t \in [0,1] \cap \Q$
\begin{equation}\label{eq:MVM_version_det_fam}
	Z[\phi_n]_t(\omega) = Y^n_t(\omega )\quad\mbox{and}\quad Z_t(\omega) \in \mathcal{K}_m. 
\end{equation}
The set $\Omega':= \bigcup_{m \in \N} \Omega^m$ is $\P$-full. On $\Omega'$ we define
\begin{align}\label{eq:prf:cadlagV}
Y_t := \lim_{s \searrow t,  s \in  \Q} Z_s.
\end{align}
First, we argue that this limit exists for all $\omega \in \Omega'$. If $\omega \in \Omega'$, then there is an $m \in \N$ such that $\omega \in \Omega^m$. Let $(s_k)_k$ be as sequence in $\Q \cap (t,1]$ that converges to $t$. As $Z_{s_k}(\omega) \in \mathcal{K}_m$ for all $k \in \N$, this sequences has at least one limit point. By the right-continuity of $Y^n$ and \eqref{eq:MVM_version_det_fam} any limitpoint $\mu$ of this sequence has to satisfy $\phi_n^\ast(\mu) = Y^n_t$, so the limit point is unique, i.e.\ the sequence converges. By the same reasoning, this limit point is independent of the choice of  $(s_k)_k$, so the limit in   \eqref{eq:prf:cadlagV} exists.

As $Z$ is right continuous in probability, we have $Z_t=Y_t$ a.s.\ for all $t \in [0,1]$. Any further \cadlag{} modification $Y'$ satisfies $Y'_t=Z_t=Y_t$ a.s.\ for all $t \in [0,1]$ and, as $Y, Y'$ are both \cadlag{} $Y$ and $Y'$ are indistinguishable.
\end{proof}

Let $\prob_M(D([0,1]))$ denote the set of all $\mu \in \prob(D([0,1]))$ such that the coordinate process on $D([0,1])$ is a martingale under $\mu$. Meyer--Zheng proved in \cite[Theorem 4]{MeZh84} that $\prob_M(D([0,1]))$ is compact.

As discussed in Section~\ref{sec:MZLusin}, we can identify a path $f \in D([0,1])$ with the corresponding pseudopath $\iota_{[0,1]}(f) \in \prob([0,1] \times [0,1])$, or more rigorously $\iota_{[0,1]} : D([0,1]) \to \prob([0,1] \times [0,1])$ is a topological embedding. Hence, we can consider measures on $D([0,1])$ as measures on $\prob([0,1] \times [0,1])$, rigorously, 
$$
\prob(\iota_{[0,1]}) : \prob(D([0,1])) \to \prob(\prob( [0,1] \times [0,1] ))
$$
is again an embedding, cf. Remark~\ref{rem:P}. As $\prob_M(D([0,1]))$ is compact, $\prob(\iota_{[0,1]})(  \prob_M(D([0,1])))$ is a closed subset of $\prob(\prob( [0,1] \times [0,1] ))$. As already discussed in Remark~\ref{rem:pseudoL0} we can suppress the embedding $\iota_{[0,1]}$ in order to avoid notational excess, so we can just say $\prob_M(D([0,1]))$ is a closed subset of $\prob(\prob( [0,1] \times [0,1] ))$.

The following proposition generalizes this fact from $[0,1]$-valued martingales to measure-valued martingales. As usual, we prove results for a measure-valued process $Z$ by considering the associated real-valued process $Z[f]$, where $f \in C_b(S)$, which were introduced in \eqref{eq:Bracket}. However, keeping Remark~\ref{rem:MVMTesting} in mind, we have to be careful here.
 
\begin{proposition}\label{prop:MclosedBig}
$\mathcal{M}(S)$ is a closed subset of $\prob(\prob([0,1] \times \prob(S) ))$.
\end{proposition}
\begin{proof}
Consider the set 
$$
\A:= \{  \mu \in \prob(\prob([0,1] \times \prob(S) )) : \Psi(f^\ast)_\# \mu \in  \prob_M(D([0,1])) \text{ for all } f \in C_b(S) \}.
$$	
We first show the following chain of inclusions:
\begin{align}\label{eq:prf:Mclosed}
\mathcal{M}(S) \subset \A \subset \prob(D(\prob(S))) \subset \prob(\prob([0,1] \times \prob(S) )).
\end{align}

The last inclusion is trivial. In order to show the first inclusion, let $\mu \in \mathcal{M}(S)$ and denote by $Z$ the canonical process on $D(\prob(S))$. For all $f \in C_b(S)$, $Z[f] = \Psi(f^\ast)(Z)$ is a \cadlag{} martingale under $\mu$, so $\Psi(f^\ast)_\# \mu \in  \prob_M(D([0,1]))$. Hence, $\mu \in \A$. 

To show the second inclusion, let $\{ f_n : n \in \N \}$ be convergence determining on $S$ and closed under multiplication and recall that Lemma~\ref{lem:P(S)_ptsep_convdet} states that  $\{ f^\ast_n : n \in \N \}$ is convergence determining in $\prob(S)$ as well. Recall that  Proposition~\ref{prop:Psi_lift}(c) states that $p \in \prob( [0,1] \times \prob(S) )$ is a \cadlag{} pseudo-path if and only if $\Psi(f_n^\ast)(p)$ is a \cadlag{} pseudo-path for all $n \in \N$. Let $\mu \in \A$, then $\Psi(f_n^\ast)_\# \mu$ is concentrated on $[0,1]$-valued \cadlag{} paths for all $n \in \N$, so $\mu$ is concentrated on $\prob(S)$-valued \cadlag{} paths, i.e.\ $\mu \in \prob(D(\prob(S)))$, and we have shown \eqref{eq:prf:Mclosed}.

Corollary~\ref{cor:MVMclosedPDPS} states that $\mathcal{M}(S)$ is closed in $\prob(D(\prob(S)))$, so it is closed in $\A$ as well. But, it is easy to see that $\A$ itself is closed in $ \prob(\prob([0,1] \times \prob(S) ))$.  Indeed,  $\prob(\Psi(f^\ast))$ is continuous for all $f \in C_b(S)$ and  $\prob(\iota_{[0,1]})(  \prob_M(D([0,1])))$ is closed in $\prob( \prob( [0,1] \times [0,1] ))$.  Hence, we conclude that $\mathcal{M}(S)$ is closed in $\prob(\prob([0,1] \times \prob(S) ))$.
\end{proof}

If $S$ is compact, then $\prob(\prob([0,1] \times \prob(S) ))$ is compact as well, so Proposition~\ref{prop:MclosedBig} implies that $\mathcal{M}(S)$ is compact. So the only thing we need to do in the upcoming proof of Theorem~\ref{thm:M_compactness} is to relax the compactness assumption on $S$. This can be done  using Lemma~\ref{lem:MVM_tight}.

\begin{proof}[Proof of Theorem~\ref{thm:M_compactness}]
Let $\A \subset \mathcal{M}(S)$ be such that  $\{ I(\law(Z_1)) : Z \in \A \}$ is relatively compact in $\prob(S)$. We need to show that every sequence  $(Z^n)_n$ in $\A$ has a convergent subsequence with limit $Z \in \mathcal{M}(S)$. After possibly passing to a subsequence, we may assume that $(I(\law(Z_1^n)))_n$ is a convergent sequence in $\prob(S)$. By Theorem~\ref{thm:prohorov_lusin}(b), the set $\{ I(\law(Z^n_1)) : n \in \N \} $ is tight in $\prob(S)$. Denote 
$$\mu^n:= \law(Z^n) \in \mathcal{M}(S) \subset  \prob(\prob([0,1] \times \prob(S) )).$$
 Our aim is to show that $(\mu^n)_n$ is tight in $ \prob(\prob([0,1] \times \prob(S) ))$. Let $\epsilon>0$. By Lemma~\ref{lem:MVM_tight} there is a compact set $\mathcal{K}_\epsilon$ such that for all $n \in \N$, 
 $$\P[ Z^n_t \in \mathcal{K}_\epsilon \text{ for all }t\in[0,1] ]  \ge 1 -\epsilon. $$
 Put differently, this asserts that $\mu^n( \prob(  [0,1] \times \mathcal{K}_\epsilon  ) ) \ge 1-\epsilon$. As  $\prob(  [0,1] \times \mathcal{K}_\epsilon  )$ is compact, we can conclude tightness of the sequence $(\mu^n)_n$. So, Prohorov's theorem implies that there is some $\mu \in \prob(\prob([0,1] \times \prob(S) ))$ such that, after passing to a subsequence, we have $\mu^n \to \mu$. By Proposition~\ref{prop:MclosedBig} we have $\mu \in \mathcal{M}(S)$. 
\end{proof}

\begin{proof}[Proof of Theorem~\ref{thm:weakFDD}]
$(i) \implies (ii)$ was already shown in Corollary~\ref{cor:MVMMargConv2}.

$(ii) \implies (iii)$ is true since  $\cont( Z)$ is dense and contains 1 according to Lemma~\ref{lem:TBig}.

$(iii) \implies (i)$: Let $D \subset [0,1]$ be dense and contain 1. Let $(Z^n)_n$ be a sequence of \cadlag{} $\prob(S)$-valued martingales and let $Z$ be a $\prob(S)$-valued \cadlag{} process such that $Z^n_T \inlaw Z_T$ for all finite $T \subset D$. 

First note that we have $Z^n_1 \inlaw Z_1$ and hence, due to the continuity of the intensity operator, $\sfE[Z^n] \to \sfE[Z]$. So, Theorem~\ref{thm:M_compactness} implies that $(\law(Z^n))_n$ is relatively compact in $\mathcal{M}(S)$. Hence, it suffices to show that any limit point of $(\law(Z^n))_n$ is $\law(Z)$.

So, let $Z^{n_j} \inlaw Y$ for some $\prob(S)$-valued \cadlag{} process $Y$. The already shown implication $(i) \implies (ii)$ yields that  $Z^{n_j}_T \inlaw Y_T$ for all finite $T \subset \cont(Y)$. This implies $\law(Y_T) = \law(Z_T)$ for all finite $T \subset D \cap \cont(Y)$. Since $Y,Z$ are \cadlag{} and $ D \cap \cont(Y)$ is dense and  contains 1, we conclude $\law(Y)=\law(Z)$. Hence, $Z^n \inlaw Z$.

If either $(i), (ii)$ or  $(iii)$ is satisfied, then $(i)$ is satisfied as we have already shown that they are equivalent. So, we have  $Z^n \inlaw Z$ and conclude that $Z$ is a measure-valued martingale because $\mathcal{M}(S)$ is closed in $\prob(D(\prob(S)))$, cf. Corollary~\ref{cor:MVMclosedPDPS}. 
\end{proof}

\section{Filtered random variables and their prediction process}\label{sec:FP}
\subsection{Filtered processes and filtered random variables} 
Recall from the introduction that a continuous-time filtered process is a 5-tuple 
\[
\fp X = (\Omega,\F,\P, (\F_t)_{t \in [0,1]}, (X_t)_{t \in [0,1]}),
\]
where $(\Omega,\F,\P)$ is a probability space, $(\F_t)_{t \in [0,1]}$ is a filtration satisfying the usual conditions, and $(X_t)_{t \in [0,1]}$ is a \cadlag\ process that is  adapted to   $(\F_t)_{t \in [0,1]}$. The collection of a filtered processes with $S$-valued \cadlag\ paths (i.e.\ $X_t \in S$ for all $t \in [0,1]$) is denoted by $\FP(S)$, or just by $\FP$ if the space $S$ is clear from the context. 

Similarly, a discrete time filtered process is a 5-tuple consisting of a probability space $(\Omega,\F,\P)$, a complete filtration $(\F_t)_{t=1}^N$ on $(\Omega,\F,\P)$ and an  $(\F_{t})_{t=1}^N$-adapted $S$-valued process $X = (X_t)_{t=1}^N$. The collection of all $S$-valued filtered processes with $N$ timesteps is denoted by $\FP_N(S)$. 

\medskip
 
Next, we introduce the more abstract setting of filtered random variables: Here the \emph{process} $(X_t)_{t \in [0,1]}$ is replaced by a \emph{random variable} $X$.  This notion was introduced by Hoover--Keisler \cite{HoKe84}. As a stochastic process $(X_t)_{t \in [0,1]}$ can be viewed as a random variable which takes values in a path space, the Hoover--Keisler viewpoint is more general. However,  the actual reason to take this stance is that it is technically more convenient for the purposes of the next sections.

\begin{definition}[filtered random variable in discrete time]\label{def:FP_discr} 
	Let $S$ be a Lusin space and $N \in \N$. An $S$-valued filtered random variable with $N$ timesteps is a 5-tuple
	\[
	\fp X = \left( \Omega, \F, \P, (\F_t)_{t=1}^N, X \right)
	\]	
	consisting of a probability space $(\Omega,\F,\P)$,  a complete filtration $(\F_t)_{t=1}^N$ on $(\Omega,\F,\P)$ and an  $\F_{N}$-measurable random variable $X : \Omega \to S$.
	
We write $\FR_N(S)$ for the class of all $S$-valued filtered random variables with $N$ time steps; if $S$ is clear from the context we write $\FR_N$ instead of $\FR_N(S)$. If $\fp X$ is a filtered random variable we always refer to the elements of this tuple by $\Omega^\fp X, \F^\fp X, \P^\fp X, (\F_t^\fp X)_{t =1}^N, X$, i.e.\  $\Omega^\fp X$ refers to the base set of the probability space of $\fp X$ etc.
\end{definition}	

%
%
%
\begin{definition}[filtered random variable in continuous time]\label{def:fp_cont}
	Let $S$ be a Lusin space. An $S$-valued filtered random variable in continuous time is a 5-tuple
	\[
	\fp X = \left( \Omega, \F, \P, (\F_t)_{t \in [0,1]}, X \right)
	\]	
	consisting of a probability space $(\Omega,\F,\P)$, a right-continuous complete filtration $(\F_t)_{t \in [0,1]} $ and an $\F_{1}$-measurable random variable $X : \Omega \to S$.

	The class of all $S$-valued filtered random variables is denoted by $\FR(S)$; if $S$ is clear from the context we write $\FR$ instead of $\FR(S)$. Again, we refer to the elements of the tuple $\fp X$ by $\Omega^\fp X, \F^\fp X, \P^\fp X$, $(\F_t^\fp X)_{t \in [0,1]}$ and $X$.
\end{definition}

As mentioned above, every filtered process is in particular a filtered random variable, but the converse is not true because in the definition of a filtered random variable no adaptedness condition is present. The relation between filtered processes and filtered random variables is further examined in Section~\ref{sec:adapted}. In particular, we show that these concepts are equivalent in a certain sense and that results on filtered random variables can be translated into results on filtered processes and vice versa.

\subsection{The prediction process}
We start with defining the prediction processes in discrete time.
\begin{definition}[prediction process in discrete time]
The prediction process of a filtered random variable $\fp X \in \FR_N(S)$ is the $\prob(S)$-valued filtered process $\bpp(\fp X)$ defined by
$$
\bpp(\fp X) = (\Omega^{\fp X}, \F^{\fp X}, \P^{\fp X}, (\F_t^{\fp X})_{t=1}^N, \pp(\fp X)),
$$
where 
$
\pp_t(\fp X) = \law(X|\F_t^{\fp X})
$ for all $t = 1,\dots, N$. 
\end{definition}

Note that $\pp(\fp X)$ is a discrete time measure-valued $(\F_t^{\fp X})_{t=1}^N$-martingale that terminates at $X$. The definition of the prediction processes in continuous time is analogous, but technically a bit more involved.

\begin{definition}[prediction process in continuous time]
The prediction process of a filtered random variable $\fp X \in \FR(S)$ is the $D(\prob(S))$-valued filtered process $\bpp(\fp X)$ defined by
$$
\bpp(\fp X) = (\Omega^{\fp X}, \F^{\fp X}, \P^{\fp X}, (\F_t^{\fp X})_{t=1}^N, \pp(\fp X)),
$$
where 
$
\pp(\fp X)$ is the up to indistinguishablity unique \cadlag{} version of the measure-valued martingale $(\law(X|\F_t^{\fp X}))_{t \in [0,1]}$.  
\end{definition}

An elementary, but crucial observation is
\begin{lemma}\label{lem:condLawSelfAware}
	Let $X$ be an $S$-valued random variable on $(\Omega,\F,\P)$ and $\G$ be a sub-$\sigma$-algebra of $\F$. Then we have
	\begin{align}\label{eq:condLawSelfAware}
		\law(X|\G) = \law(X|\law(X|\G)).
	\end{align}
	More generally, if $\mathcal{H}$ is a $\sigma$-algebra satisfying $\sigma(\law(X|\G)) \subset \mathcal{H} \subset \G$, we have $\law(X|\G) = \law(X|\mathcal{H})$.
\end{lemma}
\begin{proof}
	Clearly, it suffices to prove the second claim. Let $\mathcal{H}$ be a $\sigma$-algebra satisfying $\sigma(\law(X|\G)) \subset \mathcal{H} \subset \G$ and $f: S \to [0,1]$ be Borel. Since $\E[f(X|\G)]= f^\ast(\law(X|\G))$, $\E[f(X)|\G]$ is $\sigma(\law(X|\G))$-measurable. Therefore, it is $\mathcal{H}$-measurable as well, which yields
	$$
	\E[f(X)|\G] = \E[ \E[ f(X)|\G ] | \mathcal{H} ] = \E[f(X)|\mathcal{H}],
	$$
	where the second equality is due to the tower property. 
\end{proof}
Equation \eqref{eq:condLawSelfAware} expresses precisely a conditional independence property (see e.g.\ \cite[Proposition~5.6]{Ka97}), namely: Given the random measure $\law(X|\G)$, the random variable $X$ is independent of the $\sigma$-algebra $\G$. Loosely speaking, \eqref{eq:condLawSelfAware} says that $\law(X|\G)$ already contains all information that $\G$ has about $X$.

In particular,  $\pp_t^1(\fp X) = \law(X|\F^{\fp X}_t)$ contains all information that $\F_t^{\fp X}$ has about $X$.  On a first glance, one might think that this implies that $\pp^1(\fp X)$ contains all information that the filtration $(\F^{\fp X}_t)_{t \in [0,1]}$ has about $X$. This is not true: $\pp^1(\fp X)$ does not contain ``information of higher order'', e.g. for $s<t$, information of $\F^{\fp X}_s$ on whether $\F^{\fp X}_t$ knows something about $X$. Such information is relevant in sequential decision problems. For instance \cite[Section~7]{BaBePa21} provides an example of two filtered processes that have the same prediction process but yield different values in optimal stopping problems, see also Remark \ref{rem:FP} below.

Indeed, $\law(\law(X|\F^{\fp X}_t)|\F_s^{\fp X})$ contains the information that $\F_s^{\fp X}$ has about $\law(X|\F^{\fp X}_t)$, which is precisely the information that $\F^{\fp X}_s$ has on what $\F^{\fp X}_t$ knows about $X$. This motivates to consider iterated prediction processes: $\bpp(\fp X)$ is by definition again a filtered process, hence we can simply consider its prediction process $\bpp(\bpp(\fp X))$. This \emph{iterated prediction process} of $\fp X$ takes values in the space $\prob(\prob(S)^N)^N$ resp. $D(\prob(D(\prob(S))))$. As the path spaces of the iterated prediction processes are increasingly complicated, we need to introduce a notation for these spaces:
\begin{definition}[path spaces of higher rank prediction processes]
In the case of $N$ discrete time steps we define by induction on $r$	
$$
M_0^{(N)}(S):=S, \qquad M_{r+1}^{(N)}(S) := \prob(M_{r}^{(N)}(S))^N. 
$$	
In the continuous time case we define 
$$
M_0(S):=S, \qquad M_{r+1}(S) := \{ f \in D(\prob(M_r(S))) : f(1) \in \delta(M_r) \}.   
$$
For $r= \infty$, we set 
$$
M_\infty(S) := \left\{ f \in  D\left(\prod_{j=0}^\infty \prob(M_j(S)) \right) : f(1) \in \prod_{j=0}^\infty \delta(M_j(S))    \right\}.
$$

\end{definition}
Next, we iterate the construction of the prediction processes in order to define higher rank prediction processes.
\begin{definition}
\label{def:iterated_pp} 
Let $\fp X$ be a discrete- or continuous-time filtered random variable. We define by induction on the rank $r \in \N_0$ 
\begin{align*}
	\pp^0(\fp X) :=& X, & \bpp^0(\fp X) :=& \fp X, \\
	\pp^{r+1}({\fp X}):=& \pp(\bpp^r({\fp X})), & \bpp^{r+1}({\fp X}):=& \bpp(\bpp^r({\fp X})).
\end{align*}
The rank $r$ prediction process $\bpp^r(\fp X)$ is an $M_r$- resp. $M_r^{(N)}$-valued filtered process. The prediction process of rank $\infty$ of a continuous time filtered process is given by
\begin{align*}
	\pp_t^\infty(\fp X) :=& (\pp_t^1(\fp X),\pp_t^2(\fp X),\ldots), \\ 
	\bpp^\infty(\fp X) :=& \left( \Omega^\fp X, \F^\fp X, \P^\fp X,  (\F_t^\fp X)_{t \in [0,1]}, \pp^\infty(\fp X) \right).
\end{align*}
It is an $M_\infty$-valued filtered process. 
\end{definition}
We collect a few remarks on the definition of the iterated prediction process:
\begin{remark}
	\begin{enumerate}[label = (\alph*)]\label{rem:afterDefPp}
		\item The rank $r$ predictions processes $\pp^r(\fp X)$ is a measure-valued martingale that terminates at $\pp^{r-1}(\fp X)$ for all $r \in \N$.  Moreover, $\pp^\infty(\fp X)$ is a martingale in the sense that it is a (countable) vector, all of whose entries are measure-valued martingales.  
		\item In the definition of $M_{r+1}$, the condition $f(1) \in \delta(M_r)$ is purely for technical convenience: It helps to avoid the problem of extending $\delta^{-1} : \delta(M_r) \to M_r$ to a map that is defined on the whole space $\prob(M_r)$, (cf.\ Remark~\ref{rem:Intensity_Delta_Inv}) in Lemma~\ref{lem:R} and applications of it.
		
	For all $r \in \N \cup \{ \infty\}$, the space $M_r$ can be seen as a closed subspace of the space $\mathsf{M}_r$ introduced in the introduction.\footnote{$M_1$ is a closed subset of $\mathsf{M}_1$. For $r >1$, one can easily construct canonical topological embeddings with closed range. To that end, consider the inclusion map $\iota_1 : M_1 \to \mathsf{M}_1$ and define inductively $\iota_{r+1} :=  \Psi(\prob(\iota_r)) \circ j_{r+1}$, where $j_{r+1} : D(\prob(M_r)) \to \mathsf{M}_{r+1}$ is the inclusion. For $r = \infty$,  set $\iota_\infty := \pi \circ (\iota_1, \iota_2, \dots ) \circ \pi^{-1} : M_\infty \to \mathsf{M}_\infty$, where $\pi$ is the map introduced in Lemma~\ref{lem:product_map}.} By definition, $\pp^r(\fp X)$ has paths in  $M_r$, so $\law(\pp^r(\fp X)) \in \prob(M_r)$, but we can also regard it as element of $\prob(\mathsf{M}_r)$ that is supported on $M_r$. 
		\item\label{it:afterDefPpC} In the case of $N$ discrete time steps, only the predictions processes of rank up to $N-1$ are relevant. More precisely, \cite[Lemma 4.7 and Lemma 4.10]{BaBePa21} imply that for every $k > N-1$ there exists a continuous map $F_k : M^{(N)}_{N-1} \to M^{(N)}_k$ such that $\pp^k(\fp X) = F(\pp^{N-1}(\fp X))$ for all $\fp X \in \FR_N$. What matters is that $\pp^k(\fp X)$ can be computed using only the random variable $\pp^{N-1}(\fp X)$, without relying on the filtration $(\F_t^\fp X)_{t = 1}^N$. This means that all the necessary information is already present in $\pp^{N-1}(\fp X)$ and higher-rank prediction processes do not offer any extra insights. In particular, in discrete time $\pp^\infty(\fp X)$ is redundant, hence we omitted the definition.
	\end{enumerate}
\end{remark}
Next, we prove properties of the prediction processes. For simplicity we state the results only in the continuous time setting, but the respective results are also true in discrete time. A first important observation is that the prediction processes of higher rank contain more information: That $\pp^r(\fp X)$ terminates at $\pp^{r-1}(\fp X)$ means precisely that
$$
\pp^{r-1}(\fp X) = (\delta^{-1} \circ e_1)(\pp^{r}(\fp X)),
$$
recalling that $e_1$ denotes the map that evaluates paths at time 1 and $\delta^{-1}(\delta_x)=x$. Iterating this procedure yields
\begin{lemma}\label{lem:R}
For all $0 \le k \le r \le \infty$, there is a continuous function $R^{r,k} : M_r \to M_k$ such that we have for all $\fp X \in \FR$
$$
R^{r,k}(\pp^r(\fp X)) = \pp^k(\fp X).
$$
For all $\ell \le k \le r < \infty$, we have $R^{r,\ell} = R^{k,\ell} \circ R^{r,k}$. 
\end{lemma}

\begin{proof}
For $r \in \N$, let $R^{r,r-1} : M_r \to M_{r-1} := \delta^{-1} \circ e_1$. For $k<r$ we set 
$$
R^{r,k} : M_r \to M_k := R^{k+1,k} \circ R^{k+2,k+1} \circ  \dots \circ R^{r-1,r-2}  \circ R^{r,r-1}.
$$
For $k \in \N$ let $\pr_k := \prod_{j=0}^\infty \prob(M_j(S)) \to \prob(M_k(S))$ denote the projection. We set 
$$
R^{\infty,k} := \Psi(\pr_{k-1}) : M_\infty(S) \to M_k(S).
$$
Moreover, we set $R^{\infty,0} := R^{1,0} \circ R^{\infty,1}$; and for all $r \in \N \cup \{ \infty \}$ we adopt the convention $R^{r,r}=\id_{M_r}$. It is straight forward to check that $R^{r,k}(\pp^r(\fp X)) = \pp^k(\fp X)$.
\end{proof}
 
Lemma~\ref{lem:R} implies that one can obtain the whole path of $\pp^r(\fp X)$ from the whole path of $\pp^{r+1}(\fp X)$. There is also an adapted variant of this, i.e.\ from knowing $\pp^{r+1}(\fp X)$ up to time $t$, one can obtain $\pp^r(\fp X)$ up to time $t$. This implies that the filtrations generated by the prediction processes are ordered: 
\begin{lemma}
Let $\fp X \in \FR$. For every $r \in \N \cup  \{ \infty \}$ let $(\G^r_t)_{t \in [0,1]}$ be the right-continuous completion of the filtration generated by $\pp^r(\fp X)$, i.e.\ 
 $\G^r_t := \bigcap_{\epsilon>0} \sigma^{\P^{\fp X}}( \pp^r_s(\fp X) : s \le t + \epsilon)$. Then we have for every $t \in [0,1]$,
\[
\G_t^1 \subset \G_t^2 \subseteq \dots \subset \G_t^\infty \subset \F_t^{\fp X}.
\]
Moreover, we have $\G^\infty_t = \sigma( \G^r_t : r \in \N)$.  

If in addition, $\fp X \in \FP$, we set $\G^0_t := \bigcap_{\epsilon>0} \sigma^{\P^{\fp X}}(X_s : s \le t + \epsilon)$ and we have $\G^0_t \subset \G^1_t$ as well.
\end{lemma}
\begin{proof}
For $t \in [0,1]$, let $r_{[0,t]}$ be the map which restricts a
\cadlag\ path defined on $[0,1]$ to $[0,t]$. Let $r \in \N$ (or $r \in \N_0$ if $\fp X \in \FP$). As $\pp^r(\fp X)$ is adapted to $(\F_t^{\fp X})_{t \in [0,1]}$, we find for $t \in [0,1]$,
\[
{r_{[0,t]}}_\# \pp_t^{r+1}(\fp X) = \law( \pp^r_{[0,t]}(\fp X) | \F^{\fp X}_t) = \delta_{ \pp^r_{[0,t]}(\fp X)}. 
\]
Hence, $\pp^r_t(\fp X) = (\delta^{-1} \circ \prob(r_{[0,t]})) (\pp^{r+1}_t(\fp X))$. Therefore, the filtration generated by $\pp^r(\fp X)$ is smaller than the filtration generated by $\pp^{r+1}(\fp X)$ and this inclusion carries over to the right-continuous augmentation.

For the second claim observe that, as $\pp^\infty_t(\fp X) = (\pp^r_t(\fp X))_{r \in \N}$, sets of type 
\begin{align}\label{eq:prf:ppmarkov}
\{  \pp^{r_1}_{t_1}(\fp X) \in A_1 ,\dots,  \pp^{r_k}_{t_k}(\fp X) \in A_k   \},
\end{align}
where $k \in \N, r_1,\dots r_k \in \N, t_1,\dots t_k \in [0,1]$ and $A_i \subset \prob(M_{r_i-1})$ generate $\sigma(\pp_t^{\infty}(\fp X))$.  As the set given in \eqref{eq:prf:ppmarkov} is $\sigma(\pp^r(\fp X))$-measurable for $r = \max\{r_1,\dots,r_k\}$,  $\bigcup_{r \in \N} \sigma(\pp^r_t(\fp X))$ generates $\sigma(\pp^\infty_t(\fp X))$.  This relation carries over to the right-contiguous augmentation. 
\end{proof}

The following property of the prediction process will play an  important role throughout the paper.

\begin{corollary}
Let $\fp X \in \FR$, $t \in [0,1]$ and $r \in \N_0$. Then we have 
\begin{align}\label{eq:ppCondIndep}
\pp^{r+1}_t(\fp X)  =  \law( \pp^r(\fp X) | \G ) 
\end{align}
for any $\sigma$-algebra $\G$ that satisfies $\sigma(\pp^{r+1}_t(\fp X)) \subset \G \subset \F_t^{\fp X}$. Moreover, we have 
\begin{align}\label{eq:ppMarkov}
 \law( \pp^\infty(\fp X) | \F_t^{\fp X}) =  \law( \pp^\infty(\fp X) | \G ) 
\end{align}
for any $\sigma$-algebra $\G$ that satisfies $\sigma(\pp^{\infty}_t(\fp X)) \subset \G \subset \F_t^{\fp X}$.
\end{corollary}

\begin{proof}
The first claim is an immediate consequence of Lemma~\ref{lem:condLawSelfAware} applied to $\pp^r(\fp X)$ and $\G$.

In order to prove the second claim, let $\G$ be a $\sigma$-algebra satisfying $\sigma(\pp^{\infty}_t(\fp X)) \subset \G \subset \F_t^{\fp X}$. Due to the already proved first claim,   $\pp^r(\fp X)$ is independent of $\F_t^{\fp X}$ given  $\G$, for all $r \in \N$. Equation \eqref{eq:ppMarkov} follows because  $\bigcup_{r \in \N} \sigma(\pp_t^r(\fp X))$ is a generator of $\sigma(\pp_t^{\infty}(\fp X))$. 
\end{proof}

\begin{corollary}
Let $\fp X \in \mathcal{FR}$ and set $\widehat{\fp X} :=( \Omega^{\fp X}, \F^{\fp X}, \P^{\fp X}, (\G_t)_{t\in[0,1]}, X  )$, where $(\G_t)_{t\in[0,1]}$ is the right-continuous augmentation\footnote{We will see in Proposition~\ref{prop:contPt} below that the filtration of $\pp^\infty(\fp X)$ is already right-continuous.} of the filtration generated by $\pp^\infty(\fp X)$. Then $\fp X \approx_\infty \widehat{\fp X}$.  
\end{corollary}
\begin{proof}
We show by induction on $r \in \N$ that $\pp^r(\fp X) = \pp^r(\widehat{\fp X})$ a.s.
For $n=0$, we have $\pp^0(\fp X) = X =\pp^0(\widehat{\fp X})$. Assuming  $\pp^r(\fp X) = \pp^r(\widehat{\fp X})$ a.s., \eqref{eq:ppCondIndep} implies that for all $t \in [0,1]$  a.s.
\[
\pp^{r+1}_t(\fp X) = \law(\pp^r(\fp X) | \F_t) =  \law(\pp^r(\fp X) | \G_t) = \law(\pp^r( \widehat{\fp X})| \G_t) = \pp^{r+1}_t(\widehat{\fp X}).
\]
As the paths of prediction processes are \cadlag, we conclude $\pp^{r+1}_t(\fp X) = \pp^{r+1}_t(\widehat{\fp X})$ a.s.
\end{proof}

Next, we want to make precise that the procedure of iterating prediction processes naturally ends with $\pp^\infty(\fp X)$, i.e.\ that $\pp^1(\bpp^\infty(\fp X))$ does not contain more information about $\fp X$ than $\pp^\infty(\fp X)$. To that end, denote 
$$
C : = \left\{  (\mu^r)_{ r \in \N_0 } \in \prod_{j=0}^\infty \prob(M_j(S))  : R^{r,k}_\# \mu^r = \mu^k \text{ for all } k \le r \right\}.
$$
Note that Lemma~\ref{lem:R} implies that $\pp^\infty(\fp X)$ takes values in $C$ for every $\fp X \in \FR$ and $t \in [0,1]$.  

\begin{theorem}\label{thm:ppFeller} 
For every $\fp X \in \FR$, the process $\pp^\infty(\fp X)$ is Feller. More specifically, there is a homeomorphism  $F : C \to \prob(M_\infty(S))$ such that for every $\fp X \in \FR(S)$ and every $t \in [0,1]$ we have a.s.
\begin{align}\label{eq:F} 
\law(\pp^\infty(\fp X) |\F_t) =\pp^1_t(\bpp^\infty(\fp X)) = F(\pp^\infty_t(\fp X)).
\end{align}
\end{theorem}
The operation $\fp X \mapsto \pp^\infty(\fp X)$ can be interpreted as a canonical way to lift filtered processes to Feller processes. We emphasize that the map $F$ is independent of the process $\fp X$ and the time $t \in [0,1]$. Therefore, \eqref{eq:F} implies that the procedure of iterating prediction processes naturally terminates with $\pp^\infty(\fp X)$. Its prediction process $\pp^1(\pp^\infty(\fp X))$ can be obtained from $\pp^\infty(\fp X)$ itself (without knowing the filtration $(\F_t^{\fp X})_{t \in [0,1]}$) using the map $F$. 

\begin{proof}[Proof of Theorem~\ref{thm:ppFeller}] The map 
	$$
	\pi : \prod_{j=1}^\infty  M_j(S) \to M_\infty(S) : (z_j)_{j \in \N} \mapsto   (  t \mapsto (z_j(t))_{j \in \N}  )
	$$
	is a homeomorphism, cf.  Lemma~\ref{lem:product_map}. We have for all $t \in [0,1]$
	\begin{align}\label{eq:prf:G1}
	\prob(\pi^{-1})(\pp^1_t(\bpp^\infty(\fp X))) = \prob(\pi^{-1})( \law(\pp^\infty(\fp X) |\F_t^{\fp X} )) = \law( (\pp^r(\fp X))_{r \in \N} | \F_t^{\fp X} ).
	\end{align}	
	On the other hand, we have by definition
	\begin{align}\label{eq:prf:G2}
	\pp^\infty_t(\fp X) = (\pp^r_t(\fp X))_{ r \in \N} = ( \law(\pp^r(\fp X)  |\F_t^{\fp X}  )  )_{r \in \N_0}.
	\end{align}
	On the first glance it may appear that \eqref{eq:prf:G2} contains less information than \eqref{eq:prf:G1} because the sequence $( \law(\pp^r(\fp X)  |\F_t^{\fp X}  )(\omega)  )_{r \in \N}$ is the collection of marginal distributions of  $\law( (\pp^r(\fp X))_{r \in \N} | \F_t^{\fp X} )(\omega)$.
	
	However, this is not the case because $\pp^k(\fp X)$ is a function of $\pp^r(\fp X)$ for $k<r$, so we can derive $\law( \pp^1(\fp X),\dots,  \pp^r(\fp X)|\F_t^{\fp X} )$ from $\law(\pp^r(\fp X)|\F_t^{\fp X} )$. To formalize this, write $\pr_r : \prod_{k=1}^\infty \prob (M_k) \to \prob(M_r) $ for the projection and set 
	$$
	\psi_r := \prob(R^{r,1},\dots, R^{r,r}) \circ
    \pr_r :
    \prod_{k=0}^\infty \prob(M_k)
 \to \prob \Big( \prod_{k=1}^r M_k \Big).
	$$ 
	Indeed, $\psi_r$ is continuous and satisfies
	$$
	\psi_r(\pp^\infty_t(\fp X)) = \law( \pp^1(\fp X),\dots, \pp^r(\fp X) | \F_t^{\fp X} ).
	$$
	The next step is to ``glue'' them together to get  $\law( (\pp^r(\fp X))_{r \in \N} | \F_t^{\fp X} )$. To that end, denote by
	$$
	G : \left\{ (\mu^r)_{r \in \N} \in \prod_{r = 1}^\infty \prob \Big(\prod_{k = 1}^r M_k \Big) : {\pr_{M_1 \times \cdots \times M_r}}_\# \mu_{r + 1} = \mu_r \text{ for all }r \in \N \right\} \to \prob\left( \prod_{k=1}^\infty M_k \right) 
	$$
	the map which maps a consistent family $(\mu^k)_k$ to the unique measure $\mu$ on the infinite product space that satisfies $(\pr_1,\dots,\pr_k)_\#\mu=\mu^k$ for every $k \in \N$. 
 Consider the map
$$
\widetilde{F} := G \circ (\psi_r)_{r \in \N}:  C \to \prob\left( \prod_{k=1}^\infty M_k \right).
$$
Lemma~\ref{lem:LusinProdComp}(a) implies that $\widetilde{F}$ is continuous because 
$\prob(\pr_1,\dots,\pr_k) \circ \widetilde F = \psi_k$ is continuous for all $k \in \N$. Moreover, $\widetilde{F}$ satisfies
$$
\widetilde{F}(( \law(\pp^r(\fp X)  |\F_t^{\fp X}  )  )_{r \in \N_0})  = \law( (\pp^r(\fp X))_{r \in \N} | \F_t^{\fp X} ).
$$
It is straightforward that its inverse $
{\big(\widetilde F\big)}^{-1}(\mu) = (R^{\infty, r}_\# \mu)_{r\in \N_0}
$
is also continuous. We conclude by  setting $F := \pi \circ \widetilde{F}$. 
\end{proof}




\subsection{Continuity points of filtered random variables}
Recall from Definition~\ref{def:contPt} that the set of continuity points of a process $X$ is defined as the set of continuity points of the map $t \mapsto \law(X_t)$. We extend this definition to filtered random variables in the following way:

\begin{definition}\label{def:contPtFP}
The set of continuity points of $\fp X \in \FR$ is defined as $\cont(\fp X) := \cont(\pp^\infty(\fp X))$.
\end{definition}

A direct consequence of Lemma~\ref{lem:TBig} is 
\begin{corollary}\label{cor:ContBig}
Let $\fp X \in \FR$. Then $\cont(\fp X)$ is co-countable. Hence, it is dense and we have $\lambda( \cont(\fp X)  \cup \{ 1 \} )=1$.
\end{corollary}

\begin{lemma}\label{lem:MVMconv}
Let $X$ be an $S$-valued random variable on $(\Omega,\F,\P)$ and $(\F_t)_{t \in [0,1]}$ be a filtration\footnote{Note that we do not assume right continuity of the filtration in this lemma.} on $(\Omega,\F,\P)$. 
\begin{enumerate}[label = (\alph*)]
	\item If $(t_n)_n$ is strictly increasing to $t$, we have  
	$
	\law( X |\F_{t_n} ) \to \law(X|\F_{t-})
	$ a.s.
	\item If $(t_n)_n$ is strictly decreasing to $t$, we have  
	$
\law( X |\F_{t_n} ) \to \law(X|\F_{t+}) 
	$ a.s.
\end{enumerate}
\end{lemma}
\begin{proof}
We only prove the first statement, since the proof of the second statement is a straight forward modification, replacing martingale convergence by backward martingale convergence. Fix a sequence $(t_n)_n$ strictly increasing to $t$. Since $I( \law(X|\F_s) ) = \law(X)$ for all $s \in [0,1]$,  Lemma~\ref{lem:MVM_tight} applied with $T=\{ t_n : n \in \N \}$ implies that for all $m \in \N$, there is a compact set $\mathcal{K}_m \subset \prob(S)$ and $\Omega_m \subset \Omega$ with $\P(\Omega_m) \ge 1-1/m$ such that for all $\omega \in \Omega_m$ and $n \in \N$
$$
\law( X | \F_{t_n} ) (\omega) \in \mathcal{K}_m.
$$
Let $\{f_k : k \in \N\}$ be a convergence determining sequence for $S$ that is closed under multiplication. Then $\{f_k^\ast : k \in \N\}$ is convergence determining in $\prob(S)$. Fix a version of $\law(X|\F_{t-})$.  As $f_k^\ast(\law(X|\F_{t_n}))$ is a version of $\E[f_k(X)|\F_{t_n}]$ for all $n\in \N$ and $f_k^\ast(\law(X|\F_{t-}))$ is a version of $\E[f_k(X)|\F_{t-}]$, the martingale convergence theorem implies that there there is a $\P$-full set $\Omega'$ such that for all $\omega \in \Omega'$ and all $k \in \N$ we have 
\begin{align}\label{eq:prf:martConv}
\lim_{n \to \infty} f_k^\ast(\law(X|\F_{t_n})(\omega)) = f_k^\ast( \law(X|\F_{t-})(\omega) ).
\end{align}
The set $\Omega'' := \Omega' \cap ( \bigcup_m \Omega_m  )$ is again $\P$-full. Let $\omega \in \Omega''$. Then there is $m \in \N$ such that $\law( X | \F_{t_n} ) (\omega) \in \mathcal{K}_m$ for all $n\in \N$. Hence, the sequence $(\law( X | \F_{t_n} ) (\omega))_n$ has a limit point $\mu \in \prob(S)$. By \eqref{eq:prf:martConv} any limit point $\mu$ satisfies $f_k^\ast(\mu) = f_k^\ast( \law(X|\F_{t-})(\omega) )$ for all $k \in \N$, hence $\mu= \law(X|\F_{t-})(\omega)$. This shows that  $\law( X |\F_{t_n} ) \to \law(X|\F_{t-}) $ on $\Omega''$.
\end{proof}

\begin{lemma}\label{lem:continuityfiltration}
	Let $X$ be an $S$-valued \cadlag{} process and let $\G_t:=\sigma^\P(X_s: s \leq t)$. Then for $t \in [0,1]$ the following are equivalent:
	\begin{enumerate}[label = (\roman*)]
		\item \label{it:lem.continuityfiltration.filtration} $(\G_s)_{s \in [0,1]}$ is left-continuous (right-continuous) at $t$;
		\item \label{it:lem.continuityfiltration.prob} $ \law( X | \G_{t_n}) \to \law(X|\G_t)$ a.s.\ whenever $t_n \to t$ and $(t_n)_n$ is increasing (decreasing); 
		\item \label{it:lem.continuityfiltration.law} $s \mapsto \law( \law(X|\G_s) )$ is left-continuous (right-continuous) at $t$.
	\end{enumerate}
\end{lemma}

\begin{proof}
		As the proof of the corresponding claims of left- resp.\ right-continuity are simple modifications of each other, we only show the assertion for left-continuity.
	
%

	\ref{it:lem.continuityfiltration.filtration}$\implies$\ref{it:lem.continuityfiltration.prob}: Let $(t_n)_n$ be a sequence that increases to $t$. By Lemma~\ref{lem:continuityfiltration} we have  $\law(X|\G_{t_n}) \to \law(X|\G_{t-}) = \law(X|\G_{t})$ a.s.
	
	\ref{it:lem.continuityfiltration.prob}$\implies$\ref{it:lem.continuityfiltration.law}: Is trivial because almost sure convergence implies convergence in law. 
	
	\ref{it:lem.continuityfiltration.law}$\implies$\ref{it:lem.continuityfiltration.filtration}:
	The pair $(\law(X|\G_{t-}  ), \law(X|\G_t  )  )$ is a two step martingale which has the same initial and terminal distribution, according to \ref{it:lem.continuityfiltration.law}.  
	Therefore, we find
	\begin{equation}\label{eq:mtgequal}
		\law(X \, | \, \mathcal{G}_{t-}  )= \law(X | \G_t  ) \quad \mbox{a.s.}
	\end{equation}
	In order to check \ref{it:lem.continuityfiltration.filtration}, we need to show that $\G_t \subset \G_{t-}$. Since cylindrical sets of type
	\begin{align}\label{eq:prf:martcyl}
		\{ X_{t_1} \in A_1, \dots, X_{t_n} \in A_n  \},
	\end{align}
	where $n \in \N$,  $t_1,\dots, t_n \le t$ and $A_1, \dots, A_n \subset S$ Borel are generator of $\G_t$, so it suffices to show that all sets of type  \eqref{eq:prf:martcyl} are contained in $\G_{t-}$. Indeed, \eqref{eq:mtgequal} implies that we have a.s.
	$$
	\textbf{1}_{ 	\{ X_{t_1} \in A_1, \dots, X_{t_n} \in A_n  \} } = \P[  X_{t_1} \in A_1, \dots, X_{t_n} \in A_n |\G_t  ] = \P[  X_{t_1} \in A_1, \dots, X_{t_n} \in A_n |\G_{t-}  ].
	$$
	As $\G_{t-}$ is complete, we conclude $\{ X_{t_1} \in A_1, \dots, X_{t_n} \in A_n  \} \in \G_{t-}$. 
\end{proof}

\begin{proposition}\label{prop:contPt}
	Let $\fp X \in \FR$ and $\G_s := \sigma^{\P^{\fp X}}(\pp^\infty_r(\fp X) : r \le s)$, $s \in [0,1]$. The filtration $(\G_s)_{s \in [0,1]}$ is right-continuous and for $t \in [0,1]$ the following are equivalent:
	\begin{enumerate}[label = (\roman*)]
		\item The filtration $(\G_s)_{s \in [0,1]}$ is continuous at $t$, that is $\G_t=\G_{t-}$.
		\item The paths of $\pp^\infty(\fp X)$ are a.s.\ continuous at $t$.
		\item $t \in \cont(\fp X)$, that is $s \mapsto \law(\pp^\infty_s(\fp X))$ is continuous at $t$.
		 
	\end{enumerate}
\end{proposition}
\begin{proof}
As the filtration $(\F_t^{\fp X})_{t \in [0,1]}$ is assumed to be right-continuous, the implication \ref{it:lem.continuityfiltration.filtration} $\implies$ \ref{it:lem.continuityfiltration.law} in Lemma~\ref{lem:continuityfiltration} yields that $t \mapsto \law(\law(\pp^\infty(\fp X) |\F_t^{\fp X} ))$ is right-continuous. By \eqref{eq:ppMarkov} we have $\law(\law(\pp^\infty(\fp X) |\F_t^{\fp X} )) = \law(\law(\pp^\infty(\fp X) |\G_t ))$, thus the implication \ref{it:lem.continuityfiltration.law} $\implies$ \ref{it:lem.continuityfiltration.filtration} in Lemma~\ref{lem:continuityfiltration} yields that $(\G_t)_t$ is right-continuous. 

In order to show the claimed equivalence, we apply Lemma~\ref{lem:continuityfiltration} to $\pp^\infty(\fp X)$. Again by imposing \eqref{eq:ppMarkov}, we can conclude that for $t \in [0,1]$ that the following are equivalent:
\begin{enumerate}
	\item[(i)] $\G_t=\G_{t-}$.
	\item[(ii')] We have $\pp^1_{t_n}(\bpp^\infty(\fp X)) \to \pp^1_t(\bpp^\infty(\fp X))$ a.s.\ for every increasing sequence $t_n \to t$.
	\item[(iii')] $s \mapsto \law(\pp^1_s(\pp^\infty(\fp X)))$ is continuous at $t$.	
\end{enumerate}
Note that (ii') is equivalent to a.s.\ continuity of the paths at $t$ because $\pp^1(\bpp^\infty(\fp X))$ has \cadlag\ paths. Theorem~\ref{thm:ppFeller} yields that (ii) is equivalent to (ii')  and that (iii) is equivalent to (iii').
\end{proof}


\section{The space of filtered random variables}
\label{sec:Space}
The aim of this section is to rigorously define the space of filtered processes modulo equivalence, to introduce the Hoover--Keisler topologies $\HK_r$, $r \in \N_0 \cup \{ \infty \}$ on this space, and to prove the main topological properties of this space.

\subsection{Equivalence classes of filtered random variables}

\begin{definition}\label{def:HK}
For $r \in \N_0 \cup \{\infty \}$ we define the equivalence relation $\approx_r$ on the class of filtered random variables $\FR(S)$ by
$$
\fp X \approx_r \fp Y \iff \law (\pp^r(\fp X)) = \law (\pp^r(\fp Y)).
$$
We call the factor space
$$
\uFR(S) := \FR(S) /_{\approx_\infty}
$$
the space of $S$-valued filtered random variables.

For $r \in \N_0 \cup \{\infty \}$ we define the Hoover--Keisler topology of rank $r$, denoted by $\HK_r$, as the initial topology w.r.t.\ the mapping 
\begin{align}\label{eq:defHK}
\uFR(S) \to \prob(M_r) : \fp X \mapsto \law(\pp^{r}(\fp X)).
\end{align}
The Hoover--Keisler topology of rank $r=\infty$ is called the adapted weak topology and we write $\HK$ instead of $\HK_\infty$. 

Moreover, $\law(\pp^\infty(\fp X))$ is called the adapted distribution of $\fp X$.
\end{definition}

\begin{remark}
\begin{enumerate}[label = (\alph*)]\label{rem:FP}
	\item We have $ \fp X\approx_0 \fp Y$ if and only if $\law(X) = \law(Y)$. The equivalence relation $\approx_1$ was already considered by Aldous \cite{Al81}. Two processes are called \emph{synonymous} if their rank 1 prediction processes have the same law.  
	\item\label{it:remHKb}Lemma~\ref{lem:R} implies that for $0 \le k \le r \le \infty$ the equivalence relation $\approx_r$ is finer than the relation $\approx_k$. It is straightforward to translate the example given in \cite[Section~7]{BaBePa21} to the continuous setup.
  { }This yields that $\approx_r$ is \emph{strictly} finer than  $\approx_k$, if $k <r$. In view of Definition~\ref{def:iterated_pp} of the prediction process of order $\infty$, it is clear that $\fp X \approx_\infty \fp Y $ if and only if $\fp X \approx_r \fp Y $ for all $r < \infty$. 
	\item\label{it:remHKc} Note that the map $\uFR  \to \prob(M_\infty) : \fp X \mapsto \law(\pp^\infty(\fp X))$ is an embedding. This follows because the map is injective (its kernel is precisely  the equivalence relation $\approx_\infty$, which is factored out in the definition of $\uFR$) and the topology $\HK$ is defined as the initial topology w.r.t.\ this map.
 As $\prob(M_\infty)$ is separable metrizable and these properties are inherited by subspaces, $(\uFR,\HK)$ is separable metrizable as well. We will later see in Corollary~\ref{cor:FPPolish} that $(\uFR,\HK)$ is a Lusin space\footnote{This statement is not trivial because despite being ``separable metrizable'',
 the property ``Lusin'' is in general not inherited by subspaces.} and that $(\uFR,\HK)$ is Polish, if $S$ is Polish.
	\item The considerations in \ref{it:remHKb} imply that $(\uFR,\HK_r)$ is not Hausdorff for $r < \infty$. However, it is easy to see that is still a separable pseudometrizable space.

\end{enumerate}
\end{remark}
\begin{notation}
	For the rest of the paper $r$ will always denote the rank of the Hoover--Keisler topology $\HK_r$. If we state a result for $\HK_r$ it is always meant for all $r\in \N_0\cup\{\infty\}$, unless we explicitly specify something else. If a result is only true for $r=\infty$ we state it as result for $\HK$ (without any index). 
\end{notation}

Indeed, $\HK$ is the weakest topology that is stronger than all topologies $\HK_r$, $r \in \N$:
\begin{lemma}\label{lem:ppninfty}
Let $(\fp X^n)_n$ be a sequence in $\uFR$ and $\fp X \in \uFR$. Then $\fp X^n \to \fp X$ in $\HK$ if and only if $\fp X^n \to \fp X$ in $\HK_r$ for all $r \in \N$.
\end{lemma}
\begin{proof}
We deduce from Lemma~\ref{lem:product_map} that $\law(\pp^\infty(\fp X^n)) \to \law(  \pp^\infty(\fp X))$ is equivalent to the convergence of $\law((\pp^r(\fp X^n))_{r\in \N})$ to  $\law((\pp^r(\fp X))_{r\in \N})$. Lemma~\ref{lem:LusinProdComp} implies that this is again equivalent to $\law(\pp^1(\fp X^n),\dots,\pp^r(\fp X^n)  ) \to \law(\pp^1(\fp X),\dots,\pp^r(\fp X) )$ for all $r \in \N$. Lemma~\ref{lem:R} provides for $k<n$ the existence of a continuous function $R^{n,k}$ that satisfies $R^{n,k}(\pp^r(\fp Y)) = \pp^{k}(\fp Y)$, hence convergence of  $\law(\pp^1(\fp X^n),\dots,\pp^r(\fp X^n)  )$ to $ \law(\pp^1(\fp X),\dots,\pp^r(\fp X) )$ is equivalent to $\law(\pp^r(\fp X^n)  ) \to \law(\pp^r(\fp X) )$. All in all, we have shown that $\law(\pp^\infty(\fp X^n)) \to \law(  \pp^\infty(\fp X))$ if and only if $\law(\pp^r(\fp X^n)  ) \to \law(\pp^r(\fp X) )$ for all $r \in \N$.
\end{proof}

We introduce a notion of applying a function $f : S_1 \to S_2$ to an $S_1$-valued filtered random variable that will be useful later on.
\begin{definition}\label{def:diamond}
Let $\fp X \in \uFR(S_1)$ and $f : S_1 \to S_2$ be Borel. Then we define $f \diamond \fp X \in \uFR(S')$ as
$$
f \diamond \fp X := (\Omega^{\fp X}, \F^{\fp X}, \P^{\fp X}, (\F_t^{\fp X})_t, f \circ X  ).
$$
\end{definition}

\begin{proposition}\label{prop:diamond}
Let $\fp X, \fp Y \in \FR(S)$ and $f : S \to S'$ be Borel.
\begin{enumerate}[label = (\alph*)]
	\item If $\fp X \approx_r \fp Y$, then $ f \diamond \fp X \approx_r f \diamond \fp Y$. 
	\item The operation  $ \fp X \mapsto f \diamond \fp X$ is well defined from  $\uFR(S_1)$ to  $\uFR(S_2)$.
	\item The operation $ \fp X \mapsto f \diamond \fp X$ inherits the following properties from the function $f : S_1 \to S_2$ to  : injective, surjective, bijective, Borel, continuous, being a topological embedding.
	\item If $f$ is continuous, we have $\cont(f \diamond \fp X) \supseteq \cont(\fp X)$. 
\end{enumerate}
\end{proposition}
\begin{proof}
We assume that $f$ is continuous and discuss at the end of the proof how to relax this assumption. First we prove by induction on $r$ that we have 
\begin{align}\label{eq:DiamondSomething}
\pp^r( f \diamond \fp X) =  (\Psi \circ \prob)^r(f) ( \pp^r(\fp X ) ) \quad \text{a.s.}
\end{align}
This is trivial for $r=0$ because $\pp^0(f \diamond \fp X ) = f(X) = f(\pp^0(\fp X))$. Assume that \eqref{eq:DiamondSomething} is true for $r$, then we have for all $t \in [0,1]$ a.s.
\begin{equation}\label{eq:prf:diamond}
\begin{aligned}
\pp^{r+1}_t(f \diamond \fp X ) &= \law( \pp^r( f \diamond \fp X) |\F_t^{\fp X}  ) =  \law( \, (\Psi \circ \prob)^n(f) ( \pp^r(\fp X ) ) \, |\F_t^{\fp X}  ))\\ & = \prob( (\Psi \circ \prob)^r(f)  ) ( \law(\pp^r(\fp X) | \F_t^{\fp X} )) = \prob( (\Psi \circ \prob)^r(f)  ) (  \pp^{r+1}_t(\fp X) ).
\end{aligned}
\end{equation}
As $f$ is continuous, $(\Psi \circ \prob)^{r+1}(f)$ maps \cadlag{} paths to \cadlag{} paths (cf.\ Proposition~\ref{prop:Psi_lift}), so $(\Psi \circ \prob)^{r+1}(f) ( \pp^{r+1}(\fp X ) )$ is a \cadlag{} process. Using Remark~\ref{rem:TermMVM} we  conclude 
$\pp^{r+1}( f \diamond \fp X) =  (\Psi \circ \prob)^{n+1}(f) ( \pp^{r+1}(\fp X ) )$ from \eqref{eq:prf:diamond}.

This immediately implies (a), (b) and (d). Also (c) follows then from the respective properties of the functors $\Psi$ and $\prob$, cf.\ Remark~\ref{rem:P} and Proposition~\ref{prop:Psi_lift}. 

If $f : S_1  \to S_2$ is not continuous but merely Borel, we may replace the topology on $S_1$ by a stronger Polish topology that generates the same Borel sets and renders the map $f : S_1 \to S_2$ continuous, cf.\ \cite[Theorem 13.11]{Ke95}. Note that (a), (b), and (c) except for the items  ``continuous'' and ``topological embedding'', are statements that do not depend on the topology of $S_1$. Hence it is legitimate to replace the topology for the proof. 
\end{proof}
\subsection{Canonical representatives}
Formally, $\fp X \in \uFR$ is an equivalence class of filtered random variables whose prediction processes of rank $\infty$ have the same law. In this section we construct a canonical representative of these equivalence classes. 
This canonical representative will allow us to characterize the probability measures on $M_r$ that are the distribution of a prediction process of rank $r$. As we will see below, this  is also a crucial ingredient for the proof of the compactness result in the next section.

Let $\fp X \in  \FR$ be given. Whenever $k \le r$ (or $k<r$ if $r=\infty$), the process $\pp^k(\fp X)$ is a measure-valued martingale w.r.t.\ $(\F_t^{\fp X} )_{t \in [0,1]}$, so in particular,  $\pp^k(\fp X)$ is also a measure-valued martingale w.r.t.\ the filtration generated by $\pp^r(\fp X)$. The latter is just a property of the joint law $( \pp^r(\fp X), \pp^k(\fp X) )$. As $\pp^k(\fp X)$ is a function of $\pp^r(\fp X)$ (cf.\  Lemma~\ref{lem:R}), it is in fact just a property of the law of $\pp^r(\fp X)$. It turns out that this property already characterizes the laws of predictions processes among all probabilities on $M_r$. To make this precise, we need to introduce some notation:

\begin{notation}\label{not:Z}
Fix $r \in \N \cup \{ \infty \}$. We denote the canonical process on 

$$
M_r \subset \begin{cases}
	D(\prob(M_{r-1})) & r \in \N \\
	D(\prod_{k=0}^\infty \prob(M_k)) & r = \infty 
\end{cases}
$$	
by $Z^r = (Z^r_t)_{t \in [0,1]}$, i.e.\ $Z^r$ is a \cadlag{} process with values in $\prob(M_{r-1})$ if $r < \infty$, and in $\prod_{k=0}^\infty \prob(M_k)$ if $r= \infty$.  For $k < r$ we define 
$$
Z^k := R^{r,k} : M_r \to M_k \subset D(\prob(M_{k-1})),
$$
so $Z^k = (Z^k_t)_{t \in [0,1]}$ is a \cadlag{} process with values in $\prob(M_{k-1})$. Finally, we define the $S$-valued random variable
\begin{align}\label{eq:canonRepX}
X := Z^0 := \delta^{-1}(Z^1_1).
\end{align}
\end{notation}

\begin{definition}\label{def:consTermin} A probability
$\mu \in \prob(M_r)$ is called consistently terminating martingale law, if for all $1 \le k <r$ it holds under $\mu$ that 
\begin{itemize}
    \item 
$Z^{k+1}$ is a measure-valued martingale w.r.t.\ the filtration generated by $Z^r$ and 
\item $Z^{k+1}$ terminates at $Z^k$.  
\end{itemize}
\end{definition}

\begin{theorem}\label{thm:CharPpLaw}
Let $\mu \in \prob(M_r)$. Then the following are equivalent:
\begin{enumerate}[label = (\roman*)]
	\item The probability  $\mu$ is a consistently terminating martingale law.  \label{it:CharPpLaw1}
	\item There exists some $\fp X \in \FR$ such that $\mu = \law(\pp^r(\fp X))$. \label{it:CharPpLaw2}
\end{enumerate}	
\end{theorem}
We have already established $\ref{it:CharPpLaw2} \implies \ref{it:CharPpLaw1}$ in the consideration at the beginning of this section. The other direction $\ref{it:CharPpLaw1} \implies \ref{it:CharPpLaw2}$ is a consequence of the following proposition:

\begin{proposition}\label{prop:canonRep}
Let $\mu \in \prob(M_r)$ be a consistently terminating martingale law. Define  $\G_t := \sigma^\mu(Z^r_s : s\le t)$,  $\F_t := \G_{t+}, t\in [0,1]$ and 
\begin{align}\label{eq:canonRep}
\fp X^\mu = (M_r, \mathcal{B}_{M_r}, \mu, (\F_t)_{t \in [0,1]}, X),
\end{align}
where $\mathcal{B}_{M_r}$ denotes the (completed) Borel $\sigma$-algebra on $M_r$ and $X$ is defined as in \eqref{eq:canonRepX}.

Then we have $Z^k=\pp^k(\fp X^\mu)$ a.s.\ for all $k \le r$.
\end{proposition}
From now on $\fp X^\mu$ will always denote the filtered random variable defined in \eqref{eq:canonRep}. Before we give the proof of Proposition~\ref{prop:canonRep}, we stress the following immediate consequence of this proposition: 
\begin{corollary}\label{cor:canonRep}
Let $\fp X \in \FR$. Then we have $\fp X \approx_\infty \fp X^{\law(\pp^\infty(\fp X))}$. In particular, every $\HK$-equivalence class contains a representative that is defined on a standard Borel probability space.
\end{corollary}

\begin{proof}[Proof of Proposition~\ref{prop:canonRep}]
We show $Z^k=\pp^k(\fp X)$ a.s.\ by induction on $k \le r$. Indeed, the claim is trivial for $k=0$ because $ \pp^0(\fp X^\mu) = X =Z^0$. 

Assume that the claim is true for $k \le r-1$. As $Z^{k+1}$ is a measure-valued martingale w.r.t.\ the filtration generated by $Z^r$ and terminates at $Z^k$, we have for all $t \in [0,1]$ 
$$
\law(\pp^{k}(\fp X^\mu)|\G_t) =  \law(Z^k|Z^r_s : s \le t) = Z^{k+1}_t \quad \text{a.s.}
$$
Lemma~\ref{lem:MVMconv} yields that for $t <1 $ we have
$$
\pp^{k+1}_t(\fp X^\mu) = \law(\pp^{k}(\fp X^\mu)|\G_{t+}) = \lim_{n \to \infty} \law(\pp^{k}(\fp X^\mu)|\G_{t+1/n}) = \lim_{n \to \infty}  Z^{k+1}_{t+1/n} = Z^{k+1}_t\quad \text{a.s.}, 
$$
where the last equality holds because $Z^{k+1}$ is \cadlag. As $\F_1=\G_1$, we have $\pp^{k+1}_1(\fp X^\mu) =Z^{k+1}_1$ as well. Hence, we conclude $Z^{k+1} = \pp^{k+1}(\fp X^\mu)$ by Remark~\ref{rem:TermMVM}.

In the case $r=\infty$, it remains to check the claim for $k = \infty$. Indeed, we have for all $t \in [0,1]$ 
$$Z^\infty_t = (Z^j_t)_{j \in \N} = (\pp^j_t(\fp X^\mu))_{j \in \N} = \pp^\infty_t(\fp X^\mu) \quad\text{ a.s.}$$
As both sides are \cadlag{} $\pp^{\infty}(\fp X^\mu) = Z^\infty$ a.s.
\end{proof}

\subsection{Compactness}
The main goal of this section is to prove the following compactness result for filtered random variables:
\begin{theorem}\label{thm:Compactness}
	A set $\A \subset \uFR(S)$ is relatively compact w.r.t.\ $\HK_r$ if and only if  $\{  \law(X) : \fp X \in A \}$ is relatively compact in $\prob(S)$.
\end{theorem}
Using Prohorov's theorem (see Section~\ref{sec:ProbaLusin}) Theorem~\ref{thm:Compactness} implies: If $ \{ \law( X ) : \fp X \in \A \}$ is tight, then $\A\subset \uFR$ is relatively compact. This is an equivalence when the space $S$ is Polish.

The proof of Theorem~\ref{thm:Compactness} consists of two steps:
\begin{itemize}
	\item Preservation of compactness: We will show that $\{ \law(\pp^r(\fp X)) : \fp X \in \A \}$ is relatively compact in $\prob(M_r(S))$ if and only if $\{  \law(X) : \fp X \in \A \}$ is relatively compact in $\prob(S)$.
	\item Closedness of laws of prediction processes: Using Theorem~\ref{thm:CharPpLaw} we show that these form a closed subset of $\prob(M_r(S))$.
\end{itemize}

\begin{lemma}\label{lem:ppClosed}
The set of consistently terminating martingale measures is closed in $\prob(M_r)$.
\end{lemma}
\begin{proof}
We stick to Notation~\ref{not:Z}. Let $(\mu^n)_n$ be a sequence of consistently terminating martingale laws that converge to $\mu \in \prob(M_r)$. Fix $k<r$, let $g \in C_b(M_{k})$ and write $ \phi := g^\ast \circ R^{r,k+1} : M_r \to \R$. Then $\phi(Z^r_t)_{t \in [0,1]} = Z^{k+1}[g]$ is a martingale under $\mu^n$ w.r.t.\ the filtration generated by $Z^r$. Hence, we are precisely in the regime of Lemma~\ref{lem:MartFClosed} and conclude that the same is true under $\mu$, i.e.\ under $\mu$, $ Z^{k+1}[g]$ is a martingale w.r.t.\ the filtration generated by $Z^r$ for every $g \in C_b(M_{k})$. Hence $Z^{k+1}$ is a measure-valued martingale w.r.t.\ the filtration generated by $Z^r$.

The property that $Z^{k+1}$ terminates at $Z^k$ is exactly $Z^{k+1}_1 = \delta(Z^k)$ a.s. This is preserved by the limit $\mu^n \to \mu$ because evaluation at time $1$ and the mapping $\delta$ are both continuous.
\end{proof}

\begin{lemma}\label{lem:ppnRelComp}
Let $\A \subset \uFR(S)$. If $\{  \law(X) : \fp X \in \A \}$ is relatively compact in $\prob(S)$, then $\{ \law(\pp^r(\fp X)) : \fp X \in \A \}$ is relatively compact in $\prob(M_r(S))$. 
\end{lemma}
\begin{proof}
	
Let $\A \subset \uFR(S)$ such that $\{  \law(X) : \fp X \in \A \}$ is relatively compact in $\prob(S)$.

We first prove the result for $r \in \N$ by induction.  The claim is trivial for $r=0$. Assume it is true for $r$, i.e.\ $\{ \law(\pp^r(\fp X)) : \fp X \in \A \}$ is relatively compact in $\prob(M_r(S))$. As $\pp^{r+1}(\fp X)$ terminates at $\pp^{r}(\fp X)$, Corollary~\ref{cor:MVMcompTermin} implies that $\{ \law(\pp^{r+1}(\fp X)) : \fp X \in \A \}$ is relatively compact in $\mathcal{M}_0(M_r)$. Recall that $\mathcal{M}_0(M_r)$ is the set of laws of terminating $\prob(M_r)$-valued \cadlag{} martingales. We have $\mathcal{M}_0(M_r) \subset \prob(M_{r+1})$ and  relative compactness in $\mathcal{M}_0(M_r)$ implies relative compactness in $\prob(M_{r+1})$.

Let $r=\infty$. Due to our previous considerations, we already know that $\law(\pp^r(\fp X)) : \fp X \in \A \}$ is relatively compact in $\prob(M_r(S))$ for all $r \in \N$. By Lemma~\ref{lem:LusinProdComp}(b) we can conclude relative compactness of $\{\law( (\pp^r(\fp X))_{r\in \N} ) : \fp X \in \A \}$ in $\prob(\prod_{j \in \N} M_j)$, and by Lemma~\ref{lem:product_map} this is equivalent to relative compactness of $\{ \law(\pp^\infty(\fp X)) : \fp X \in \A \}$ in $\prob(M_\infty)$.
\end{proof}
%

\begin{proof}[Proof of Theorem~\ref{thm:Compactness}]
		Let $\A \subset \uFR$ be such that $\{ \law(\fp X) : \fp X \in \A \} $ is relatively compact. 	 Lemma~\ref{lem:ppnRelComp} yields the relative compactness of $\{ \law(\pp^r(\fp X)) : \fp X \in \A \}$ in $\prob(M_r)$. Let $(\fp X^n)_n$ be a sequence in $\A$ and denote $\mu^n:= \law(\pp^r(\fp X^n))$. By relative compactness, there are a subsequence $(\mu^{n_k})_k$  and  $\mu \in \prob(M_r)$ such that $\mu^{n_k} \to \mu$. By Lemma~\ref{lem:ppClosed}, $\mu$ is a consistently terminating martingale law, hence by Theorem~\ref{thm:CharPpLaw}, there is  $\fp X \in \uFR$ such that $\mu= \law(\pp^{r}(\fp X))$. Therefore, $\fp X^{n_k} \to \fp X$ w.r.t.\ $\HK_r$. We conclude that  $\A$ is relatively compact in $(\uFR,\HK_r)$.
	
Conversely, assume that $\A$ is relatively compact in $(\uFR,\HK_r)$ for some $r \in \N \cup \{ \infty \}$. As the mapping $(\uFR,\HK) \ni \fp X \mapsto \law(X) \in \prob(S)$ is continuous (see Lemma~\ref{lem:R}), $\{  \law(X) : \fp X \in \A \}$ is relatively compact as well. 
\end{proof}

\begin{remark}
Theorem~\ref{thm:Compactness} was already established by Hoover in \cite[Theorem 4.3]{Ho91}. Depending on the reader's taste, our proof might be considered more rigorous or just more pedantic than the one given by Hoover.\footnote{While the work of Hoover appears extraordinarily innovative, some results are represented in an unusual framework and their proofs are not always entirely rigorous. E.g.\ the amalgamation theorem \cite[Theorem~3.2]{Ho92} is not true and a counterexample to the uniqueness result for SDEs in \cite[Theorem~3.3]{Ho87} can be found in  \cite{En02}.}
 Specifically, we aim to rigorously define the 
prediction processes as well as  the spaces where they take their values, we take some care to work out not only the case $r=1$, but also the inductive step $r \mapsto r+1$, as well as the limit step $r=\infty$.

\end{remark}

We close this section with two useful corollaries of the compactness result.

\begin{corollary}\label{cor:FPPolish}
If $S$ is Polish (Lusin), then $(\uFR(S),\HK)$ is Polish (Lusin) as well.
\end{corollary}
\begin{proof}
By Remark~\ref{rem:FP}\ref{it:remHKc}, $\uFR(S)$ is homeomorphic to a subset of the Lusin space $\prob(M_\infty(S))$. By Theorem~\ref{thm:CharPpLaw} and Lemma~\ref{lem:ppClosed}, this subset is closed. As closed subspaces of Lusin spaces are Lusin, we conclude that $\uFR(S)$ is Lusin.

If $S$ is assumed to be Polish, 
$$
\uFR(S) \to \prob(S)  : \fp X \mapsto \law(X) 
$$
is a continuous (cf.\ Lemma~\ref{lem:R}) map into a Polish space. Theorem~\ref{thm:Compactness} states precisely that preimages of compact sets are compact, so Lemma~\ref{lem:Lusin2Polish} implies that $(\uFR(S),\HK)$ is Polish.
\end{proof}

%
%

\subsection{Fixed points}
Tychonoff's fixed point theorem  states that if $V$ is a locally convex topological vector space and $K \subset V$ is a compact convex set, then every continuous function $f : K \to K$ has a fixed point. Theorem~\ref{thm:CharPpLaw} implies that $\uFP$ can be seen as a convex subset of $\prob( M_\infty)$, which is itself a convex subset of the  vector space of signed measures on $M_\infty$. Hence, we obtain the following fixed point result on the space of filtered processes $\uFP$

\begin{theorem}\label{thm:fixedpoint}
Let $f : \uFP \to \uFP$ be continuous and assume that $\{ \law(X) : \fp X \in f(\uFP) \}$ is tight. Then $f$ has a fixed point. 
\end{theorem}
\begin{proof}

Let $\widehat{M}_\infty$ be a compactification of $M_\infty$ (e.g.\ embed the Lusin space $M_\infty$ into $[0,1]^\N$ and take the closure therein). We denote the space of signed measures on $\widehat{M}_\infty$ as $\mathscr{M}(\widehat{M}_\infty)$ and equip it with the weak-$\ast$-topology as dual of the Banach space $C(\widehat{M}_\infty)$. 

In order to apply Tychonoff's fixed point theorem, we embed $\uFP$ into the locally convex vector space $\mathscr{M}(\widehat{M}_\infty)$ using the map
\[
\phi : \uFP \to \mathscr{M}(\widehat{  M }_\infty) : \fp X \mapsto \law(\pp^\infty(\fp X)). 
\]   
Note that $\phi$ is indeed an embedding because the adapted weak topology is the initial topology w.r.t.\ the map $\uFP \to \prob(  M _\infty) : \fp X \mapsto \law(\pp^\infty(\fp X))$, the weak topology on $\prob( M_\infty)$ is the subspace topology of the weak topology on $\prob(\widehat{ M}_\infty)$, cf.\ Remark~\ref{rem:P}, which is in turn the subspace topology of the weak-$\ast$-topology on $\mathscr M( \widehat{ M}_\infty )$. Further note that $\prob(M_\infty)$ is a convex subset of $\mathscr{M}(\widehat{  M }_\infty)$.

 By Theorem~\ref{thm:CharPpLaw} we can characterize the measures $\mu \in \prob(M_\infty)$ that are in the range of $\phi$ by equations that are linear in $\mu$, namely 
 \begin{align*}
    \int g_n(Z^r_t)f_1(Z_{s_1}^\infty) \cdots f_k(Z^\infty_{s_k}) d\mu &= \int g_n(Z^r_t)f_1(Z_{s_1}^\infty) \cdots f_k(Z^\infty_{s_k}) d\mu,\\
    \int g_r(Z^r_1) d\mu &= \int g_r(\delta_{Z^{r-1}}) d\mu, 
\end{align*}
for all $k,r \in \N$, all $s_1 \le \dots \le s_k \le s \le t$,  all $f_1, \dots, f_k \in C_b(\prod_{j=0}^\infty \prob(M_j))$ and  all $g_r \in C_b(\prob(M_{r-1}))$. Hence, $\phi(\uFP)$ is convex.

The assumption that $\{ \law(X) : \fp X \in f(\uFP) \}$ is tight, implies that there is a sequence $(K_n)_n$ of compact subsets of $S$ such that $\P^{\fp X}(X \in K_n) \ge 1-1/n$ for all $\fp X \in f(\uFP)$. The set  
\[
\A := \{ \fp X \in \uFP : \P^{\fp X}(X \in K_n) \ge 1-1/n \text{ for all } n \in \N \}
\]
is relatively compact in $\uFP$ by Theorem~\ref{thm:Compactness} and closed because $K_n \subset S$ is compact and hence closed. Therefore, $\A$ is compact and, as $\phi$ is continuous, $\phi(\A)$ is compact as well. Moreover, $\phi(\A)$ is convex because $\phi(\uFP)$ is convex and a measure $\mu \in \phi(\uFP)$ is in $\phi(\A)$ if and only if $\mu((Z^0)^{-1}(K_n)) \ge 1-1/n$ for all $n \in \N$ and the latter conditions are linear in $\mu$. 

Next, we apply Tychonoff's fixed point theorem to the map 
\[
F := \phi \circ f \circ \phi^{-1}  :  \phi(\A) \to \phi(\A).  
\]
Note that $F$ is continuous, because $\phi$ being an embedding implies that $\phi^{-1}: \phi(\uFP) \to \uFP$ is continuous. Further note that $F(\phi(\A)) \subset \phi(\A)$ because the range of $f$ is contained in $\A$. Therefore, Tychonoff's fixed point theorem yields the existence of $\mu \in \phi(\A)$ such that $F(\mu) = \mu$. Clearly, $\fp X := \phi^{-1}(\mu)$ is a fixed point of $f$.
\end{proof}

\section{Discretization}
\label{sec:Discr}
A natural way to discretize  filtered random variables in time is to  restrict the filtration $(\F_t)_{t \in [0,1]}$ to a finite set $T \subset [0,1]$. Throughout this section we fix a finite set of times $T = \{t_1,\ldots,t_N\}$ with $0 \le t_1 < \ldots < t_N = 1$. We consider the discretization operator
\begin{align}\label{eq:def.discretisation.FP}
	D_T : \FR(S) \to \FR_N(S) : \fp X \mapsto \fp X_T := (\Omega^{\fp X},\F^{\fp X}, \P^{\fp X}, (\F_{t_i}^{\fp X})_{i=1}^N , X).
\end{align}
The main result of this section is that this operator can be well defined on the factor spaces $\uFR$ and $\uFR_N$ and that we can  give conditions for continuity of $D_T$ at a point $\fp X$:

\begin{theorem}\label{thm:discr}
	The discretisation operator 
	$$
	D_T : (\uFR,\HK_r) \to (\uFR_N,\HK_r) :  \fp X \mapsto \fp X_T
	$$
	is well-defined and Borel measurable. Moreover, it is continuous at all $\fp X \in \uFR$ satisfying $T \subset \cont(\fp X)$, i.e.\ if $\fp X^n \to \fp X$ in $(\uFR,\HK_r)$ and if $\fp X$ satisfies $T \subset \cont(\fp X)$, then $\fp X^n_T \to \fp X_T$ in $(\uFR_N,\HK_r)$.
\end{theorem}

Recall that Corollary~\ref{cor:ContBig} ensures that $\cont(\fp X)$ is co-countable for all $\fp X \in \FR$, so Theorem~\ref{thm:discr} states that $D_T$ is  continuous at $\fp X$ for ``typical'' sets $T$. The following example shows that we cannot expect continuity of the discretization operator without assumptions on $T$:

\begin{example}
	Let $(\Omega,\F,\P) =( \{0,1\}, 2^{\{0,1\}}, \frac{1}{2}(\delta_0+\delta_1) )$ and let $X=\id$. For $s \in [0,1)$ define the filtration $(\F^s_t)_{t \in [0,1]}$ such that $\F_t^s$ is trivial for $t<s$ and $\F_t^s := \F$ for $t \ge s$. Then for each $s \in [0,1]$ we have a filtered random variable
	$$
	\fp X^s := (\Omega, \F,\P, (\F^s_t)_{t \in [0,1]}, X). 
	$$
	It is easy to see that $\cont(\fp X^s)=[0,1]\setminus \{s\}$ and $\fp X^{s+1/n} \to \fp X^s$ in $(\uFR,\HK)$. However, we have $\pp^1_s(\fp X^{s+1/n})(\omega)=\P$ for all $n \in \N$ and $\omega \in \Omega$, whereas $\pp^1_s(\fp X^{s})(\omega)=\delta_\omega$ for $\omega \in \Omega$. Hence, when setting $T= \{s,1\}$ we do not have $\fp X_T^{s+1/n} \to \fp X_T^s$ in $(\uFR_2,\HK)$. 
\end{example}

The rest of this section is dedicated to the proof of Theorem~\ref{thm:discr}. The idea is to recursively construct  mappings 
$$
F^r_T : M_r \to M_r^{(N)}
$$
that map the rank $r$ prediction processes of continuous-time filtered random variable $\fp X$ to the rank $r$ prediction process of its discretization $\fp X_T$ and investigate their continuity properties. We first set $F_T^0 := \id : S \to S$. Assuming that $F^r_T : M_r \to M_r^{(N)}$ is already defined, we define $F^{r+1}_T$ by
 \begin{align*}
	F_T^{r+1} : M_{r+1} \to M_{r+1}^{(N)} : z \mapsto ( {F^r_T}_\# z(t_1),\dots, {F^r_T}_\# z(t_N) ).
\end{align*}
This is indeed welldefined: First we evaluate the path $z \in M_{r+1} \subset D(P(M_r))$ at the times $T = \{ t_1,\dots, t_N \}$, that is a map from $M_{r+1} \subset D(P(M_r))$ to $\prob(M_r)^N$. Then at every time $t_i$ we apply the pushforward w.r.t.\ $F_T^r$, that is the map
 $\prob(F_T^r) : \prob(M_r) \to \prob(M_r^{(N)})$. Hence, the target space of $F_T^{r+1}$ is $\prob(M_r^{(N)})^N = M_{r+1}^{(N)}$. 

We first show that $F^r_T$ indeed maps the rank $r$ prediction process of $\fp X \in \mathcal{FP}$ to the rank $r$ prediction process of its discretization: 
\begin{lemma} \label{lem:FTn_vs_XT}
	For all $\fp X \in \mathcal{FR}$ it holds
	$$
	F_T^r(\pp^r(\fp X)) = \pp^r(\fp X_T).
	$$
\end{lemma}
\begin{proof}
	We show this by induction on the rank $r$. For $r=0$ it is trivial. 	Assume the result is true for $r$. Then we find for every $\fp X \in \FR$
	\begin{align*}
		F^{r+1}_T(\pp^{r+1}(\fp X)) &= ( {F_T^r}_\# \law(\pp^r(\fp X)| \F_{t_1}^{\fp X} )), \dots, {F_T^r}_\# \law(\pp^r(\fp X)| \F_{t_N}^{\fp X} ))  \\
		&=( \law(F_T^r(\pp^r(\fp X))| \F_{t_1}^{\fp X} )), \dots,  \law(F_T^r(\pp^r(\fp X))| \F_{t_N}^{\fp X} ))   \\
		&=(\law(\pp^r(\fp X_T)|\F^{\fp X}_{t_1}),\dots , \law(\pp^r(\fp X_T)|\F^{\fp X}_{t_N}))\\
		&=\pp^{r+1}(\fp X_T). \qedhere
	\end{align*}
\end{proof}
We give some calculation rules for convergence in law which are useful to prove the main theorem of this section:
\begin{lemma}[{\cite[Theorem~4.1]{Ed19}}]\label{lem:contGluing}
	Let $Y^n_i, Y_i$ be $S_i$-valued random variables, $i \in \{1,2,3\}$,  and let $f : S_2 \to S_3$ be Borel. Then we have 
\begin{align*}
\begin{rcases}
	(Y_1^n,Y_2^n) \inlaw (Y_1,Y_2) \\
	(Y_2^n,Y^n_3) \inlaw (Y_2,Y_3)  \\
	Y_3 =f(Y_2) \text{ a.s.} 
\end{rcases}
\implies (Y_1^n,Y_2^n,Y_3^n) \inlaw (Y_1,Y_2,Y_3) .
\end{align*}
\end{lemma}




\begin{corollary}\label{cor:contGluing}
Let $Y^n,Y$ be $S$-valued random variables and for $i \in \{ 1,\dots, k \}$ let  $f_i : S \to S_i$ be Borel. Then we have
\begin{align*}
	\begin{rcases}
		(Y^n,f_1(Y^n)) \inlaw (Y,f_1(Y)) \\
	\qquad \qquad \qquad \,\,\,	\vdots  \\
		(Y^n,f_k(Y^n)) \inlaw (Y,f_k(Y))
	\end{rcases}
	\implies (Y^n,f_1(Y^n),\dots, f_k(Y^n)) \inlaw (Y,f_1(Y),\dots, f_k(Y)).
\end{align*}
\end{corollary}
\begin{proof}
This follows from Lemma~\ref{lem:contGluing} by induction on $k$.
\end{proof}

\begin{lemma}\label{lem:ConcOnGraph}
Let $Z, Z'$ be $\prob(S_1 \times S_2)$-valued random variables and $ f : S_1 \to S_2$ be Borel. If $I(\law(Z))(\graph(f))=I(\law(Z'))(\graph(f))=1$ and  $\law({\pr_1}_\# Z ) = \law( {\pr_1}_\# Z' )$, then $\law(Z)=\law(Z')$.
\end{lemma}

\begin{proof}
    Observe that
    \[
        \mathbb E[Z(\graph(f))] = \int p(\graph(f)) \law(Z)(dp) = I(\law(Z))(\graph(f)) = 1.
    \]
    Since $Z(\graph(f)) \in [0,1]$, we deduce $Z(\graph(f)) = 1$ a.s., which entails $Z = \prob(\id,f)(\pr_{1\#}Z)$ a.s.
    We obtain
    \begin{equation}
        \label{eq:alternative}
        \law(Z) = \law(\prob(\id,f)({\pr_1}_\# Z)) = \prob(\id,f)_\#\law({\pr_{1}}_\#Z).
    \end{equation}
    Analogously, we find that \eqref{eq:alternative} holds when $Z$ is replaced by $Z'$.
    Consequently, using that $\law({\pr_{1}}_\# Z) = \law({\pr_{1}}_\#Z')$, it follows from \eqref{eq:alternative} that $\law(Z) = \law(Z')$.
\end{proof}

\begin{proposition}\label{prop:D}
Let $X^n, X$ be $S$-valued random variables on probability spaces $(\Omega^n,\F^n,\P^n)$ and $(\Omega,\F,\P)$ resp. and let $\G^n \subset \F^n$ and $\G \subset \F$ resp. be sub-$\sigma$-algebras. Then we have
\begin{align*}
\begin{rcases}
	(Y^n,f(Y^n)) \inlaw (Y,f(Y)) \\
	\law( Y^n|\G^n) \inlaw \law(Y|\G)  \\ 
\end{rcases}
\implies \law( Y^n, f(Y^n)|\G^n)) \inlaw \law(Y,f(Y)|\G)).
\end{align*}
\end{proposition}
\begin{proof}
Since $I(\law(  \law( Y^n, f(Y^n)|\G^n)   ) ) = \law( Y^n,f(Y^n)  )$ converges for $n \to \infty$, Proposition~\ref{prop:intensity_tightness} implies that  $(\law(  \law( Y^n, f(Y^n)|\G^n)   ) )_{n \in \N}$ is relatively compact. Hence, it suffices to show that $\law( \law(Y,f(Y)|\G) )$ is the only possible limit point of this sequence.

To that end, write $Z := \law(Y,f(Y)|\G)$ and observe that $Z$ has the properties
\begin{align}\label{eq:prf:propD}
\begin{aligned}
&I(\law(Z))(\graph(f)) = \law(Y,f(Y))(\graph(f))=1,  \\
&{\pr_1}_\# Z = {\pr_1}_\# \law(Y,f(Y)|\G) = \law(Y|\G).
\end{aligned}
\end{align}
Consider a subsequence $(\law( Y^{n_j}, f(Y^{n_j})|\G^{n_j}))_{j}$ such that
$$
\law( Y^{n_j}, f(Y^{n_j})|\G^{n_j}) \inlaw Z'.
$$
Our aim is to check that $Z'$ also has the properties listed in \eqref{eq:prf:propD} and to invoke Lemma~\ref{lem:ConcOnGraph} to conclude that $\law(Z) = \law(Z')$. Indeed, we have 
$$
I(\law(Z')) = \lim_{j \to \infty} I(\law( \law( Y^{n_j}, f(Y^{n_j})|\G^{n_j}) )) =\lim_{j \to \infty} \law(Y^{n_j},f(Y^{n_j})) = \law(Y,f(Y))
$$
and hence $I(\law(Z'))(\graph(f))= 1$.  Moreover, we have in distribution
$$
{\pr_1}_\# Z' = \lim_{j \to \infty} {\pr_1}_\#\law( Y^{n_j}, f(Y^{n_j})|\G^{n_j}) = \lim_{j \to \infty} {\pr_1}_\#\law( Y^{n_j}|\G^{n_j}) = \law(Y|\G).
$$
Hence, we can use Lemma~\ref{lem:ConcOnGraph} to conclude that $\law(Z)=\law(Z')$.
\end{proof}

\begin{lemma}\label{lem:CondMarginal}
	Let $Y_1^n,Y_1$ be $S_1$-valued random variables and $Y_2^n, Y_2$ be $S_2$-valued random variables on probability spaces $(\Omega^n,\F^n,\P^n)$ and $(\Omega,\F,\P)$ resp. and let $\G^n \subset \F^n$ and $\G \subset \F$ resp. be sub-$\sigma$-algebras. Then we have
	\begin{align*}
		\law( Y_1^n,Y_2^n |\G^n ) \inlaw \law(Y_1,Y_2|\G) \implies ( \law(Y^n_1|\G^n),  \law(Y^n_2|\G^n) ) \inlaw ( \law(Y_1|\G),  \law(Y_2|\G) ). 
	\end{align*}
\end{lemma}
\begin{proof}
	The map 
	$$
	\phi : \prob(S_1 \times S_2 ) \to \prob(S_1) \times \prob(S_2) : \mu \mapsto ({\pr_1}_\#\mu,{\pr_2}_\#\mu)
	$$
	is continuous and we have for all $m \in \N$ (and resp.\ for $Y_1,Y_2$)   
	\begin{align*}
		\phi_\# \law( Y_1^n,Y_2^n |\G^n  ) &= ( \law(Y^n_1|\G^n),  \law(Y^n_2|\G^n) )\quad \text{a.s.} \qedhere 
	\end{align*}
\end{proof}

We are now ready to prove the main result of this section.
\begin{proof}[Proof of Theorem~\ref{thm:discr}]
Let $\fp X^n \to \fp X$ in $\HK_r$ and $T \subset \cont(\fp X)$. We prove by induction on $k \le r$
\begin{align}\label{eq:prf:dInd}
	(\pp^{k}(\fp X^n), F_T^k(\pp^k(\fp X^n)) ) \inlaw  (\pp^{k}(\fp X), F_T^k (\pp^k(\fp X))) .
\end{align}	
Indeed, this is trivial for $k=0$. Assume it is true for some $k<n$. Since $k+1 \le n$ we have $\pp^{k+1}(\fp X^n) \inlaw \pp^{k+1}(\fp X)$. Since $ T \subset \cont(\fp X) \subset  \cont(\pp^{k+1}(\fp X))$.  Corollary~\ref{cor:MVMMargConv2} yields 
\begin{align}\label{eq:prf:d1}
(\pp^{k+1}(\fp X^n), \pp^{k+1}_T(\fp X^n)) \inlaw (\pp^{k+1}(X), \pp^{k+1}_T(\fp X)).
\end{align}
Fix $t \in T$. Since $\pp^{k+1}_t( \fp X^n ) = \law(\pp^{k} (\fp X^n)  |\F_t^{\fp X^n} )$, equation \eqref{eq:prf:d1} implies  
\begin{align}\label{eq:prf:d2}
\law( \pp^{k}(\fp X^n) |\F_t^{\fp X^n }  ) \inlaw \law( \pp^k (\fp X) | \F_t^{\fp X} ).
\end{align}
So, \eqref{eq:prf:dInd} and \eqref{eq:prf:d2} put us precisely in the setting of Proposition~\ref{prop:D} and we conclude that
\begin{align*}
	\law( \pp^k(\fp X^n), F_T^k( \pp^k(\fp X^n) )  | \F_t^{\fp X^n} ) \inlaw \law( \pp^k(\fp X), F_T^k( \pp^k(\fp X) )  | \F_t^{\fp X} ).
\end{align*}
Using Lemma~\ref{lem:CondMarginal},  the fact that $\law(\pp^k(\fp X^n |\F_t^{\fp X^n} )) = \pp_t^{k+1}(\fp X^n)$ and that $\law( F_T^k( \pp^k(\fp X^n)| \F_t^{\fp X} ) ) = {F_T^k}_\#\pp_t^{k+1}(\fp X^n)$, 
we obtain 
\begin{align}\label{eq:prf:d3}
	(\pp^{k+1}_t(\fp X^n), {F_T^k}_\#\pp_t^{k+1}(\fp X^n) ) \inlaw 	(\pp^{k+1}_t(\fp X),  {F_T^k}_\#\pp_t^{k+1}(\fp X)).
\end{align}
We apply Lemma~\ref{lem:contGluing} with $Y_1^n = \pp^{k+1}(\fp X^n)$, $Y_2^n = \pp^{k+1}_t(\fp X^n)$, $Y_3^n = {F_T^k}_\#\pp_t^{k+1}(\fp X^n)$ and $f = \prob(F_T^k)$  to \eqref{eq:prf:d1} and \eqref{eq:prf:d3}  and  conclude
\begin{align}\label{eq:prf:d4}
	(\pp^{k+1}(\fp X^n), {F_T^k}_\#\pp_t^{k+1}(\fp X^n) ) \inlaw 	(\pp^{k+1}(\fp X),  {F_T^k}_\#\pp_t^{k+1}(\fp X)).
\end{align}
Noting that $F_T^{k+1}( \pp^{k+1}(\fp X^n) ) = ({F_T^k}_\#\pp_t^{k+1}(\fp X^n))_{t \in T}$, we derive from \eqref{eq:prf:d4} for all $t \in T$ using Corollary~\ref{cor:contGluing} with $f_t = \prob(F_T^k) \circ e_{t}$ for all $t\in T$ 
\begin{align}\label{eq:prf:dIndplus1}
	(\pp^{k+1}(\fp X^n), F_T^{k+1}(\pp^{k+1}(\fp X^n)) ) \inlaw  (\pp^{k+1}(\fp X), F_T^{k+1} (\pp^{k+1}(\fp X))) .
\end{align}	
This is exactly \eqref{eq:prf:dInd} for $k+1$, hence the induction is complete.

From \eqref{eq:prf:dInd} with $k=r$ and $F_T^r(\pp^r(\fp Y)) = \pp^r(\fp Y_T)$ (cf. Lemma~\ref{lem:FTn_vs_XT}) we obtain $\pp^r(\fp X^n) \inlaw \pp^r(\fp X)$. This proves Theorem~\ref{thm:discr} for $r \in \N$. Note that there is no extra work to be done in the case $r=\infty$ as on $\uFR_N$  all topologies $\HK_r$, where  $r>N-1$, coincide with $\HK_{N-1}$, cf. Remark~\ref{rem:afterDefPp}\ref{it:afterDefPpC}.
\end{proof}

\section{Adapted functions}
\label{sec:AF}

We recall the concept of adapted functions from \cite{HoKe84,BaBePa21}. Loosely speaking, adapted functions are operations that take a filtered random variables as an argument and return a random variable defined on the underlying probability space of the respective filtered random variable. The rank of an adapted function is a measure for its complexity. More precisely,  it is the maximum number of nested conditional expectations appearing when evaluating this adapted function, as we will see below.
\begin{definition}\label{def:AF}
	An adapted function is a function\footnote{Strictly speaking, an adapted function is an element of a term algebra: (AF1) defines for every $f \in C_b(S)$ an operation symbol of arity 0. (AF2) defines for every $g \in C_b(\R^n)$ an operation symbol of arity $n$. (AF3) defines for every $t \in [0,1]$ an operation symbol of arity 1. Definition~\ref{def:AFvalue} below states how to interpret these terms. Note that two different terms can lead to the same map $\fp X \mapsto f[\fp X]$. The rank is a property of the term, not of the function $\fp X \mapsto f[\fp X]$, as a given function $\fp X \mapsto f[\fp X]$ can also be represented by `unnecessarily complicated' terms.
	 } that can be built using the following rules:
	\begin{enumerate}
		\item[(AF1)] Every $f \in C_b(S)$ is an adapted function and $\rank(f)=0$.
		\item[(AF2)] If $f_1, \dots, f_n$ are adapted functions and $g \in C_b(\R^n)$, then $g(f_1,\dots,f_n)$ is an adapted function with $\rank(g(f_1,\dots,f_n)) = \max \{ \rank(f_i) \colon i = 1,\ldots,n\}$.
		\item[(AF3)] If $f$ is an adapted function and $t \in [0,1]$, then $(f|t)$ is an adapted function and we set $\rank((f|t)) = \rank(f)+1$.
	\end{enumerate}
	We write $\AF$ for the set of adapted functions and $\AF_r$ for the set of adapted functions of rank at most $r$.
\end{definition}
Note that $\AF_0=C_b(S)$. The value of an adapted function at a filtered random variable is  again  random variable (on the same probability space). Formally it is defined through the following induction: 
\begin{definition}\label{def:AFvalue}
	Let ${\fp X} \in \FR$.
	\begin{enumerate}
		\item[(AF1)] If $f \in \AF_0$, then its value at ${\fp X}$ is $f[{\fp X}] := f(X)$.
		\item[(AF2)] The value of $g(f_1,\dots,f_n)$ at ${\fp X}$ is $g(f_1,\dots,f_n)[{\fp X}] := g(f_1[{\fp X}], \dots, f_n[{\fp X}])$.
		\item[(AF3)] The value of $(f|t)$ at ${\fp X}$ is $(f|t)[{\fp X}] = \E[f[{\fp X}]|\F_t^{\fp X}]$.
	\end{enumerate}
\end{definition}

One can define equivalence relations on $\FR$ using adapted functions.

\begin{definition}
	Let ${\fp X}, {\fp Y} \in \FR$. We write ${\fp X} \sim_r {\fp Y}$, if $\E[f[{\fp X}]] = \E[f[{\fp Y}]]$ for every $f \in \AF_r$. Moreover, we write ${\fp X} \sim_\infty {\fp Y}$, if $\E[f[{\fp X}]] = \E[f[{\fp Y}]]$ for every $f \in \AF$.
\end{definition}

Our aim is to show that the equivalence relation $\sim_r$ defined in terms of adapted functions coincides with the equivalence relation $\approx_r$ defined in terms of prediction processes, cf. Definition~\ref{def:HK}.
The next lemma implies that the relation $\approx_r$ is stronger than the relation $\sim_r$, that is, ${\fp X} \approx_r {\fp Y}$ implies ${\fp X} \sim_r {\fp Y}$.

\begin{definition}
	
	Let $f \in \AF$ and $F : M_{\rank(f)} \to \R$ be bounded and Borel. We say that $F$ represents $f$, if we have for all $\fp X \in \mathcal{FR}$
	\begin{equation*}
		f[\fp X] = F(\pp^{\rank(f)}(\fp X)). 
	\end{equation*}
\end{definition}

\begin{lemma}\label{lem:af_as_func_of_pp}
	All adapted functions are representable.
\end{lemma}

\begin{proof}
	All adapted functions of rank 0 are represented by themselves. Hence, it suffices to show that the set of representable adapted functions is closed under the operations (AF2) and (AF3).
	
	Let $g \in C_b(\R^n)$ and let $f_1, \dots, f_n$ be adapted functions with $\rank(f_i)=r_i$ that are represented by $F_i : M_{r_i} \to \R$. Set $r := \max_i r_i$. Recalling the mappings $R^{r, k}, r\geq k$ from Lemma \ref{lem:R} we   define the bounded Borel function
	\[
	G \colon M_r \to \R \colon z \mapsto g( F_1( R^{r,r_1}(z) ), \dots,  F_n( R^{r,r_n}(z) )  ).
	\] 
	Then we have for every ${\fp X} \in \mathcal{FR}$
	\begin{align*}
		G(\pp^r({\fp X})) &= g( F_1( R^{r,r_1}(\pp^r({\fp X})) ), \dots,  F_n( R^{r,r_n}(\pp^r({\fp X})) )  ) \\
		&= g(F_1(\pp^{r_1}({\fp X})), \dots, F_n(\pp^{r_n}({\fp X})))\\
		&= g(f_1[{\fp X}],\dots,f_n[{\fp X}]) = g(f_1,\dots,f_n)[{\fp X}].
	\end{align*}
	Next, let $f \in \AF$ be represented by $F$ and $t \in [0,1]$. Then
	\[
	G \colon M_{r+1} \to \R \colon z \mapsto  \int F(w)  \, e_t(z)(dw)
	\]
	represents $(f|t)$ because we have
	\begin{align*}
		G(\pp^{r+1}({\fp X})) 
		&= \int F(w) \law(\pp^r({\fp X})|\F_t^{\fp X})(dw) 
		\\
		&= \E[F(\pp^r({\fp X})) |\F_t^{\fp X} ] = \E[f[{\fp X}]|\F_t^{\fp X}] = (f|t)[{\fp X}],
	\end{align*}
	for every $\fp X \in \FR$.
\end{proof}

We introduce a notion to keep track  which time evaluations were used in a given adapted function:
\begin{definition}
	For every $f \in \AF$ we define $T(f) \subset [0,1]$ inductively via
	\begin{itemize}
		\item[(AF1)] $T(f) = \emptyset$ for all $f \in \AF_0$;
		\item[(AF2)] $T(g(f_1,\dots,f_n)) =\bigcup_{i=1}^n T(f_i)$;
		\item[(AF3)] $T((f|t)) = T(f) \cup \{t\}$.
	\end{itemize}
\end{definition}
Note that $T(f)$ is finite for all $f \in \AF$. 

\begin{remark}\label{rem:af_as_func_of_pp}
Let  $f \in \AF_r$ and write $N:= |T(f)|$. By minor modifications in the proof of Lemma~\ref{lem:af_as_func_of_pp} one can show that for every $f \in \AF$ there exists a continuous bounded function $F: M_r^{(N)} \to \R$ such that 
\begin{equation*}
	f[\fp X] = F(\pp^{r}(\fp X_{T(f)})). 
\end{equation*}
Indeed, the construction is the same as in the proof above and the continuity of $F$ is due to the continuity of evaluations at time points in the discrete time setting.
\end{remark}

Next, we  prove that ${\fp X} \sim_r {\fp Y}$ implies ${\fp X} \approx_r {\fp Y}$. To that end, we inductively construct sufficiently large families $\mathcal{A}_r$ consisting of bounded Borel functions on $M_r$ which represent adapted functions.

\begin{lemma}\label{lem:sep_alg}
	Let $T \subset [0,1]$ be a dense set that contains 1. For every $r \in \N$, there is a family $\mathcal{A}_r$ consisting of bounded measurable functions $M_r \to \R$ such that 
	\begin{enumerate}[label = (\roman*)]
		\item\label{it:lem.sep_alg.1} $\A_r$ is countable, closed under multiplication, and separates points in $\prob(M_r)$;
		\item\label{it:lem.sep_alg.2} Every $F \in \A_r$ represents some $f \in \AF_r$ satisfying $T(f) \subset T$.
	\end{enumerate} 
\end{lemma}
\begin{proof}
	For each $n \in \N$, there exists a point separating family $\mathcal{C}_n \subset C_b(\R^n; [0,1])$ which is countable and closed under multiplication by Lemma~\ref{lem:convergence.determining}. W.l.o.g.\ we can require that $\mathcal{C}_1$ contains the function $$\psi_0: \R \to [0,1]: \, x \mapsto (x \vee 0) \wedge 1.$$
	Let us prove the result by induction on the rank $r$. For $r = 0$, we can choose $\A_0 = \mathcal{C}_1$.       
	
	Assume that we have defined an algebra $\A_r$ with the desired properties. We define $\A_{r+1}$ as the collection of all functions of type
	\begin{align}\label{eq:algebra}
		z \mapsto \psi\left( \int F_1(w) \,  e_{t_1}(z)(dw), \dots, \int F_n(w)  \, e_{t_n}(z)(dw)    \right), 
	\end{align}
	where $n \in \N$, $\psi \in \mathcal{C}_n$, $F_1,\dots, F_n \in \A_r$ and $t_1,\dots, t_n \in T$. Clearly, $\A_{r+1}$ is countable and closed under multiplication. 
	
	We show that $\A_{r+1}$ separates points on $M_{r+1}$. Let $z \neq z' \in M_{r+1} \subseteq D(\prob(M_r))$. Then there is some $t \in T$ such that $e_t(z) \neq e_t(z')$. Since $\A_r$ separates points in $\prob(M_r)$, there is some $F \in \A_r$ such that
	\begin{align*}
		\psi_0\left( \int F(w) \, e_t(z)(dw)\right) &= \int F(w) \, e_t(z)(dw) \\
		& \neq \int F(w) \, e_t(z')(dw)=\psi_0\left( \int F(w) \, e_t(z')(dw)\right),
	\end{align*}
	and we have $\psi_0\left( \int F(w) \, e_t(\cdot )(dw)\right) \in \A_{r+1}$. By Lemma~\ref{lem:P(S)_ptsep_convdet}(b),  $\A_{r+1}$ separates points on $\prob(M_{r+1})$.
	
	Let $G \in \A_{r+1}$ be composed of  $\psi \in C_b(\R^n)$, $F_1,\dots, F_n \in \A_r$ and $t_1,\dots, t_n \in T$ as in \eqref{eq:algebra} and let $f_1, \dots, f_n \in \AF_r$ such that $F_i$ represents $f_i$. Then $G$ represents
	$
	g:= \psi( (f_1|t_1),\dots, (f_n|t_n)  ) \in \AF_{r+1}.
	$   
\end{proof}

\begin{proposition}\label{prop:simapprox}
	For $\fp X, \fp Y \in \FR$ the following are equivalent:
	\begin{enumerate}[label = (\roman*)]
		\item ${\fp X} \approx_r {\fp Y}$, that is $\law(\pp^r(\fp X)) = \law( \pp^r (\fp Y))$; 
		\item ${\fp X} \sim_r {\fp Y}$, that is $\E [f[\fp X]] = \E[f[ \fp Y ]]$ for every $f \in \AF_r$; 
		\item There exists a dense set $T \subset [0,1]$ that contains 1 such that $\E [f[\fp X]] = \E[f[ \fp Y ]]$ for all $f \in \AF_r$ satisfying $T(f) \subset T$.
	\end{enumerate}
\end{proposition}
\begin{proof}
	It suffices to show the claim for $r \in \N$, cf. Remark~\ref{rem:FP}\ref{it:remHKb}. Fix $r \in \N$.
	
	$(i) \implies (ii)$: follows from Lemma~\ref{lem:af_as_func_of_pp}. 
	
	$(ii) \implies (iii)$: is trivial.
	
	$(iii) \implies (i)$: Assume that ${\fp X} \not \approx_r {\fp Y}$, that is $\law(\pp^r({\fp X})) \neq \law(\pp^r({\fp Y}))$.
	By Lemma~\ref{lem:sep_alg} there is a point separating family $\A_r$ of bounded Borel functions $M_r \to \R$ such that  every $F \in \A_r$ represents some $f \in \AF_r$ satisfying $T(f) \subset T$.	As $\A_r$ is point separating, there is some $F \in \A_r$ such that $\E [F(\pp^r({\fp X}))] \neq \E [F(\pp^r({\fp Y}))]$. 
	 By the definition of $\A_r$, there is some $f\in \AF_r$ satisfying $T(f) \subset T$ that is represented by $F$. For this $f$ we have $\E[f[{\fp X}]] \neq \E[f[{\fp Y}]]$.
\end{proof}

The techniques developed in this section allow us to show that the Hoover--Keisler topology defined via the prediction processes is equal to the topology defined via adapted functions. This result can be seen as an extension of Theorem~\ref{thm:weakFDD}, in fact in the case $r=1$ it is just a slight reformulation of that theorem.
\begin{theorem}
Let $(\fp X^n)_n$ be a sequence in $\uFR$ and $\fp X \in \uFR$. Then the following are equivalent:
\begin{enumerate}[label = (\roman*)]
	\item $\fp X^n \to \fp X$ in $\HK_r$.
	\item For every $f \in \AF_r$ satisfying  $T(f) \subset \cont(\fp X) \cup \{ 1 \}$ we have  $\E[ f[\fp X^n] ] \to \E[ f[\fp X]]$. 
	\item There exists a dense set $T \subset [0,1]$ that contains 1 such that for all $f \in \AF_r$ satisfying  $T(f) \subset T$ it holds  $\E[ f[\fp X^n] ] \to \E[ f[\fp X]]$.
\end{enumerate}
\end{theorem}

\begin{proof}
	$(i) \implies (ii)$: Let  $\fp X^n \to \fp X$ in $\HK$. As $T(f) \subset \cont(\fp X)$ Theorem~\ref{thm:discr} yields the convergence discrete time filtered random variables $\fp X^n_{T(f)}$ to $\fp X_{T(f)}$ in $\HK$, i.e.\ $\law(\pp^r( \fp X^n_{T(f)})) \to \law(\pp^r( \fp X_{T(f)}))$. By Remark~\ref{rem:af_as_func_of_pp} there is continuous bounded function $F$ such that for all $\fp Y \in \uFR$
	\begin{equation*}
	\E	[f[\fp Y]] = \E[ F(\pp^{r}(\fp Y_{T(f)}))]. 
	\end{equation*}
	Hence, we can conclude $\E [f[\fp X^n] ]\to \E [f[\fp X]]$.
	
	$(ii) \implies (iii)$ is trivial.
	
	$(iii) \implies (i)$: We assume that there is $T \subset [0,1]$ dense, containg 1, such that  $\E[ f[\fp X^n] ] \to \E[ f[\fp X]]$ for all  $f \in \AF_r$ satisfying  $T(f) \subset T$. 	Recall that for $f \in \AF_0$ we have  $T(f) = \emptyset$ and $\E[f[\fp Y]] = \E[f(Y)]$ for all $\fp Y \in \uFR$. So, we have  $\law(X^n) \to \law(X)$, in particular, the sequence $(\law(X^n))_n$ is relatively compact in $\prob(S)$. Theorem~\ref{thm:Compactness} yields that the sequence $(\fp X^n)_n$ is relatively compact in $(\uFR,\HK)$. So, it suffices to show that every limit point of this sequence is equal to $\fp X$.  
	
	Let $\fp Y$ be a limit point and let  $(\fp X^{n_k})_k$ be a subsequence such that $\fp X^{n_k} \to \fp Y$ in $(\uFR,\HK_r)$. Denote $T':= (T \cap \cont(\fp Y)) \cup \{1\}$. For every $f \in \AF$ satisfying $T(f) \subset T'$ we have 
	$$
	\E [f[\fp X]] = \lim_{k \to \infty} \E [f[\fp X^{n_k}]] = \E [f[\fp Y]],
	$$ 
	where the second equality is due to  $\fp X^{n_k} \to \fp Y$ in $\HK_r$ and the already proven implication $(i) \implies (ii)$. By Proposition~\ref{prop:simapprox} we conclude $\fp X = \fp Y$.
\end{proof}


\section{Adapted filtered processes}\label{sec:adapted}
\subsection{Continuous time}
Every continuous-time filtered process can be seen as a $D(S)$-valued filtered random variable, hence $\FP(S) \subset \FR(D(S))$. This inclusion is strict because the definition of a $D(S)$-valued filtered random variable contains no adaptedness constraint. For given $\fp X \in \FR(D(S))$, the process $X$ is adapted to $(\F_t^{\fp X})_{t \in [0,1]}$ if and only if $\law(X_t|\F^{\fp X}_t) = \delta_{X_t}$ for every $t \in [0,1]$. Note that $\law(X_t |\F_t^{\fp X}) = {e_t}_\# \pp^1_t(\fp X)$, where $e_t : D(S) \to S : f \mapsto f(t)$ is the evaluation at time $t$. Hence, $\fp X \in \FP(S)$ if and only if for every $t \in [0,1]$
\begin{equation}\label{eq:adapted}
    {e_t}_\# \pp^1_t(\fp X) = \delta_{X_t}. 
\end{equation}
This shows that being adapted is a property of  the law of the prediction process of rank 1 and in particular a well-defined notion in the factor spaces, i.e.\ we have  $\uFP(S) \subset \uFR(D(S))$. Moreover, using \eqref{eq:adapted} we will show that $\uFP(S)$ is a closed subset of $\uFR(D(S))$. 

Let $\uFPc(S)$ denote the set of $S$-valued filtered processes with continuous paths. Note that $\uFPc(S)$ is $\HK_r$-closed in $\uFP(S)$ for all $r \in \N_0 \cup \{\infty\}$ if $D(S)$ is equipped with the $J_1$-topology. 

\begin{proposition}\label{prop:adaptClosed}
 $\uFP(S)$ is $\HK_1$-closed in $\uFR(D(S))$ if $D([0,1],S)$ is equipped with either the Meyer--Zheng or the $J_1$-topology. Moreover, $\uFPc(S)$ is closed in $\uFR(C(S))$ if $C(S)$ is equipped with the supremum distance.
\end{proposition}

As $\HK_1$ is the weakest topology among the topologies $\HK_r$,  $r \in \N \cup \{ \infty\}$, $\HK_1$-closedness readily implies  closedness in $\HK_r$ for all $r \in \N \cup \{ \infty\}$. In order to prove Proposition~\ref{prop:adaptClosed}, we provide  two auxiliary lemmas: 
\begin{lemma}\label{lem:cond_law_dirac}
Let $X,Y$ be $S$-valued random variables on a probability space $(\Omega,\F,\P)$ and $\G$ a sub-$\sigma$-algebra of $\F$. If $\law(X|\G) = \delta_Y$ a.s.\ then $X=Y$ a.s.
\end{lemma}
\begin{proof}
Let $\{ \phi_j : j \in \N\}$ be a point separating family on $S$. Then we have for all $j \in \N$
\[
\E[\phi_j(X) | \G] = \phi_j(Y) \quad \mbox{and} \quad \E[\phi_j(X)^2|\G] = \phi_j(Y)^2.
\]
Using these equalities and the tower property we get $\E[|\phi_j(X)-\phi_j(Y)|^2] = 0$, so  $\phi_j(X) = \phi_j(Y)$ a.s.\ for all $j \in \N$. As the family $\{\phi_j : j \in \N\}$ is countable and point separating, we conclude that $X = Y$ a.s.\ 
\end{proof}

\begin{lemma}\label{lem:adapt}
Let $\phi : S \to S'$ be continuous and $t \in [0,1]$. Then  
\begin{align}\label{eq:prf:adapt}
\{ \fp X \in \uFP(S) : \phi(X) \textup{ is } \F_t^{\fp X} \textup{-measurable} \}
\end{align}
is  closed in $\uFP(S)$ w.r.t.\ $\HK_1$.
\end{lemma}

\begin{proof}
Note that $\phi(X)$ is $\F_t^{\fp X}$-measurable if and only if $\phi(X)$ is $\F_s^{\fp X}$-measurable for all $s \ge t$. By Lemma~\ref{lem:cond_law_dirac}, the latter is equivalent to $\pp^1_s( \phi \diamond \fp X)$
taking values in $\delta(S')$ for all $s \in [t,1]$.  As the operation $\fp X \mapsto \phi \diamond \fp X$ is continuous (cf. Proposition~\ref{prop:diamond}), the restriction to $[t,1]$ is continuous, and $D([t,1]; \delta(S'))$ is closed in $D([t,1]; \prob(S'))$, the set \eqref{eq:prf:adapt} is $\HK_1$-closed.
\end{proof}

\begin{proof}[Proof of Proposition~\ref{prop:adaptClosed}]
First observe that it suffices to prove the claim for the case of $D(S)$ with the Meyer--Zheng topology. This readily implies the claim when $D(S)$ is equipped with the $J_1$-topology because the $J_1$-topology is stronger than the Meyer--Zheng-topology\footnote{and hence for every $r$, the topology $\HK_r$ w.r.t.\ Meyer--Zheng on the path space $D(S)$ is weaker than $\HK_r$ w.r.t.\ $J_1$ on $D(S)$. To see this, apply Proposition~\ref{prop:diamond} the identity map from $D(S)$ with Meyer--Zheng to $D(S)$ with $J_1$.}. Moreover, the result for $(D(S),J_1)$ implies the result for $C(S)$ because the $C(S)$ is $J_1$-closed subset of $D(S)$ and the subspace topology of $J_1$ on $C(S)$ is precisely the uniform topology. 

In order to  overcome the issue that point evaluation is not continuous w.r.t.\ the Meyer--Zheng topology, we need to construct a family of continuous functions which allows us to characterize adaptedness. 

First note that there is a metric $d$ that induces the topology of $S$ and a set $F \subset C(S)$ which has the following properties: $F$ is convergence determining, countable, closed under multiplication, every $f \in F$ is Lipschitz w.r.t.\ $d$ and satisfies $0 \le f  \le 1$.  To see this, embed $S$ into $[0,1]^\N$, set $d((x_n)_n,(y_n)_n) := \sum_{n \in \N} 2^{-n}|x_n-y_n|$ and let $F$ be the collection of all finite products of projections $[0,1]^\N \to [0,1]$.

For $t \in [0,1]$ and $n \in \N$, let $g_{t,n}(s) = (1- n(t-s)_+)_+$, i.e.\ we set $g_{t,n}(s)=1$ for $s \le t$, $g_{t,n}(s) = 0$ for $s \ge t+1/n$ and interpolate linearly between $t$ and $t+1/n$. For  $f \in F$, $t \in [0,1]$ and $n \in \N$ we define the map
$$
\phi_{f,t,n} : D(S) \to D([0,1]) : h \mapsto (f \circ h) \cdot g_{t,n}.
$$
Note that  $d_1(h_1,h_2) := \int d(h_1(s),h_2(s)) \lambda(ds)$ is a metric for the Meyer--Zheng topology on $D(S)$ and $d_2(h_1,h_2) := \int |h_1(s)-h_2(s)| \lambda(ds)$ is a metric for the Meyer--Zheng topology on $D([0,1])$. A straightforward calculation shows that $\phi_{f,t,n} : (D(S),d_1) \to (D(S),d_2)$ is Lipschitz and hence continuous.

It remains to show that a filtered random variable $\fp X \in \uFR(D(S))$ is adapted if and only if $\phi_{f,t,n}(X)$ is $\F_{t+1/n}^{\fp X}$-measurable for all $t \in [0,1], f\in F$ and $n \in \N$. Then Lemma~\ref{lem:adapt} yields the claim.

Assume that $\fp X$ is adapted. As $g_{t,n}(s)=0$ for $s >t+1/n$, $\phi_{f,t,n}(X)$ only depends on $X|_{[0,t+1/n]}$, so $\phi_{f,t,n}(X)$ is $\F_{t+1/n}^{\fp X}$-measurable. 

Conversely, assume that $\phi_{f,s,n}(X)$ is $\F_{t+1/n}^{\fp X}$-measurable for all $s \in [0,1], f\in F$ and $n \in \N$. 
Fix $t \in [0,1]$. By the right-continuity of the paths of $X$ and as $g_{t,n}(s)>0$ for $s<t+1/n$, the assumption implies that $f(X_t)$ is $\F_{t+1/n}^{\fp X}$-measurable for all $n \in \N$ and all $f \in F$.  By the right-continuity of the filtration, $f(X_t)$ is $\F_{t}^{\fp X}$-measurable for all $f \in F$. As $F$ is convergence determining, the Borel-$\sigma$-algebra on $S$ is the initial $\sigma$-algebra w.r.t.\ $F$. Hence, we conclude that $X_t$  is $\F_{t}^{\fp X}$-measurable.
\end{proof}

The most important application of Proposition~\ref{prop:adaptClosed} is 
to extend the compactness result Theorem~\ref{thm:Compactness} from filtered random variables to filtered processes. 
\begin{theorem}
Let $r \in \N_0 \cup \{ \infty \}$.
\begin{enumerate}[label=(\alph*)]
    \item Let $C(S)$ be equipped with the supremum distance. Then $\A \subset \uFPc(S)$ is relatively compact w.r.t.\ $\HK_r$ if and only if $ \{ \law(X) : \fp X \in \A \}$ is relatively compact in $\prob(C(S))$.
    \item Let $D(S)$ be either equipped with the $J_1$-topology or with the Meyer--Zheng topology. Then $\mathcal{A} \subset \uFP(S)$ is relatively compact w.r.t.\ $\HK_r$ if and only if $ \{ \law(X) : \fp X \in \A \}$ is relatively compact in $\prob(D(S))$.
\end{enumerate}
\end{theorem}
\begin{proof}
This is an immediate consequence of Theorem~\ref{thm:Compactness} and Proposition~\ref{prop:adaptClosed}.
\end{proof}

The strict inclusion $\FP(S) \subset \FR(D(S))$ suggests that the concept of filtered random variables is more general than the concept of filtered processes. However, it turns out that these concepts are equally general in a specific sense. One can ``simulate'' a filtered random variable by a  filtered process that is constant on $[0,1)$ and equal to the given random variable at $t=1$. 

To make this precise, fix $s_0 \in S$ and define $\iota : S \to D(S)$ by $\iota(s)(t) = s_0$ if $t<1$ and $\iota(s)(1)=s$. Note that $\iota : S \to D(S)$ is a topological embedding with closed range if $D(S)$ is equipped with  Meyer--Zheng (or $J_1$) topology. 

\begin{proposition}
For every $r \in \N_0 \cup \{ \infty \} $, the map
\[
\boldsymbol{\iota} : (\uFR(S),\HK_r) \to (\uFP(S),\HK_r) : (\Omega,\F,\P, (\F_t)_{t \in [0,1]}, X) \mapsto (\Omega,\F,\P, (\F_t)_{t \in [0,1]}, \iota(X))
\]
is a topological embedding with closed range.
\end{proposition}
\begin{proof}
Note that $\boldsymbol{\iota}(\fp X)$ is adapted because $\iota(X)_t$ is constant and hence $\F_t^{\fp X}$-measurable for $t<1$ and $\iota(X)_1=X$ is $\F_1^{\fp X}$-measurable. The remaining claims follow by applying Proposition~\ref{prop:diamond} to $\iota$.
\end{proof}

In Section~\ref{sec:Space} we proved the characterisation of adapted distributions in the framework of filtered random variables. We translate this result into the framework of filtered processes as introduced in Sections~\ref{sec:IntroInteratedPP} and~\ref{ssec:adap_dist}. Recall in particular the definition of $\mathsf{M}_\infty$ given in \eqref{eq:Minfty} and Remark~\ref{rem:afterDefPp}(b). 
\begin{proof}[Proof of Theorem~\ref{thm:introAdaptedDistrChar}]
Let $\mu \in \prob(\mathsf{M}_\infty)$ and assume  that there exists a filtered process $\fp X \in \FP$ satisfying $\mu = \law(\pp^\infty(\fp X))$.  By Theorem~\ref{thm:CharPpLaw}, $\mu$ is a martingale measure and $Z^{r}_1 = \delta_{Z^{r-1}}$ under $\mu$ for every $r\in \N$. As $\fp X$ is adapted, we have ${e_t}_\# Z^1_t = \delta_{Z^0_t}$ $\mu$-a.s. for all $t \in [0,1]$ by \eqref{eq:adapted}. 

Conversely, assume that $\mu$ is a martingale measure such that under $\mu$ it holds $Z_1^{r}= \delta_{Z^{r-1}}$ for every $r \in \N$ and ${e_t}_\# Z^1_t = \delta_{Z^0_t}$ for all $t \in [0,1]$. By Proposition~\ref{prop:canonRep} there is an $\fp X \in \FR(D(S))$ such that $\mu = \law(\pp^\infty(\fp X))$ and $\pp^\infty(\fp X) = Z$ a.s. As we have ${e_t}_\# Z^1_t = \delta_{Z^0_t}$ 
$\mu$-a.s.\ for all $t \in [0,1]$, we have $\law(X_t |\F_t^{\fp X}) = \delta_{X_t}$ a.s. for all $t$, so $\fp X$ is adapted, cf.\ \eqref{eq:adapted}. Hence, $\fp X \in \FP(S)$. 
\end{proof}

We close this section with a proof of the version of the fixed point result that is stated in the introduction.
\begin{proof}[Proof of Theorem~\ref{thmm:FixedPointIntro}]
In view of the proof of Theorem~\ref{thm:fixedpoint} it suffices to observe that the adaptedness condition \eqref{eq:adapted} is linear in $\law(\pp^\infty(\fp X))$. Moreover, the assumption in Theorem~\ref{thm:fixedpoint} that $\{ \law(X) : \fp X \in f(\uFP) \}$ is tight simplifies to $f(\uFP)$ being relatively compact in the setting of the introduction because the space $S=D([0,1],\R^d)$ with Skorohod's $J_1$-distance is Polish. 
\end{proof}

\subsection{Discrete time}
The aim of this section is to reconcile the notions of filtered random variable and filtered process in discrete time. All results in this section can be proven analogously to the continuous-time case (or obtained as a corollary by considering piecewise constant processes), so we omit the proofs. 

Throughout this section $N \in \N$ always denotes the number of time steps. Recall that an $S$-valued filtered process is a 5-tupel $ \fp X = (\Omega,\F,\P,(\F_t)_{t=1}^N, (X_t)_{t=1}^N)$, where $(X_t)_{t=1}^N$ is an $S$-valued process (i.e.\ $X_t$ takes values in $S$ for all $t$) that is adapted to $(\F_t)_{t=1}^N$; whereas an $S$-valued filtered random variable is a 5-tupel $\fp X = (\Omega,\F,\P,(\F_t)_{t=1}^N, X)$, where $X$ is an $S$-valued $\F_N$-measurable random variable.

Clearly, every $S$-valued filtered process is also an $S^N$-valued filtered random variable, i.e.\ $\FP_N(S) \subset \FR_N(S^N)$. If $N>1$, this inclusion is strict: For $\fp X \in \FR(S^N)$ we only require that $X$ is $\F_N$-measurable but no further adaptedness conditions.
 
Conversely, every $S$-valued random variable with filtration can be seen as an $S$-valued filtered process. To that end, fix some $s_0 \in S$ and set $\iota(s) := (s_0, \dots, s_0,s)$. Consider the mapping
\begin{align}\label{eq:embed1}
\boldsymbol{\iota} : \FR_N(S) \to \FP_N(S) : (\Omega,\F,\P,(\F_t)_{t=1}^N,X) \mapsto  (\Omega,\F,\P,(\F_t)_{t=1}^N,\iota(X)).
\end{align}
Note that $\iota(X)_t$ is  constant for $t<N$ and hence $\F_t$-measurable, so $\iota( X)$ is indeed adapted to $(\F_t)_{t=1}^N$. In \cite{BaBePa21} the space of filtered processes $\uFP_N(S)$ was defined as the factor space $\FP_N(S)$ modulo Hoover--Keisler equivalence and equipped  with the Hoover--Keisler topologies $\HK_r$. One can perform the very same constructions for $\FR_N(S)$ to define the factor space of $\uFR_N(S) := \FR_N(S)/_{\approx_\infty}$ and Hoover--Keisler toplogies $\HK_r$ on it. 
\begin{proposition}\label{prop:FPvsRF1}
The following statements hold true:
\begin{enumerate}[label=(\alph*)]
	\item $\uFP(S)$ is a closed subset of $\uFR_N(S^N)$ w.r.t.\ the topology $\HK_r$ for all $r \in \N \cup \{ \infty \}$.
	\item For every $r \in \N_0 \cup \{ \infty \}$, the map
	\begin{align}\label{eq:embed2}
		\boldsymbol{\iota} : (\uFR_N(S), \HK_r) \to (\uFP_N(S), \HK_r) : (\Omega,\F,\P,(\F_t)_{t=1}^N,X) \mapsto  (\Omega,\F,\P,(\F_t)_{t=1}^N,\iota(X))
	\end{align}
is a topological embedding with closed range.
\end{enumerate}
\end{proposition}

Using this result, one can translate topological statements about $\uFR_N$ into statements about $\uFP_N$ and vice versa. In particular, we have:

\begin{corollary}
A set $\A \subset \uFR_N(S)$ is relatively compact w.r.t. $\HK_r$ if and only if $\{\law(X) : \fp X \in \A\}$ is relatively compact in $\prob(S)$.
\end{corollary}
\begin{proof}
This is an immediate consequence of Proposition~\ref{prop:FPvsRF1} and \cite[Theorem~1.7]{BaBePa21}.
\end{proof}

\begin{corollary}
Let $\fp X^n, \fp X \in \uFP(S)$. Let $T \subset \cont(\fp X)$ and denote the cardinality of $T$ by $N$. If $\fp X^n \to \fp X$ in $(\uFP(S),\HK_r)$, then  $\boldsymbol{\iota}(D^T(\fp X^n)) \to \boldsymbol{\iota}(D^T(\fp X))$ in $(\uFP_N(D(S)),\HK_r)$. 
\end{corollary}
\begin{proof}
This is an immediate consequence of Proposition~\ref{prop:FPvsRF1} and Theorem~\ref{thm:discr}.
\end{proof}

\appendix
\section{Omitted proofs}\label{app:proofs}

\begin{proof}[Proof of Proposition~\ref{prop:lusin_facts}]
	(a) For the direct implication let $(S',\T)$ be a Polish space and $\iota : S \to S'$ a topological embedding such that $\iota(S) \subset S'$ is Borel. By \cite[Theorem 13.1]{Ke95} there is a stronger Polish topology $\T'$ on $S'$ such that $\iota(S)$ is clopen in $(S',\T')$. Then  $\iota^{-1} : (\iota(S),\T'|_{\iota(S)}) \to (S,\T|_S)$ is a continuous bijection.
	
	Conversely, let $(S,\T)$ be a metrizable space such that there is a stronger Polish topology $\T'$ on $S$. As $(S,\T')$ is separable, $(S,\T)$ is separable as well, so there is an embedding $\iota : (S,\T') \to [0,1]^\N$ (see \cite[Theorem 4.14]{Ke95}). Then $\iota : (S,\T) \to [0,1]^\N$ is a continuous injection from a Polish space to Polish space, so it maps Borel sets to Borel sets (cf.\ \cite[Theorem 15.1]{Ke95}). In particular, $\iota(S)$ is a Borel subset of $[0,1]^\N$.
	
	(b) is an easy modification of (a).
\end{proof}
\begin{proof}[Proof of Lemma~\ref{lem:LusinProdComp}]
(a)	The collection of $f \in C_b(S)$ that depend on only on finitely many coordinates is convergence determining for the product topology on $S$ and closed under multiplication. Hence, Lemma~\ref{lem:P(S)_ptsep_convdet} yields that testing against these functions is convergence determining on $\prob(S)$.   
	
(b) Note that we have to be a bit careful because compactness of a set of measures does not imply tightness in Lusin spaces (see Section~\ref{sec:ProbaLusin}). 	Let $(\mu^n)_n$ be a sequence in $\A$. As ${\pr_1}_\#(\A)$ is relatively compact, there is a sub sequence $(\mu^{n_\ell})_\ell$  and some $\nu^1 \in \prob(S_1)$ such that ${\pr_1}_\#(\mu^{n_\ell}) \to \nu_1$. By inductively in $j$ choosing further subsequences and then passing to a diagonal sequence, we find a subsequence  $(\mu^{n_k})_k$ and measures $\nu^j \in \prob(S^j)$ such that for all $j \in \N$ we have ${ \pr_j }_\#(\mu^{n_k}) \to \nu^j$. 
	
As convergent sequences are tight, see Theorem~\ref{thm:prohorov_lusin}(b), for every $j \in \N$ there is a compact set $K_j \subset S_j$ such that $\prob( \pr_j )(\mu^{n_k})(K_j^c) \le \epsilon 2^{-j}$ for all $k \in \N$. The set $K := \prod_{j \in \N} K_j$ is compact by Tychonoff's theorem and we have for all $k \in \N$
	$$
	\mu^{n_k}(K^c) \le \sum_{j \in \N} {\pr_j}_\#\mu^{n_k}(K_j^c) \le \sum_{j \in \N} \epsilon 2^{-j} = \epsilon,
	$$
	so the sequence $(\mu^{n_k})_k$ is tight. By Prohorohov's Theorem~\ref{thm:prohorov_lusin}(a), a further subsequence converges to some $\mu \in \prob(S)$. Hence, $\A$ is relatively compact.
	
(b) This is easy consequence of (a). 	
\end{proof}

\begin{proof}[Proof of Lemma~\ref{lem:convinprob_monotone_pw}]
	As $f$ is increasing, $\cont(f)$ is co-countable and $\cont(f) \cup \{1 \}$ is $\lambda$-full.
	Therefore, it remains to show that convergence in probability implies pointwise convergence for all $t \in \cont(f) \cup \{ 1 \}$. Since $\lambda(\{1\})>0$, we have $f_n(1) \to f(1)$, thus,
	\[
	f_n(t) \le f_n(1) \le \sup_k f_k(1) < \infty,
	\]
	for all $n \in \N$ and $t \in [0,1]$.	
	Assume that there is some $t_0 \in (0,1)$ such that $(f_n(t_0))_{n \in \N}$ is not bounded from below. Then there is a subsequence $f_{n_k}$ such that $f_{n_k}(t_0) \le -k$ and hence $f_{n_k}(t) \le -k$ for all $t \in [0,t_0]$, which is a contradiction to $f_n \to f$ in measure.
	
	Fix some $t_0 \in \cont(f)$. 
	By the above considerations, $(f_n(t_0))_{n \in \N}$ is a bounded real sequence. Therefore, it suffices to show that any convergent subsequence converges to $f(t_0)$. 
	Assume for the sake of contradiction that there is a subsequence $f_{n_k}(t_0) \to a > f(t_0)$. The case $f_{n_k}(t_0) \to a < f(t_0)$ can be treated similarly.
	
	Let $\epsilon := (a-f(t_0))/3$. Then $f_{n_k}(t_0) \ge f(t_0) +2\epsilon$ for sufficiently large $k$. 
	Since $f$ is continuous in $t_0$, there is some $\delta>0$ such that $t_0+\delta \le 1$ and $f(t) \le f(t_0) + \epsilon$ for all $t \in [t_0,t_0+\delta]$. By the monotonicity of $f_{n_k}$ this implies $f_{n_k}(t) - f(t) \ge f_{n_k}(t_0) - f(t_0) - \epsilon \ge \epsilon$ for $t \in [t_0,t_0+\delta]$. 
	We find 
	$
	\lambda( |f_{n_k} - f| \ge \epsilon ) \ge \lambda([t_0,t_0+\delta] ) >0,
	$
	for sufficiently large $k$ contradicting $f_n \to f$ in measure.
\end{proof}

\begin{lemma}\label{lem:ContBigApp}
	Let $S$ be a separable metric space and $f : [0,1] \to S$ such that for all $t \in [0,1]$ the right-limit $\lim_{s \searrow t} f(s)$ exists. Then $f$ is continuous except for at most countable many points. 
\end{lemma}
\begin{proof}
Assume w.l.o.g.\ that $S \subset [0,1]^\N$, cf.\ Section \ref{sec:ProbaLusin}. As a function $f : [0,1] \to [0,1]^\N$ is continuous if and only if every component is continuous, it suffices to show the claim for $f :[0,1] \to [0,1]$. We define the oscillation of $f$ at $t$ by
\[
\textup{osc}_f(t) := \limsup_{\delta \to 0} \sup_{t_1,t_2 \in (t-\delta,t+\delta) \cap [0,1]}  |f(t_1)-f(t_2)|.
\]
Observe that $f$ is continuous at $t$ if and only if $\textup{osc}_f(t) =0$. In order to show that $f$ is continuous except for countable many points, it suffices to show that 
$$
A_n := \{ t \in [0,1] : \textup{osc}_f(t) \ge 1/n  \}
$$
is countable for all $n \in \N$. We observe that 
\begin{align}\label{eq:prf:contBig}
\forall t \in [0,1]\, \exists \epsilon >0 : A_n \cap (t,t+\epsilon) = \emptyset.
\end{align}
Otherwise, there would be a sequence $(t_k)_k$ strictly decreasing to $t$ such that $|f(t_k)-f(t_{k+1})| \ge 1/(2n)$ for all $k \in \N$, which contradicts the existence of the right limit in $t$. To see that this already implies that $A_n$ is countable, consider the set 
$$
B_n = \{ t \in [0,1] : A_n \cap [0,t]  \text{ is  countable} \}.
$$
Clearly, $ 0 \in B_n$.  If $s \le s'$ and $s' \in B_n$ then $s \in B_n$. So $B_n$ is of the form $[0,s)$ or $[0,s]$ for some $s \in [0,1]$. Indeed, if $A_n \cap [0,s)$ is countable, then $A_n \cap [0,s]$ is countable as well, so $B_n = [0,s]$ for some $s \in [0,1]$. Assume for the sake of contradiction that $s<1$. By \eqref{eq:prf:contBig} there is $\epsilon >0$ such that $A_n\cap (s,s+\epsilon) =\emptyset$, so $A_n \cap [0,s+\epsilon]$ is countable as well, i.e.\ $s +\epsilon \in B_n$.
\end{proof}
\begin{lemma}
	\label{lem:Lusin2Polish}
	Let $S_1$ be separable metrizable and $S_2$ be Polish. If there exists a continuous map $f : S_1 \to S_2$ such that  the preimages of compact sets are compact, then $S_1$ is Polish.
\end{lemma}

\begin{proof}
	Let $d_1$ and $d_2$ denote compatible metrics with the topologies on $S_1$ and $S_2$, respectively, and let $d_2$ be complete. We define a metric on $S_1$ by 
	\[
	d(s_1,s_1') := d_1(s_1,s_1') + d_2(f(s_1),f(s_1')).  
	\]
	As $f$ is continuous from $(S_1,d_1)$ to $(S_2,d_2)$, the metrics $d$ and $d_1$ induce the same topology on $S_1$. 
	It remains to show that $d$ is a complete metric. To that end,  let $(s_1^k)_{k \in \N}$ be a $d$-Cauchy sequence. The definition of $d$ implies that $(f(s_1^k))_{k \in \N}$ is  $d_2$-Cauchy as well. As $d_2$ is complete, $(f(s_1^k))_{k \in \N}$ is convergent in $S_2$ and hence relatively compact. Since $f$-preimages of compact sets are compact, $(s_1^k)_{k \in \N}$ is relatively compact in $S_1$, i.e.\ it has a limit point in $S_1$. As $(s_1^k)_{k \in \N}$ is $d$-Cauchy, it can have at most one limit point, so it converges.
\end{proof}

\bibliographystyle{abbrv}
\bibliography{joint_biblio}

\end{document}